\documentclass[a4paper,11pt,reqno]{amsart}
\usepackage[colorlinks=true]{hyperref}
\usepackage{amssymb}
\usepackage{mathtools}
\usepackage{comment}

\usepackage{graphics}

\usepackage[utf8]{inputenc}

\usepackage[textsize=tiny]{todonotes}
\usepackage[all]{xy}

\theoremstyle{plain}
\newtheorem{lemma}{Lemma}[section]
\newtheorem{theorem}[lemma]{Theorem}
\newtheorem{proposition}[lemma]{Proposition}
\newtheorem{prop}[lemma]{{Proposition}}
\newtheorem{corollary}[lemma]{Corollary}
\newtheorem{convention}[lemma]{Convention}

\newtheorem*{convention*}{Convention}

\theoremstyle{definition}
\newtheorem{definition}[lemma]{Definition}
\newtheorem{deff}[lemma]{{Definition}}
\newtheorem{example}[lemma]{Example}
\newtheorem{remark}[lemma]{Remark}
\newtheorem*{definition*}{Definition}

\theoremstyle{remark}

\newcommand{\E}{\mathcal{E}}
\newcommand{\lb}{\left[ \cdot\,,\cdot\right] }
\usepackage[left=2cm, right=3cm, top=3cm, bottom=1.7cm]{geometry}
\newcommand{\dd}{\mathrm{d}}
\newcommand{\dif}{\text{d}}
\newcommand{\X}{\mathfrak {X} }
\title{Lie-Rinehart algebras $\simeq $ Acyclic Lie $\infty $-algebroids}

\let\oldtocsection=\tocsection

\let\oldtocsubsection=\tocsubsection

\let\oldtocsubsubsection=\tocsubsubsection

\renewcommand{\tocsection}[2]{\hspace{0em}\oldtocsection{#1}{#2}}
\renewcommand{\tocsubsection}[2]{\hspace{1em}\oldtocsubsection{#1}{#2}}
\renewcommand{\tocsubsubsection}[2]{\hspace{2em}\oldtocsubsubsection{#1}{#2}}

\author{Camille Laurent-Gengoux}

\author{Ruben Louis}
\thanks{Both authors acknowledge support of the CNRS Project Miti 80Prime \emph{Granum}.}

\address{Institut \'Elie Cartan de Lorraine \\ 
	UMR 7502 du CNRS \\ 
	Université de Lorraine \\
	Metz, France}

\date{June 2021}

\begin{document}

\maketitle

\begin{abstract}
We show that there is an equivalence of categories between Lie-Rinehart algebras over a commutative algebra $\mathcal O $ and homotopy equivalence classes of negatively graded Lie $\infty $-algebroids over their resolutions (=acyclic Lie $\infty$-algebroids). This extends to a purely algebraic setting the construction of the universal $Q$-manifold of a locally real analytic singular foliation of \cite{LLS,LavauSylvain}. 
In particular, it makes sense for the universal Lie $\infty$-algebroid of every singular foliation, without any additional assumption, and for Androulidakis-Zambon singular Lie algebroids. Also, to any ideal $\mathcal I \subset \mathcal O $ preserved by the anchor map of a Lie-Rinehart algebra $\mathcal A $, we associate a homotopy equivalence class of negatively graded Lie $\infty $-algebroids over complexes computing ${\mathrm{Tor}}_{\mathcal O}(\mathcal A, \mathcal O/\mathcal I) $.
Several explicit examples are given.
\end{abstract}

\tableofcontents

\section*{Introduction}

The recent surge of studies about Lie $\infty $-algebras  or Lie $\infty $-groups, their morphisms and their -oids equivalent (i.e. Lie $\infty $-algebroids \cite{Campos,Voronov,Voronov2} and \textquotedblleft  higher groupoids\textquotedblright\,\cite{Severa}) is usually justified by their use in various fields of theoretical physics and  mathematics. Lie $\infty $-algebras or -oids appear often where, at first look, they do not seem to be part of the story, but end up to be  needed to answer natural questions, in particular questions where no higher-structure concept seems a priori involved.
Among examples of such a situation, let us cite deformation quantization of Poisson manifolds \cite{Kontsevich} and many recent developments of BV operator theory, e.g. \cite{CamposBV}, deformations of coisotropic submanifolds \cite{CattaneoFelder},  integration problems of Lie algebroids by stacky-groupoids \cite{Zhu}, complex submanifolds and Atiyah classes  \cite{Kapranov,ChenStienonXu,LSX}.  The list could continue.

\noindent
For instance, in \cite{LLS}-\cite{LavauSylvain}, it is proven that \textquotedblleft behind\textquotedblright\,most singular foliations $ \mathcal F$ there is a natural homotopy class of Lie $\infty $-algebroids, called \emph{universal Lie $\infty $-algebroid of $\mathcal F $}, and that the latter answers natural basic questions about the existence of \textquotedblleft good\textquotedblright\,generators and relations for a singular foliation.  The present article is mainly an algebraization of \cite{LLS}, algebraization that allows to enlarge widely the classes of examples. More precisely, our Theorems \ref{thm:existence} and \ref{th:universal} are similar to the main theorems Theorem 2.8. and Theorem 2.9 in \cite{LLS}:
\begin{enumerate}
    \item Theorem \ref{thm:existence} equips  any free $\mathcal O $-resolution of a Lie-Rinehart algebra $\mathcal A $ with a Lie $\infty $-algebroid structure (Theorem 2.8. in \cite{LLS} was a statement for geometric resolutions of locally real analytic singular foliation on an open subset with compact closure). This is a sort of homotopy transfer theorem, except that no existing homotopy transfer theorem applies in the context of \underline{generic} $\mathcal O $-modules (for instance, \cite{Campos} deals only with projective $\mathcal O $-modules).
    { The difficulty is that we cannot apply the explicit transfer formulas that appear in the homological perturbation lemma because there is in general no $\mathcal O$-linear section of $\mathcal A $ to its projective resolutions.}
    \item Theorem \ref{th:universal} states that any Lie $\infty $-algebroid structure that terminates in $\mathcal A $ comes equipped with a unique up to homotopy Lie $\infty $-algebroid morphism to any structure as in the first item (Theorem 2.8. in \cite{LLS} was a similar statement for Lie $\infty $-algebroids whose anchor takes values in a given singular foliation).
\end{enumerate}
As in \cite{LLS}, an immediate corollary of the result is that any two Lie $\infty $-algebroids as in the first item are homotopy equivalent, defining therefore a class canonically associated to the Lie-Rinehart algebra, that deserve in view of the second item to be called \textquotedblleft universal\textquotedblright.

However:
\begin{enumerate}
 \item  While \cite{LLS} dealt with Lie $\infty $-algebroids over projective resolutions of finite length and finite dimension, we work here with  Lie $\infty $-algebroids over any \emph{free} resolution -even those of infinite length and of infinite dimension in every degree.
    \item In particular, since we will work in a context where taking twice the dual does not bring back the initial space, we can not work with $Q$-manifolds (those being the \textquotedblleft dual\textquotedblright\,of Lie $\infty $-algebroids): it is much complicated to deal with morphisms and homotopies.  
\end{enumerate}
By doing so, several limitations of \cite{LLS} are overcome. While \cite{LLS} only applied to singular foliations which were algebraic or locally real analytic on a relatively compact open subset, the present article associates a natural homotopy class of Lie $ \infty$-algebroids to any Lie-Rinehart algebra, and in particular
\begin{enumerate}
    \item[a)]  to any singular foliation on a smooth manifold, (finitely generated or not). This construction still works with singular foliations in the sense of Stefan-Sussmann for instance,
    \item[b)] to any affine variety, to which we associate its Lie-Rinehart algebra of vector fields),
     and more generally to derivations of any commutative algebra,
    \item[c)] to singular Lie algebroids in the sense of Androulidakis and Zambon  \cite{ZambonMarco2},
    \item[d)] to unexpected various contexts, e.g. Poisson vector fields of a Poisson manifold, seen as a Lie-Rinehart algebra over Casimir functions, or symmetries of a singular foliation, seen as a Lie-Rinehart algebra over functions constant on the leaves. 
\end{enumerate}
These Lie $\infty $-algebroids are constructed on $\mathcal O$-free resolutions of the initial Lie-Rinehart algebra
over $\mathcal O $. They are acyclic universal in some sense, and they also are in particular unique up to homotopy equivalence. Hence the title. 

A similar algebraization of the main results of \cite{LLS}, using semi-models category, appeared recently in Ya\"el Fr\'egier and Rigel A. Juarez-Ojeda \cite{Fregier}.
There are strong similarities between our results and theirs, but morphisms and homotopies in \cite{Fregier} do not match ours. It is highly possible, however, that  Theorem \ref{thm:existence} could be recovered using their results. { Luca Vitagliano \cite{VitaglianoLuca} also constructed Lie $\infty$-algebra  structures out of regular foliations, which are of course a particular case of Lie-Rinehart algebra. These constructions do not have the same purposes. For regular foliations, our Lie $\infty$-algebroid structure is trivial in the sense that it is a Lie algebroid, while his structures become trivial when a good transverse submanifold exists. Lars Kjeseth \cite{LARSKJESETH,LARSKJESETHBRST} also has a notion of resolutions of Lie-Rinehart algebras. But his construction is more in the spirit of Koszul-Tate resolution: Definition 1. in \cite{LARSKJESETH} defines Lie-Rinehart algebras resolutions as resolutions of their Chevalley-Eilenberg coalgebra, not of the Lie-Rinehart algebra itself as a module. It answers a different category of questions, related to BRST and the search of cohomological model for Lie-Rinehart algebra cohomology. For instance, the construction in \cite{LARSKJESETHBRST} for an affine variety and its normal bundle does not match the constructions of Section 3.3 and, in our opinion, are of independent interest.}

Our construction admits an important consequence. For any ideal $\mathcal I \subset \mathcal O$, and any $\mathcal O $-free resolution $\mathcal E_\bullet$ of a Lie-Rinehart algebra $\mathcal A $ over $\mathcal O $, the tensor product $ \mathcal E_\bullet \otimes_\mathcal O \mathcal O/\mathcal I $ computes ${\mathrm{Tor}}_{\mathcal O}(\mathcal A, \mathcal O / \mathcal I ) $. It is easy to check that if $\mathcal I $ is a Lie-Rinehart ideal (i.e. if $ \rho (\mathcal A)[\mathcal I] \subset \mathcal I$), then the universal Lie $\infty $-algebroid structure that we constructed goes to the quotient to $ \mathcal E_\bullet \otimes_\mathcal O \mathcal O/\mathcal I $ (which is a complex computing ${\mathrm{Tor}}_{\mathcal O}(\mathcal A, \mathcal O / \mathcal I )$).
Also, two different universal Lie $\infty$-algebroid structures on two different resolutions being homotopy equivalent, they lead to homotopy equivalent Lie $\infty $-algebroid structures on two complexes computing ${\mathrm{Tor}}_{\mathcal O}(\mathcal A, \mathcal O / \mathcal I ) $. 
To a Lie-Rinehart ideal $\mathcal I $ is therefore associated a homotopy equivalence class of Lie $\infty $-algebroids on such complexes.

This article is organized as follows:
In Section \ref{sec-prerequisite}, we fix notations and review definitions, examples, and give main properties of Lie-Rinehart algebras. 
Afterwards, we present the concept of Lie $\infty$-algebroids, their morphisms, and homotopies of those.
 In Section \ref{sec:main}, we state and prove the main results of this paper, i.e. the equivalence of categories between Lie-Rinehart algebras and homotopy classes of free acyclic Lie $\infty$-algebroids, which justifies the name \emph{universal Lie $\infty$-algebroid} of a Lie-Rinehart algebra.  Section \ref{sec:examples} is devoted to a precise description of the universal Lie $\infty$-algebroids of several Lie-Rinehart algebras.
 The complexity reached by the higher brackets in these examples should convince the reader that it is not a trivial structure, 
 even for relatively simple Lie-Rinehart algebras. 

\section*{Acknowledgements}
{

We would like to thank the CNRS MITI 80Prime project GRANUM, and the Institut Henri Poincaré for hosting us in november 2021. We thank the referee for pointing relevant references and a careful reading.  We acknowledge discussion with S. Lavau at early stages of the project. We would like to thank C. Ospel, P. Vanhaecke and V. Salnikov for giving the possibility to present our results at the \textquotedblleft Rencontre Poisson à La Rochelle, 21-22 October 2021\textquotedblright. Last, R. Louis would like to express sincere gratitude to Université d'État d'Haïti and more precisely the department of mathematics of \'Ecole Normale Supérieure (ENS), for giving a golden opportunity to meet mathematics. He also would like to acknowledge the full financial support for this Phd from R\'egion Grand Est.

}

\vspace{0.5cm}
\noindent
\begin{convention*}
Throughout this article, $ \mathcal O$ is a commutative unital algebra over a field $ \mathbb K$ of characteristic zero, and ${\mathrm{Der}}(\mathcal O) $ stands for its Lie algebra of $\mathbb{K}$-linear derivations. Also, $ \delta [f]$ stands for the derivation $\delta \in  {\mathrm{Der}}(\mathcal O) $ applied to $f\in \mathcal O$.
\end{convention*}

\noindent
An other important convention is that for $\mathcal O $ an algebra over $\mathbb K $, and $\E $ an $\mathcal O$-module, we will denote by $S_\mathbb K \mathcal E $ the symmetric powers over the domain $\mathbb K $ and $\bigodot \E$ its symmetric powers over $\mathcal O$. 

\section{Lie-Rinehart algebras and Lie $\infty$-algebroids}

Except for Remark \ref{loc:res}, this section is essentially a review of the literature on the subject, see, e.g. \cite{MR1625610,MR2075590}. 

\label{sec-prerequisite}

\subsection{Lie-Rinehart algebras and their morphisms}
\label{LR-Mor}
\subsubsection{Definition of Lie-Rinehart algebras}

\begin{definition}
A \emph{Lie-Rinehart algebra} over $ \mathcal O$ is a a triple $(\mathcal A , [\cdot, \cdot]_\mathcal A , \rho_{\mathcal A})$ with $ \mathcal A$ an $ \mathcal O$-module, $[\cdot, \cdot]_\mathcal A  $ a Lie $\mathbb{K}$-algebra bracket on  $\mathcal A $, and $ \rho_\mathcal A \colon \mathcal A \longrightarrow  {\mathrm{Der}}(\mathcal O)$ a $ \mathcal O$-linear Lie algebra morphism called \emph{anchor map}, satisfying the the so-called \emph{Leibniz identity}:
 $$   [  a,  f b ]_\mathcal A  = \rho_\mathcal A (a ) [f] \, b + f [a,b]_\mathcal A  \hbox{ for all $ a,b \in \mathcal A, f \in \mathcal O$}. $$
\end{definition}

\noindent
A Lie-Rinehart algebra is said to be a \emph{Lie algebroid} if $ \mathcal A$ is a projective $ \mathcal O$-module.
\begin{definition}
Given Lie-Rinehart algebras
$(\mathcal A , [\cdot, \cdot]_\mathcal A , \rho_{\mathcal A})$ and $(\mathcal A' , [\cdot, \cdot]_\mathcal A' , \rho_{\mathcal A'})$ over algebras $\mathcal O $  and $\mathcal O' $ respectively.
A \emph{Lie-Rinehart algebra morphism over an algebra morphism $\eta\colon\mathcal O\longrightarrow\mathcal{O}'$} is a Lie algebra morphism $\Phi\colon\mathcal A\longrightarrow\mathcal{A}'$  such that for every $a\in\mathcal A$ and $f\in\mathcal O$:\begin{enumerate}
    \item $\Phi(fa)=\eta(f)\Phi(a)$
    \item $\eta(\rho_\mathcal A(a)[f])=\rho_{\mathcal A'}(\Phi(a)[\eta(f)]$.
\end{enumerate}When $\mathcal O =\mathcal O'$ and $\eta=\text{id}$, we say that $\Phi$ is a \emph{Lie-Rinehart algebra morphism} \emph{over $\mathcal{O}$}.
\end{definition}
    
\begin{remark}\label{loc:res} Let us recall some basic constructions for Lie-Rinehart algebras.
\begin{enumerate}
    \item \emph{Restriction.}
Consider a Lie-Rinehart algebra $(\mathcal A , [\cdot, \cdot]_\mathcal A , \rho_{\mathcal A})$  over $\mathcal O $.
For every \emph{Lie-Rinehart ideal} $\mathcal I \subset \mathcal O$, i.e. any ideal
such that
$$ \rho_\mathcal A (\mathcal A) [\mathcal I] \subset \mathcal I$$
the quotient space $\mathcal A / \mathcal I \mathcal A $ inherits a natural Lie-Rinehart algebra structure over $\mathcal O / \mathcal I$. We call this Lie-Rinehart algebra the \emph{restriction w.r.t the Lie-Rinehart ideal $ \mathcal I$}. In the context of affine varieties, when $\mathcal I$ is the ideal of functions vanishing on an affine subvariety $W$, we shall denote $\frac{\mathcal A}{\mathcal{I}\mathcal A}$ by $\mathfrak i_W^* \mathcal A $. 
    \item \emph{Localization.} Assume that $\mathcal O $ is unital. Let $ S\subset\mathcal O$ be a multiplicative  subset with no zero divisor which contains the unit element. The localization module $S^{-1}\mathcal A=\mathcal A\otimes_\mathcal O S^{-1}\mathcal O$ comes equipped with a natural structure of Lie-Rinehart algebra over the localization algebra $S^{-1}\mathcal O$. The localization map $\mathcal A\hookrightarrow S^{-1}\mathcal A$ is a Lie-Rinehart algebra morphism over the localization map $\mathcal O\hookrightarrow S^{-1}\mathcal O$.
    Since localization exists, the notion of sheaf of Lie-Rinehart algebras \cite{VillatoroJoel} over a projective variety, or a scheme, makes sense.
      \item \emph{Algebra extension.} Assume that the algebra $\mathcal O $ is unital and has no zero divisor, and let $\mathbb O $ be its field of fraction. For any subalgebra $\tilde{\mathcal O}$ with $\mathcal O \subset \tilde{\mathcal O} \subset \mathbb O$ such that $\rho(a) $ is for any $a \in \mathcal A$ valued in derivations of $ {\mathbb O}$ that preserves $ \tilde{\mathcal O} $, there is natural Lie-Rinehart algebra structure  over  $\tilde{\mathcal O} $ on the space $ \tilde{\mathcal O}\otimes_{\mathcal O}\mathcal A$. 
\item {\emph{Blow-up at the origin.}} Let us consider a particular case of the previous construction, when $\mathcal O $ is the algebra $\mathbb C [x_1, \dots, x_N] $. If  the anchor map of a Lie-Rinehart algebra $\mathcal A $ over $\mathcal O $ takes values in vector fields on $\mathbb C^N $ vanishing at the origin, then for all $i=1, \dots, N$,   the polynomial algebra ${\mathcal O}_{U_i} $ generated by $ \frac{x_1}{x_i} , \dots,\frac{x_{i-1}}{x_i},$ $  x_i, \frac{x_{i+1}}{x_i},\dots, \frac{x_N}{x_i} $ satisfies the previous condition, and  ${\mathcal O}_{U_i} \otimes_\mathcal O \mathcal A$ comes equipped with a Lie-Rinehart algebra. Geometrically, this operation corresponds to taking the blow-up of $\mathbb C^N $ at the origin, then looking at the $i$-th natural chart $U_i $ on this blow-up: $\mathcal O_{U_i} $ are the polynomial functions on $U_i $. The family $ {\mathcal O}_{U_i} \otimes_{\mathcal O}\mathcal A$ (for $i=1, \dots,N$) is therefore an atlas for a sheaf of Lie-Rinehart algebras (in the sense of \cite{VillatoroJoel}) on the blow-up of $\mathbb C^N $ at the origin, referred to as the \emph{blow-up of $\mathcal A $ at the origin}.
\end{enumerate}
\end{remark}

\subsubsection{Geometric and Algebraic Examples}

Below is an ordered list of examples to which we intend to apply our results: Vector fields vanishing on subsets of a vector space (Example \ref{ex:vanishing}), Lie algebroids (Example \ref{ex:LieAlg}),  cohomology in degree of $-1$ a Lie $\infty $-algebroid (Example \ref{ex:LieInftyInduced}), singular foliations, or non-finitely generated generalizations of those (Example \ref{ex:singFol}) are  Lie-Rinehart algebras over the algebra of functions on a manifold $M $. We also give an example of a Lie-Rinehart algebra over Casimir functions (Example \ref{ex:OverCasimir}).

\begin{example}
For every commutative $\mathbb K $-algebra $ \mathcal O$, the Lie algebra $\mathcal A= {\mathrm Der}({\mathcal O}) $ of derivations of a commutative algebra $ \mathcal O$ is a Lie-Rinehart algebra over $ \mathcal O$, with the identity as an anchor map. Vector fields on a smooth  or Stein manifold or an affine variety are therefore instances of Lie-Rinehart algebras over their respective natural algebras of functions.
\end{example}
\begin{example}\label{ex:vanishing}
Let $(\mathcal A,\lb_\mathcal A,\rho_\mathcal A)$ a Lie-Rinehart algebra over an algebra $\mathcal O$ and $\mathcal I\subset\mathcal O$ be an ideal. The sub-module $\mathcal I\mathcal A$ is a sub-Lie-Rinehart algebra of $\mathcal A$ and its anchor is given by the restriction of $\rho_\mathcal A$ over $\mathcal I \mathcal A\subset\mathcal A$. It follows easily from 
 $$ [ f a, g b]_\mathcal A = fg [a,b]_\mathcal A + f \rho_\mathcal A (a) [g] \, b - g \rho_\mathcal A (b) [f] \, a  \hbox{ for all $a,b \in \mathcal A, f,g \in \mathcal I$}.$$ 
\end{example}
\begin{example} \label{ex:LieAlg}
Let $M$ be a smooth manifold. By Serre-Swan Theorem, Lie algebroids (in the sense of \cite{Mackenzie}) over $M$ are precisely Lie-Rinehart algebras over $C^\infty(M)$
of the form $(\Gamma(A), [\cdot, \cdot], \rho) $
	where $A$ is a vector bundle over $M$ and  $\rho : A\rightarrow T M$ is a vector bundle morphism.
	
\end{example}
\begin{example}
\label{ex:LieInftyInduced}(see section \ref{sec:infty-alg-oid})	For every Lie $\infty$-algebroid $(\E,(\ell_k)_{k\geq 1}, \rho_\mathcal E)$ over an algebra $\mathcal O$, the quotient space $\frac{\E_{-1}}{\ell_1(\E_{-2})}$ comes equipped with a natural Lie-Rinehart algebra over $\mathcal O$, that we call the \emph{basic Lie-Rinehart algebra of $ (\E,(\ell_k)_{k\geq 1}, \rho_\E) $}.
\end{example}
\begin{example} \label{ex:singFol}
There are several manner to define singular foliations on a manifold $M$.
All these definitions have in common to define sub-Lie-Rinehart algebras $\mathcal F $ of the Lie-Rinehart algebra $\X (M) $ of vector fields on $M$ (or $\X_c (M) $, i.e. compactly supported vector fields on $M$).
With this generality, unfortunately, there are no good definition of leaves: as a consequence, several assumptions are generally made on $\mathcal F $, and
singular foliations are usually defined as sub-Lie-Rinehart algebras $\mathcal F $ of $\X (M) $ (or $\X_c (M) $) satisfying one of the conditions below:
\begin{enumerate}
    \item \emph{\textquotedblleft singular foliation admitting leaves\textquotedblright}: there exists a partition of $M$ into submanifolds called leaves such that for all $m\in M$, the image of the evaluation map $\mathcal F \to T_m M  $  is the tangent space of the leaf through $m$ (when $\mathcal F $ coincides with the space of vector fields tangent to all leaves at all points, we shall speak of a \emph{\textquotedblleft Stefan-Sussman singular foliation\textquotedblright})
     \item \emph{\textquotedblleft self-preserving singular foliations\textquotedblright}: the flow of vector fields in $\mathcal F $, whenever defined, preserves $\mathcal F $,
    \item \emph{\textquotedblleft locally finitely generated singular foliations\textquotedblright}: $\mathcal{F} \subset \X_c(M)$ is locally finitely generated over $C^\infty(M)$ and closed under Lie bracket, see e.g. \cite{Cerveau,Dazord,Debord,AZ}),
    \item \emph{\textquotedblleft finitely generated singular foliations\textquotedblright}: when $\mathcal{F} \subset \X(M)$ is finitely generated over $C^\infty(M)$ and closed under Lie bracket.
\end{enumerate}
It is known that Condition $n+1$ above implies Condition $n$, for $n=1,2,3 $. The converse implications are not true in general. See \cite{LavauSylvain} for an overview of the matter.
\end{example}


	\begin{example}
	\emph{Singular subalgebroids} of a Lie algebroid $ (A, [\cdot, \cdot], \rho)$, i.e. submodules of $\Gamma(A)$ stable under Lie bracket, are examples of Lie-Rinehart algebras: their anchors and brackets are the restrictions of the anchors and brackets of $\Gamma(A)$.
	Locally finitely generated ones are studied in \cite{ZambonMarco,ZambonMarco2,ZambonMarco3}.
	In particular, sections of a Lie algebroid valued in the kernel of the anchor map form a Lie-Rinehart algebra ${\mathrm{Ker}}(\rho) $ for which the anchor map is zero.
\end{example}

\begin{example}
	For a singular foliation $\mathcal{F}$ (in any one of the four senses explained Example \ref{ex:singFol}) on a manifold $M$, consider $\mathcal{S}:=\left\lbrace X\in\mathfrak{X}(M)\mid [X,\mathcal{F}]\subseteq\mathcal{F} \right\rbrace $ (i.e. infinitesimal symmetries of $\mathcal F$) and $$\mathcal{C}:=\left\lbrace f\in\mathcal{C}^\infty(M)\mid Y[f]=0, \;\text{for all}\;Y\in\mathcal{F}\right\rbrace $$ (that can be thought of as functions constant along the leaves of $\mathcal F$). The quotient $\frac{\mathcal{S}}{\mathcal{F}}$ is Lie-Rinehart algebra over  $\mathcal{C}$.
\end{example}

\begin{example} \label{ex:OverCasimir}
	Let $(M,\pi)$ be a Poisson manifold. We define $\mathcal{A}:=\text{H}_\pi^1(M)$ to be the first Poisson cohomology of $\pi$ and $\text{H}_\pi^0(M)=\text{Cas}(\pi)$ to be the algebra of Casimir functions. The bracket of vector fields makes $\mathcal A $ a Lie-Rinehart algebra over $\text{Cas}(\pi)$.
\end{example}

\subsection{Lie $\infty$-algebroids and their morphisms}

Lie $\infty$-algebras are well-known  to be coderivations of degree $-1$ squaring to $0$ of the graded symmetric algebra $S(E)$. 
For Lie $\infty $-algebroids, the situation is more involved, because the $2$-ary bracket is not $\mathcal O $-linear. In the finite dimensional case \cite{Voronov}, rather than seeing it as a coderivation of the symmetric algebra, it is usual to see it as a derivation of the symmetric algebra of the dual, i.e. as a $Q$-manifold.
The duality \textquotedblleft finite rank Lie $\infty $-algebroids\textquotedblright\,$\simeq$ \textquotedblleft Q-manifolds\textquotedblright\,is especially efficient to deal with morphisms. In the lines above, we present a co-derivation version of Lie $\infty $-algebroids, which is subtle, but gives a decent description of morphisms and their homotopies as co-algebra morphisms.

\subsubsection{Graded symmetric algebras}

Let us fix the sign conventions and recall the definition of (negatively-graded) Lie $\infty $-algebroids. For $\mathcal{E}$ is a graded-$ \mathcal O$ module, we denote by $\lvert x\rvert \in \mathbb{Z}$ the degree of a homogeneous element $x\in \mathcal{E}$. 

\begin{enumerate}
    \item We denote by $\bigodot^\bullet\E  $ and call \emph{graded symmetric algebra of $\E $ over $\mathcal O$} the quotient of the tensor algebra over $\mathcal O $ of $\E$, i.e.
$$ T^\bullet_\mathcal O \E:= \oplus_{k =1}^\infty
\underbrace{ \E \otimes_\mathcal O \cdots \otimes_\mathcal O \E}_{\hbox{\small{$k$ times}}}
$$
 by the ideal generated by $x\otimes_{\mathcal O} y-(-1)^{\lvert x\rvert\lvert y\rvert}y\otimes_{\mathcal O} x$, with $x,y$ arbitrary homogeneous elements of $\E $. We denote by $\odot$ the product in $\bigodot^\bullet \E$.
\item Similarly, we denote by $ S^\bullet_\mathbb K( \E ) $ and call \emph{graded symmetric algebra} of $\E $ over \emph{the field $\mathbb K$} the quotient of the tensor algebra (over $\mathbb K $) of $\E$, i.e.
$$ T^\bullet_\mathbb K \E:= \oplus_{k =1}^\infty
\underbrace{ \E \otimes_\mathbb{K} \cdots \otimes_\mathbb{K} \E}_{\hbox{\small{$k$ times}}}$$  by the ideal generated by $x\otimes_\mathbb{K}  y-(-1)^{\lvert x\rvert\lvert y\rvert}y\otimes_\mathbb{K}  x$, with $x,y$ arbitrary homogeneous elements of $\E $. We denote by $\otimes_\mathbb{K}$ or $\cdot$ the product in $S^\bullet_\mathbb K (\E)$.

\end{enumerate}

The algebras $\bigodot^\bullet\E$ and $ S^\bullet_\mathbb{K} (\E)$ come equipped with two different \textquotedblleft degrees\textquotedblright\, that must not be confused. 
\begin{enumerate}
    \item  We define the \emph{degree} of $x= x_1 \odot \cdots  \odot x_n\in \bigodot^n\E$ or $x=x_1 \cdot \cdots  \cdot x_n\in S^n_\mathbb{K}(\E)$ by 
$$| x_1 \cdot \cdots  \cdot x_n | =  | x_1 \odot \cdots  \odot x_n | = |x_1|+ \cdots  + |x_n|$$
for any homogeneous $ x_1, \dots, x_n\in\E$. With respect to this degree, $\bigodot^\bullet\E $ and $S^\bullet_\mathbb{K}(\E)$ are graded commutative algebras. 
    \item The \emph{arity} of $ x_1 \odot \cdots  \odot x_n\in \bigodot^n\E$ or $ x_1 \cdot \cdots  \cdot x_n\in S^n_\mathbb{K}(\E)$ is defined to be $n$. 
    We have
$ \bigodot^\bullet \E = \oplus_{k \geq 1} \bigodot^k \E $ and $ S^\bullet_\mathbb{K} (\E) = \oplus_{k \geq 1} S^k_\mathbb{K} (\E) $
where $ \bigodot^k\E$ and $ S^k_\mathbb{K}(\E)$ stand for the $\mathcal O$-module of elements of arity~$k$ {and the $\mathbb{K}$-vector space of elements of arity~$k$, respectively}.
\end{enumerate}
\begin{convention}
For $ \E$ a graded $ \mathcal O$-module,
elements of arity $k$ and degree $d$ in $ \bigodot^\bullet\E$ (resp. $S^k_\mathbb{K}(\E)$) shall be denoted by $\bigodot^k \mathcal E_{|_d}$ (resp. $S^k_\mathbb{K}(\E)_{|_d}$).
\end{convention}
\noindent
For any homogeneous elements $x_1 , \ldots , x_k \in \mathcal{E}$ and $\sigma\in\mathfrak{S}_k$ a permutation of $\{1, \ldots, k\}$, the \emph{Koszul sign} $\epsilon(\sigma; x_1 , \ldots , x_k )$ is defined by:
$$  x_{\sigma(1)} \odot \cdots \odot x_{\sigma(k)}= \epsilon(\sigma; x_1 , \ldots , x_k ) \, x_1 \odot \cdots \odot x_k.$$   We often write $\epsilon(\sigma )$ for $ \epsilon(\sigma; x_1 , \ldots , x_k )$.\\
 
\noindent
 For $i, j \in\mathbb{N}$, a \emph{$(i, j)$-shuffle} is a permutation
$\sigma \in \mathfrak{S}_{i+j}$ such that $\sigma(1) <\ldots < \sigma(i)$ and $\sigma(i + 1) < \ldots < \sigma(i + j)$, and the set of
all $(i, j)$-shuffles is denoted by $\mathfrak{S}_{i,j}$. Moreover, for $\Phi \colon \E \rightarrow \E'$ and
$\Psi \colon \E'' \rightarrow \E'''$  two homogeneous morphisms of $\mathbb{Z}$-graded $\mathcal{O}$-modules, then $\Phi\otimes\Psi : \E \otimes\E''\rightarrow \E' \otimes\E'''$ stands for the following morphism:
\begin{equation*}
    (\Phi\otimes\Psi)(x\otimes y) = (-1)^{\lvert \Psi\rvert\lvert x\rvert}\Phi(x)\otimes\Psi(y),\; \text{for all homogeneous}\; x \in \E , y \in \E''.\end{equation*}

\begin{lemma}
Both $\bigodot^\bullet\E$ and $S^\bullet _\mathbb K (\mathcal E) $ admit natural co-commutative co-unital co-algebra structures with respect to the deconcatenation $\Delta$ defined by:
$${\Delta(x_1\odot\cdots\odot x_n)= 
\sum_{i=1}^{n-1}\epsilon(\sigma)\sum_{\sigma\in\mathfrak{S}_{i,n-i}}x_{\sigma(1)}\odot\cdots \odot x_{\sigma(i)}\otimes x_{\sigma(i+1)}\odot\cdots\odot x_{\sigma(n)}
}$$for every $x_1,\ldots, x_n\in \E$.  
\end{lemma}

\subsubsection{Lie $\infty$-algebroids as co-derivations of graded symmetric algebras}\label{sec:infty-alg-oid}

{
Lie $\infty $-algebroids over manifolds were introduced (explicitly or implicitly) by various authors, e.g. \cite{MR2966944}, \cite{Voronov2}, and \cite{SeveraTitle}. We refer to Giuseppe Bonavolont\`a and Norbert Poncin for a complete overview of the matter \cite{Poncin}. Also, \cite{LARSKJESETH,VitaglianoLuca_R} extend theses notions to the Lie-Rinehart algebra setting. 
}

\begin{definition}\label{def:Linfty}
 A negatively graded Lie $ \infty$-algebroid over $ \mathcal O$ is a collection $\E=(\mathcal E_i)_{i \leq -1} $ of projective $ \mathcal O$-modules, equipped with:
\begin{enumerate}
    \item a collection of $\mathbb K $-linear maps $\ell_i:\bigodot^i \E \longrightarrow \E$ of degree $ +1$ called $i$-ary \emph{brackets}
    \item a $\mathcal O$-linear map $\E_{-1}  \longrightarrow {\mathrm{Der}}(\mathcal O) $ called \emph{anchor map}
\end{enumerate}
satisfying the following axioms~:
\begin{enumerate}\label{def:Jacobi}
    \item[$(i)$] the \emph{higher Jacobi identity}:
    \begin{equation} \sum_{i=1}^{n}\sum_{\sigma\in \mathfrak{S}_{i,n-i}}{{\epsilon(\sigma)}} \, \, \ell_{n-i+1}(\ell_i( x_{\sigma(1)},\ldots,x_{\sigma(i)}),x_{\sigma(i+1)},\ldots,x_{\sigma(n)})=0, \end{equation} for all $n\geq1$ and homogeneous elements $x_1,\ldots,x_n\in\E$;
    \item[$(ii)$] for $ i \neq 2$, the bracket $\ell_i $ is $\mathcal O $-linear, while for $ i=2$,
     $$   \ell_2 ( x,  f y )  = \rho_\mathcal E (x ) [f] \, y + f \ell_2(x,y)  \hbox{ for all $ x,y \in \mathcal E , f \in \mathcal O$ } ,$$
     where, by convention, $  \rho_{\mathcal E}$ is extended by zero on $ \mathcal E_{-i}$ for all $i \geq 2$.
     \item[$(iv)$]  $\rho_\E\circ\ell_1=0$ on $\E_{-2}$.
     \item[$(v)$] $\rho_\E$ is a morphism of brackets, i.e., $\rho_\E(\ell_2(x, y)) = [\rho_\E(x), \rho_\E(y)]$ for all $x,y\in\E_{-1}$.
\end{enumerate}
\end{definition}
\begin{remark}
The third and the fourth axiom are consequences of item $(i)$ and $(ii)$ if $\E_{-1}$ has no zero divisors.
\end{remark}
\begin{convention}
From now on, we will simply say \textquotedblleft Lie $ \infty$-algebroid\textquotedblright\,for \textquotedblleft negatively graded Lie $ \infty$-algebroid\textquotedblright.
\end{convention}

It follows from Definition \ref{def:Linfty} that
  $$   \cdots\stackrel{\ell_1}{\longrightarrow} \mathcal E_{ -3}\stackrel{\ell_1}{\longrightarrow} \mathcal E_{-2}  \stackrel{\ell_1}{\longrightarrow}  \mathcal E_{ -1} $$
  is a complex of projective $\mathcal O$-modules. 
A Lie $\infty $-algebroid is said to be \emph{acyclic} if this complex has no cohomology in degree $\leq -2 $.
 
 There is an  equivalent way to define Lie $ \infty$-algebroids in term of co-derivations. We will use such a definition to deal with  morphisms of Lie $ \infty$-algebroids.

 \begin{remark}
 \cite{Kassel} Recall that a co-derivation $Q_\E$ of the symmetric algebra $S^\bullet_\mathbb{K}(\E)$ is entirely determined by the collection indexed by $k \in \mathbb{N}_0$ of maps called its \emph{$k$-th Taylor coefficients}:
\begin{equation}
\label{eq:Taylor}
    Q_\E^{(k)} \colon S^{k+1}_\mathbb{K}(\E)\stackrel{ Q_\E}{\longrightarrow} S^\bullet_\mathbb{K}(\E)\stackrel{\text{pr}}{\longrightarrow} \E,
\end{equation}
with $\text{pr}$ being the projection onto the term of arity $1$, i.e.
$\text{pr}\colon S^\bullet_\mathbb{K} (\mathcal E) \rightarrow S^{1}_\mathbb K (\mathcal E) \simeq \E$. 
 \end{remark}
 \begin{definition}
A co-algebra morphism $\Phi\colon S^\bullet_\mathbb{K} (\E')\longrightarrow S^\bullet_\mathbb{K}(\E)$ or a co-derivation $Q_\E$ of the symmetric algebra $S^\bullet_\mathbb{K}(\E)$ are said to be of \emph{arity $k\in \mathbb{N}_0$}, if  $\Phi^{(n)}=0$, for $n\neq k$ or $Q_\E^{(n)}=0$, for $n\neq k$.
\end{definition}
 
\subsubsection{Lie $\infty $-algebroids and Richardon-Nijenhuis brackets} 
 
 The space of $\mathbb K $-multilinear maps from $ \E$ to $\E$ admits two gradings, the degree and the arity. Elements of arity $k$ and degree $j$ shall be, by definition, the space
  $$ \text{Hom}_{\mathbb K}^{j} \left( S^{k+1}_\mathbb K  \E \,  , \E \right) := \oplus_{m \in \mathbb Z}\text{Hom}_{\mathbb K} \left( S^{k+1}_\mathbb K  \E \, _{|_{m-j}} , \E_{m}  \right)$$
 The \emph{Richardon-Nihenhuis bracket} {\cite{MR1202431}}-\cite{YKS}
$$ \left[ 
\text{Hom}_{\mathbb K}^i \left(  S^{k+1}_\mathbb K   \E 
, \E  \right) , 
\text{Hom}_{\mathbb K}^j \left(  S^{l+1}_\mathbb K  \E ,\E \right) \right]_{\hbox{\tiny{RN}}} \subset  \text{Hom}_{\mathbb K}^{i+j} \left(  S^{l+k+1}_\mathbb K  \E  , \E\right)$$
is defined on homogeneous elements $A,B\in \text{Hom}_{\mathbb K}^{\bullet} \left( S^{\bullet}_\mathbb K  \E \,  , \E \right) $ by \begin{equation}[A,B]_{\hbox{\tiny{RN}}}=\iota_AB-(-1)^{|A||B|}\iota_BA.\end{equation} Here $\iota_AB$ (also denoted by $A\circ B$) is the interior product defined by
 \begin{equation}( \iota_A B ) (x_1,\ldots,x_{p+q} )=\sum_{\sigma \in \mathfrak{S}_{p,q}} \epsilon( \sigma )A( B ( x_{\sigma ( 1)} ,\ldots , x_{\sigma ( p)} ) , x_{\sigma ( p+1 )} ,\ldots, x_{\sigma ( p+q )} ). \end{equation} The bracket is extended by bilinearity.
It is classical that this bracket is a graded Lie algebra bracket.

Let us relate the bracket with co-derivations. For a given $i \in \mathbb Z $, and a given $A = \sum_{k \geq 0} A^{(k)}$ with $A^{(k)}\in {\mathrm{Hom}}^i_\mathbb K  \left( S^{k+1}_\mathbb K (\E ), \E \right) $, we denote by $\delta_A $ the unique co-derivation with Taylor coefficients $(A^{(k)} )_{k \in\mathbb{N}_0} $. This co-derivation has degree $i$, and the following Lemma is easily checked:

\begin{lemma}\cite{Kassel} \label{lem:RN}
For every $A,B$ of degrees $i,j$ as above, we have
 $$ \delta_A \circ \delta_B - (-1)^{ij}\delta_B  \circ \delta_A = \delta_{[A,B]_{\hbox{\tiny{RN}}}} .$$
\end{lemma}

We can now give an alternative description of Lie $ \infty$-algebroids in terms of co-derivation (extending the usual \cite{Voronov} correspondence between Lie $ \infty$-algebroids and $Q$-manifolds in the finite rank case).
\noindent
For a good understanding of the next Proposition, see notations of Taylor coefficients in Equation \eqref{eq:Taylor}.

 \begin{proposition}\label{co-diff}
  For a collection $\E=(\mathcal E_k)_{k \leq -1} $ of projective $ \mathcal O$-modules, there is a one-to-one correspondence between Lie $\infty$-algebroid structures on $\E$ and pairs made of co-derivations $Q_\E\colon S^\bullet_\mathbb{K} (\mathcal E) \rightarrow S^\bullet_\mathbb{K} (\mathcal E) $ of degree $+1$ which satisfies $Q_\E^2 = 0$, and a $\mathcal O$-linear morphism, $\rho_\E\colon\E_{-1}\rightarrow \emph{Der}(\mathcal O)$ called the \emph{anchor}, such that 
 \begin{enumerate}
     \item for $k\neq 1$ the $k$-th Taylor coefficient {$Q_\E^{(k)}\colon S_\mathbb{K}^{k+1}(\E')\longrightarrow \E$} of $Q_\E$ is $\mathcal O$-multilinear,
     \item for all $x,y \in \mathcal E$ and $ f \in \mathcal O$, we have, $ Q_\mathcal E^{(1)} (x\cdot  fy) = f   Q_\mathcal E^{(1)} (x\cdot y)  + \rho_\E(x)[f] \, y $,\; 
     \item $\rho_\E \circ Q_\mathcal E^{(0)} =0 $ on $\mathcal E_{-2} $,
     \item {$\rho_\E\circ Q^{(1)}_\E(x.y)=[\rho_\E(x),\rho_\E(y)]$, for all $x,y \in \mathcal E_{-1}$}.
 \end{enumerate}
 The correspondence consists in assigning to a Lie $\infty $-algebroid $(\E, (\ell_k)_{k\geq 1}, \rho_\E) $ the co-derivation $Q_\E$ whose  $k$-th Taylor coefficient  is the $k$-ary bracket $\ell_k $ for all $k \in \mathbb{N}_0$.
 \end{proposition}
 \begin{proof}
 The higher Jacobi identity is equivalent to
 $$ \sum_{i=1, \dots,i=n} [\ell_i, \ell_{n+1-i}]_{\hbox{\tiny{RN}}} =0 .$$
For all positive integer $n$. The statement is then an immediate consequence of Lemma \ref{lem:RN}.
\end{proof}
\begin{convention}
 From now, when relevant, we will sometime denote an underlying structure of Lie $\infty$-algebroid $(\E,(\ell_k)_{k\geq 1},\rho_\E)$ on $\E$ by  $(\E,Q_\E)$ instead. 
\end{convention}

 \begin{remark}
 Notice that $Q_\E\colon S^\bullet_\mathbb{K}(\E)\longrightarrow S^\bullet_\mathbb{K}(\E)$ does not induce a co-derivation on $\bigodot^\bullet\E$ unless $\rho_\E=0$.
 \end{remark}

\subsubsection{Universal Lie $\infty$-algebroids of Lie-Rinehart algebras}
  
\begin{definition}
We say that a Lie $\infty $-algebroid $(\E, (\ell_k)_{k\geq 1},\rho_\E)$ \emph{terminates} at a Lie-Rinehart algebra $(\mathcal A, [\cdot, \cdot]_\mathcal A, \rho_\mathcal A)  $ when it is equipped with a $\mathcal O $-linear map $\pi: \mathcal E_{-1} \longrightarrow \mathcal A  $, called \emph{hook},  such that
 $$ \rho_A  \circ \pi = \rho_{\mathcal E}  \hbox{ and }  [\pi(x),\pi(y)]_\mathcal A = \pi (\ell_2 (x,y)),\; \text{for all}\;x,y\in\E_{-1}.$$
\end{definition}
 
\begin{example}
\label{ex:basic}
Every Lie $\infty $-algebroid $(\E, (\ell_k)_{k\geq 1},\rho_\E)$ terminates at the basic Lie-Rinehart algebra $\mathcal E_{-1} / \ell_1 (\mathcal E_{-2} ) $ of Example \ref{ex:LieInftyInduced}, the projection $\pi \colon \E_{-1} \longrightarrow \mathcal E_{-1} / \ell_1 (\mathcal E_{-2} ) $ being the hook.
\end{example}


\begin{definition}
Let  $(\mathcal A, [\cdot, \cdot], \rho_\mathcal A) $ be a Lie-Rinehart algebra. A Lie $\infty $-algebroid $(\E, (\ell_k)_{k\geq 1}, \rho_\E)$ that terminates at $\mathcal A $ through an hook $\pi $ is said to be \emph{universal for $\mathcal A $} if 
    $$   \cdots\stackrel{\ell_1}{\longrightarrow} \mathcal E_{ -2}\stackrel{\ell_1}{\longrightarrow} \mathcal E_{-2}  \stackrel{\ell_1}{\longrightarrow}  \mathcal E_{ -1} \stackrel{\pi}{\longrightarrow}    \mathcal A \longrightarrow 0. $$
 is a projective resolution of $\mathcal A $ in the category of $\mathcal O$-modules.
\end{definition}

In other words, a universal Lie $\infty $-algebroid of a Lie-Rinehart algebra $(\mathcal A, [\cdot, \cdot], \rho) $ is a Lie $\infty $-algebroid built on a projective resolution of $\mathcal A$ as an $\mathcal O $-module, whose  Lie-Rinehart algebra is $(\mathcal A, [\cdot, \cdot], \rho) $.

\subsubsection{Morphisms of Lie  $\infty$-algebroids and their homotopies}
This section extends Section 3.4 of \cite{LLS} to the infinite dimensional setting.
Let $\E$ and $ \E'$ be graded $ \mathcal O$-modules. A co-algebra morphism $\Phi\colon S^\bullet_\mathbb{K}(\E')\longrightarrow S^\bullet_\mathbb{K}(\E)$ is completely determined by the collection indexed by $k \in \mathbb{N}_0 $ of maps called its \emph{$k$-th Taylor coefficients}:
$$\Phi^{(k)} \colon S^{k+1}_\mathbb{K}(\E')\stackrel{ \Phi}{\longrightarrow} S^\bullet_\mathbb{K}(\E)\stackrel{\text{pr}}{\longrightarrow} \E,$$
with $\text{pr}$ being the projection onto the term of arity $1$, i.e.
$\text{pr} \colon S^\bullet_\mathbb{K} (\mathcal E) \rightarrow S^{1}_\mathbb K (\mathcal E)\simeq \E$. 

The following Lemma is straightforward:

\begin{lemma}\label{Tay-C}
Let $ \Phi \colon S^\bullet_\mathbb{K} (\mathcal E') \longrightarrow S^\bullet_\mathbb{K} (\mathcal E)$ be a co-algebra morphism. The following conditions are equivalent:
\begin{enumerate}
    \item[(i)] For every $n \geq 0$, the $n$-th Taylor coefficient $ \Phi^{(n)} \colon S^{n+1}_\mathbb K (\mathcal E' ) \longrightarrow \mathcal E$ of  $ \Phi$ is $\mathcal O$-multilinear
    \item[(ii)] There exists an induced co-algebra morphism $\Phi^\mathcal O\colon \bigodot^\bullet\mathcal E' \longrightarrow \bigodot^\bullet\mathcal E $ making the following diagram commutative~:
    $$ \xymatrix{  \ar[d] S^\bullet_\mathbb{K} (\mathcal E') \ar[r]^{\Phi} &   \ar[d]  S^\bullet_\mathbb{K} (\mathcal E) \\ \bigodot^\bullet\mathcal E'  \ar[r]^{\Phi^\mathcal O} & \bigodot^\bullet\mathcal E}  
    $$
\end{enumerate}
\end{lemma}

\noindent
We say that a co-algebra morphism  $ \Phi \colon S^\bullet_\mathbb{K} (\mathcal E') \longrightarrow S^\bullet_\mathbb{K} (\mathcal E)$ is \emph{$\mathcal O $-multilinear} when one of the equivalent conditions above is satisfied.

Now, let $(\E, \rho_\E, (\ell_k)_{k\geq 1})$ 
and $(\E', \rho_\E', (\ell_k')_{k\geq 1})$ be Lie $\infty$-algebroids. Let $Q_\E$ and $Q_{\E'}$ be their square-zero co-derivations of $S^\bullet_\mathbb{K} (\mathcal E) $ and  $S^\bullet_\mathbb{K} (\mathcal E') $ respectively as in Proposition \ref{co-diff}. Recall from \cite{Stasheff} that Lie $ \infty$-algebra morphisms from $(S^\bullet_\mathbb{K}(\E), Q_\E)$ 
to $(S^\bullet_\mathbb{K}(\E'), Q_{\E'})$ are defined to be co-algebra morphisms $ \Phi\colon S^\bullet_\mathbb{K} (\mathcal E') \longrightarrow S^\bullet_\mathbb{K} (\mathcal E)$ such that
	\begin{equation}\label{def:LM}
	\Phi\circ Q_{\E'}=Q_{\E}\circ\Phi.
	\end{equation}
We will need two additional assumptions to turn a Lie $\infty$-algebr\underline{a} morphism into a Lie $\infty$-algebr\underline{oid} morphism:

\begin{deff}	
\label{def:morph}
	A Lie $\infty$-algebroid morphism from a Lie $\infty$-algebroid $(S^\bullet_\mathbb{K}(\E'), Q_{\E'})$ to a Lie $\infty$-algebroid $(S^\bullet_\mathbb{K}(\E), Q_\E)$, is a  Lie $\infty$-morphism $ \Phi\colon S^\bullet_\mathbb{K} (\mathcal E') \longrightarrow S^\bullet_\mathbb{K} (\mathcal E)$ which is 
	\begin{enumerate}
	    \item  $\mathcal O $-multilinear,
	    \item and satisfies  $\rho_{\E}\circ \Phi^{(0)}=\rho_{\E'}$ on $\E'_{-1}$.
	\end{enumerate}
	Above, $\Phi^{(0)}:(\E',\ell'_1)\longrightarrow(\E,\ell_1)$ is the chain map induced by $ \Phi$ (i.e. the restriction of  $ \Phi\colon S^\bullet_\mathbb{K} (\mathcal E') \longrightarrow S^\bullet_\mathbb{K} (\mathcal E)$ to $\mathcal E'\to\E$).
\end{deff}	

\noindent
When the Lie $\infty$-algebroids $(S^\bullet_\mathbb{K}(\E'), Q_{\E'})$ and $(S^\bullet_\mathbb{K}(\E), Q_\E)$ terminate at a given Lie-Rinehart algebra $(\mathcal{A},[\cdot, \cdot]_\mathcal A , \rho_\mathcal A)$, we define \emph{ morphisms  of Lie $\infty$-algebroids that terminate at $\mathcal A$} as being Lie $\infty$-algebroid morphisms that satisfy  $ \pi \circ \Phi^{(0)} = \pi' $, where $ \pi,\pi'$  are their respective hooks. This property implies the second condition in Definition \ref{def:morph}.

\begin{example}
An $\mathcal O$-linear Lie algebroid morphism $\Phi\colon \mathcal A\longrightarrow \mathcal B$ \cite{Mackenzie} is a Lie $\infty$-algebroid morphism: the corresponding co-algebra morphsism is $a_1\cdot\cdots\cdot a_n \mapsto \Phi(a_1)\cdot\cdots\cdot\Phi(a_n)$,\; for all $a_1,\ldots, a_n \in \mathcal A$.
\end{example}

Let us now define homotopy of Lie $\infty $-algebroid morphisms. 
We start with a technical but important object:

\begin{deff}\label{coder}Let $\Phi:S^\bullet_\mathbb{K} (\E')\mapsto S^\bullet_\mathbb{K} (\E)$ be a graded co-algebra morphism. A $\Phi$-co-derivation of degree $k$ on $S^\bullet(\E')$ is a degree $k$ multilinear map $\mathcal H : S^\bullet_\mathbb{K} (\E')\mapsto S^\bullet_\mathbb{K} (\E)$ which satisfies the (co)Leibniz identity:
\begin{equation} \label{eq:Phicoder} 
	\Delta \circ \mathcal H(v) = (\mathcal H\otimes \Phi)\circ \Delta'(v) + (\Phi \otimes\mathcal H) \circ \Delta'(v)\quad \text{for every}\;\,v\in S^\bullet_\mathbb K (\E').
\end{equation}
\end{deff}

\noindent
For a given co-algebra morphism $\Phi$, $\mathcal H$ is entirely determined by the Taylor coefficient $$ \mathcal H^{(n)} \colon S^{n+1}_\mathbb K (\E') \longrightarrow \E,$$ which are called its \emph{$n$-th Taylor coefficients}:
\begin{equation}\label{TC-expension}
\mathcal H(x_1, \dots,x_n) = \sum_{k=0}^{n-1} \sum_{ \tiny{\begin{array}{c}I \coprod J_1 \coprod \dots \coprod J_k  \\= \{1...n\} \end{array}}} \epsilon(x_I, x_{J_1} , \dots, x_{J_k})  \, \mathcal H^{(|I|)} (x_{I} ) \cdot \Phi^{(|J_1|)}( x_{J_1}) \cdot  \ldots \cdot\Phi^{(|J_k|)} (x_{J_k}) \end{equation}
where for every subset $J  = \{j_1 < ... < j_r\} \subset \{1, \cdots, n\} $, $x_J$ stands for the list $x_{j_1}, \dots, x_{j_r} $ and $|J|=r $ stands for the length of the list.
This formula shows that the Taylor coefficients of $\mathcal H $ are $ \mathcal O$-multilinear if and only if $\mathcal H$ induces a $ {\Phi^\mathcal O}$-co-derivation\footnote{defined as in \eqref{eq:Phicoder} with ${\Phi^\mathcal O} $ instead of $ \Phi$} ${\mathcal H}^\mathcal O \colon \bigodot^\bullet\E'\longrightarrow \bigodot^\bullet\E$.

\begin{proposition}\label{linearity}
Let $ \Phi\colon S^\bullet_\mathbb{K} (\E')\rightarrow S^\bullet_\mathbb{K} (\E)$ be a  Lie $\infty $-algebroid morphism.
For $\mathcal H\colon S^\bullet_\mathbb{K} (\E')\rightarrow S^\bullet_\mathbb{K} (\E)$ a $\mathcal O $-multilinear $ \Phi$-co-derivation of degree $k\leq -1$, 
$\mathcal H \circ Q_{\mathcal E'} -(-1)^k  Q_{\mathcal E} \circ \mathcal H $ is a $\mathcal O $-multilinear $ \Phi$-co-derivation of degree $ k+1$.
\end{proposition}
\begin{proof}
We first check that    $\mathcal H \circ Q_{\mathcal E} - (-1)^k  Q_{\mathcal E'} \circ \mathcal H$ is  a $ \Phi$-co-derivation: \begin{align*}
   \Delta\circ\mathcal H \circ Q_{\E'}&=(\mathcal H\otimes\Phi+\Phi\otimes\mathcal H)\circ\Delta'\circ Q_\E',\;\text{by definition of $\mathcal H$}\\&=(\mathcal H\otimes\Phi+\Phi\otimes\mathcal H)\circ(Q_{\E'}\otimes \text{id}+\text{id}\otimes Q_{\E'})\circ\Delta',\;\text{by definition of $Q_\E'$}\\&=(\mathcal H\circ Q_{\E'}\otimes\Phi+(-1)^k\Phi\circ Q_{\E'}\otimes\mathcal H+\mathcal H\otimes\Phi\circ Q_{\E'}+\Phi\otimes\mathcal H\circ Q_{\E'})\circ\Delta'
\end{align*}
Subtracting a similar equation for $(-1)^k\Delta\circ Q_{\E}\circ\mathcal H$ and using \eqref{def:LM}, one obtains the $\Phi$-co-derivation
property for $\mathcal H \circ Q_{\mathcal E'} - (-1)^k  Q_{\mathcal E'} \circ \mathcal H$.
We now check that $\mathcal H \circ Q_{\mathcal E'} -(-1)^k  Q_{\mathcal E} \circ \mathcal H$ is $ \mathcal O$-multilinear, for which it suffices to check that its Taylor coefficients are $\mathcal{O}$-multilinear by Lemma \ref{Tay-C}. Let $x_1,\ldots,x_n\in\E'$ be homogeneous elements. Assume $x_i\in\E_{-1}'$ 
(if we have more elements of degree $-1$ the same reasoning holds). To verify $\mathcal{O}$-multilinearity it suffices to check that for all $f\in \mathcal O $:
\begin{align}\text{pr}\circ(\mathcal H \circ Q_{\mathcal E'} -(-1)^k  Q_{\mathcal E} \circ \mathcal H) (x_1,\ldots,&x_i, \ldots, f x_j \ldots,x_n)=\\&  f \text{pr}\circ(\mathcal H \circ Q_{\mathcal E'} -(-1)^k  Q_{\mathcal E} \circ \mathcal H) (x_1,\ldots,x_i, \ldots,  x_j \ldots,x_n).\end{align}
Only the terms where the $2$-ary bracket with a degree $-1 $ element on one-side and  $f $ on the other side may forbid $f$ to go in front. There are two such terms:
$$  \epsilon(x_i, x_j, x_{I^{ij}})   \mathcal H^{(n-1)} (\ell_2'(x_i, f x_j), x_{I^{ij}}  ) \hbox{ and } 
-(-1)^k  \epsilon(x_i,  x_{I^{i}})\ell_2(\Phi^{(0)}(x_i), f\mathcal H^{(n-1)} ( x_{I^{i}}  ))
$$ 
where $x_{I^i}$ and $ x_{I^{ij}}$ stand for the list $ x_1, \dots, x_n$ where $x_i$ and $x_i,x_j $ are missing respectively, and $ \mathcal H^{(n-1)}$ is the $(n-1)$-th Taylor coefficient of $\mathcal H$. Since  $\rho_\E\circ\Phi^{(0)}=\rho_{\E'}$, in both terms $\rho_\E (e_i)[f] $ appears, and these two terms add up to zero.
\end{proof}
\begin{remark}
If the degree of $\mathcal H$ is non-negative, then $\mathcal H \circ Q_{\mathcal E'} -(-1)^k  Q_{\mathcal E} \circ \mathcal H $ may not be $\mathcal O$-multilinear any more, since there may exist extra terms where the anchor map appears, e.g. terms of the form $\ell_2(\mathcal H( x_{I^j} ) , f \Phi^{(0)} (x_j) ) $.
\end{remark}
We can now define homotopies between Lie $ \infty$-algebroid morphisms, extending \cite{LLS} from finite dimensional $Q$-manifolds to arbitrary Lie $\infty $-algebroids. 

Let $V$ be a vector space. Unless a topology on $V$ is chosen, the notion of $V$-valued continuous or differentiable or smooth function on an interval  $ I =[a,b] \subset \mathbb R$ does not make sense. However, we can always define the notion of a \emph{piecewise rational function $f \subset I \longrightarrow V$} as follows: we choose a finite increasing sequence $  a=t_0 \leq \dots \leq t_n=b   $ of \emph{gluing points}, and we require that for all $i=0, \dots, n-1$  the restriction $f^i$ of $f$ to $[t_i,t_{i+1}] $ is a finite sum of functions of the form $ g(t) v$ with $v \in V$ and $ g(t)$ a real rational function on $  [t_i,t_{i+1}]$ which has no pole on $[t_i,t_{i+1}]$. If $ f_i$  and $ f_{i+1}$ coincide at the gluing point $ t_{i+1}$, we say that $f$ is \emph{continuous}. When $V$ is a space of linear maps between the vector spaces $S $ and $T$,  we shall say that a $V$-valued map $f_t$ is a \emph{piecewise rational (continuous)} if $f_t(s)$ is a piecewise rational (continuous) $T$-valued function for all $s \in S$. 

Here is an important feature of such functions.

\begin{lemma} \label{lem:primitivesDerivatives}
The derivative of a piecewise rational continuous function is defined at every point which is not a gluing point and is piecewise rational. Conversely, every piecewise rational functions admits a piecewise rational continuous primitive, unique up to a constant. 
\end{lemma}

A family $ (\Phi_t)_{t \in I}$ of co-algebra morphisms can now be defined to be \emph{piecewise rational continuous } if its Taylor coefficients $ \Phi_t^{(n)}$ are piecewise rational continuous for all $n\in\mathbb{N}$. For such a family $(\Phi_t)_{t\in I}$, a family $(H_t)_{t\in I} $ made of $ \Phi_t$-co-derivations is said to be \emph{piecewise rational} if all its Taylor coefficients are.

\begin{remark}
In the above definitions, we do not assume the gluing points of the various Taylor coefficients $ \Phi_t^{(n)}$ or $H^{(n)}_t $ to be the same for all $n\in\mathbb{N}_0$. 
\end{remark}

We now extend Definition 3.53 in \cite{LLS} to the infinite rank case.

\begin{deff}
\label{def:homotopy}
	Let $\Phi$ and $\Psi$ be Lie $\infty$-algebroid morphisms from $(S^\bullet_\mathbb{K}(\E'),Q_{\E'} ) $ to $(S^\bullet_\mathbb{K}(\E), Q_{\E})$. A \emph{homotopy} between $\Phi$ and $\Psi$ is a pair $(\Phi_t , H_t)_{t\in[a,b]}$ consisting of:
	\begin{enumerate}
		\item  a piecewise rational continuous  path $t\mapsto \Phi_t$ valued in Lie $\infty$-algebroid morphisms between $S^\bullet_\mathbb{K}(\E')$ and $S^\bullet_\mathbb{K}(\E)$ satisfying the boundary condition:
		$$\Phi_a = \Phi\quad
		\text{and}\quad
		\Phi_b = \Psi,$$
		\item a piecewise rational  path $t\mapsto H_t$, with $ H_t$ a $\Phi_t$-co-derivations of degree $-1$ from $S^\bullet_\mathbb{K}(\E')$ to $S^\bullet_\mathbb{K}(\E)$, such that the
		following equation:
		\begin{equation}\label{eq-diff}
		\frac{\dd\Phi_t}{\dd t}=Q_{\E}\circ H_t+H_t\circ Q_{\E'}
		\end{equation}
		holds for every $t \in ]a\,, b[$ where it is defined (that is, not a gluing point for the Taylor coefficients). More precisely, for every  $v \in S^{\leq n}_\mathbb{K} (\mathcal E')$, \begin{equation}\label{eq-diffv}
		\frac{\dd\Phi_t}{\dd t}(v)=Q_{\E}\circ H_t(v)+H_t\circ Q_{\E'}(v)
		\end{equation}for all $t$ which is not a gluing point of the Taylor coefficient of $\Phi_t^{(k)},H_t^{(k)}$ for $k=0,\ldots,n$.
	\end{enumerate}
\end{deff}

Definition \ref{def:homotopy} is justified by the following statement:

\begin{prop}
\label{prop:justify}
Let $\Phi $ be a Lie $\infty $-algebroid morphism from $(S^\bullet_\mathbb{K}(\E'),Q_{\E'}) $ to $(S^\bullet_\mathbb{K}(\E), Q_{\E})$. For all $t\in [a\,,b]$, let $ H_t^{(n)}\colon S^{n+1}_\mathbb K (\mathcal E') \longrightarrow ~\mathcal E$ be a family $\mathcal O $-multilinear piecewise rational maps indexed by $n\in\mathbb{N}_0$. Then,
\begin{enumerate}
    \item There exists a unique piecewise rational continuous family of co-algebra morphisms $\Phi_t $ such that
      \begin{enumerate}
          \item $\Phi_a = \Phi $ 
          \item $(\Phi_t,H_t )$ is a solution of the differential equation \eqref{eq-diff}, where $  H_t$ is the $\Phi_t $-co-derivation whose $n$-th Taylor coefficient is $H_t^{(n)} $ for all $ n \geq 0$.  
      \end{enumerate}
    \item Moreover, for all $ t \in [a,b] $, $(\Phi_s, H_s)_{s\in[a,t]}$ is a Lie $\infty $-algebroid homotopy between  $\Phi $ and $\Phi_t$.
 \end{enumerate}
\end{prop}
\begin{proof}
Let us show item (1). We claim that equation \eqref{eq-diff} is a differential equation that can be solved recursively. In arity zero, it reads,\begin{equation}\label{eq-diffRecursion}
		\frac{\dd\Phi_t^{(0)}}{\dd t}=Q_{\E}^{(0)}\circ H_t^{(0)}+H_t^{(0)}\circ Q_{\E'}^{(0)}
		\end{equation}and \begin{equation}\label{sol}\Phi_t^{(0)}=\Phi^{(0)}+\int_a^t\left(Q_{\E}^{(0)}\circ H_s^{(0)}+H_s^{(0)}\circ Q_{\E'}^{(0)}\right)\dd s\end{equation} is defined for all $t\in[a,b]$. Also, $\frac{\dd}{\dd t}\Phi_t^{(n+1)}\colon S^{(n+2)}(\E')\rightarrow\E$ is an algebraic expression of $Q_\E^{(0)},\ldots,Q_\E^{(n+1)}$, $Q_{\E'}^{(0)},\ldots, Q_{\E'}^{(n+1)}$ $\Phi_t^{(0)},\ldots,\Phi_t^{(n)},H_t^{(0)},\ldots,H_t^{(n+1)}$. But $\Phi_t^{(n+1)}$ does not appear in the $(n+1)$-th Taylor coefficient of $Q_{\E}\circ H_t+H_t\circ Q_{\E'}$ by Equation \eqref{TC-expension}. 
By Lemma \ref{lem:primitivesDerivatives}, there exists a unique piecewise rational continuous solution $\Phi_t^{(n+1)} $ such that $\Phi_a^{(n+1)} = \Phi^{(n+1)}$. The construction of the Taylor coefficients of the co-algebra morphisms $ \Phi_t$ then goes by recursion. Recursion formulas also show that $ \Phi_t$ is unique.

Let us show 2), i.e. that $\Phi_t $ is a a $\mathcal O $-multilinear chain map for all $ t \in [a,b]$.
The function given by $$\Lambda_k(t)=(Q_{\E}\circ \Phi_t-\Phi_t\circ Q_{\E'})^{(k)}\quad\text{for all}\quad t\in[a,b],\, k\in\mathbb{N}_0$$ are differentiable w.r.t $t$ at all points except for a finitely many  $t \in [a,b]$ and  are piecewise rational continuous. The map $\frac{\dd\Phi_t}{\dd t}$ is a Lie $\infty$-morphism because $ Q_\E^2=0$ and $ Q_{\E'}^2=0$, hence $\Lambda'(t)=0$. By continuity, $\Lambda_k(t)$ is constant over the interval $[a,b]$. Since $\Phi_a=\Phi$ is a Lie $\infty$-algebroid morphism, we have $\Lambda_k(a)=0$. Thus, $\Lambda_k(t)=0$ and, $$Q_{\E}\circ \Phi_t=\Phi_t\circ Q_{\E'},\quad \text{for all}\; t\in[a,b].$$ \end{proof}

\begin{lemma}
Let $(\E',(\ell_k')_{k\geq 1},\rho_{\E'})$ and $(\E,(\ell_k)_{k\geq 1},\rho_{\E})$ be Lie $\infty$-Lie algebroids that terminate in $\mathcal A$ through hooks $\pi'$ and $\pi$. Let $(\Phi_t,H_t)_{t\in[a,b]}$ be a homotopy. If $\Phi_a$ is Lie $\infty$-algebroid morphism that terminates at $\mathcal A$ (i.e. $\pi \circ \Phi = \pi'$), then so is the $\infty$-algebroid morphism $\Phi_t$ for all $t\in[a,b]$.
\end{lemma}
\begin{proof}
This is a direct consequence of Equation \eqref{sol},
since $Q_{\mathcal E}^{(0)}= \ell_1 : \mathcal E_{-2} \to \mathcal E_{-1} $ and $Q_{{\mathcal E}'}^{(0)}= \ell_1' :  \mathcal E_{-2}' \to \mathcal E_{-1}' $ are valued in the kernels of $ \pi$ and $ \pi'$ respectively.

Last, $\mathcal{O}$-multilinearity of $\Phi_t$ follows from the $\mathcal{O}$-multilinearity of $Q_{\E}\circ H_t+H_t\circ Q_{\E'}$, which is granted by Proposition \ref{linearity}. This completes the proof.
\end{proof}
Let us show that homotopy in the sense above defines an equivalence relation $\mathtt{\sim}$ between Lie $\infty$-morphisms. We have the following lemma.

\begin{lemma}\label{Homotopy-lemma}{
 A pair $(\Phi_t , H_t)$ is a homotopy between Lie $\infty$-algebroid morphisms $ \Phi_a$ and $ \Phi_b $  if and only if for all rational function, $g\colon[a,b]\rightarrow[c,d]$ without poles on $[a,b]$, the pair $(\Phi_{g(t)} , g'(t)H_{g(t)})$ is a homotopy between $ \Phi_{g(a)}$ and $ \Phi_{g(b)}$.}
 
\end{lemma}

\begin{proof}
 Let $g\colon [a,b]\rightarrow[c,d]$  be a rational function without poles on $[a, b]$. A straightforward computation  gives:
 $$
 \begin{array}{rrcll}
    &\dfrac{\dd\Phi_t}{\dd t}&=&Q_\E\circ H_t+H_t\circ Q_{\E'}&\hspace{.2cm} \hbox{(by definition}) \\ 
    \Rightarrow&\frac{\dd\Phi}{\dd t}(g(t))&=&Q_\E\circ H_{g(t)}+H_{g(t)}\circ Q_{\E'}& \hspace{.2cm}  \hbox{(by replacing $t$ by $g(t)$)} \\ \Rightarrow& \frac{\dd\Phi_{g(t)}}{\dd t}&=&Q_\E\circ \left(g'(t)H_{g(t)}\right)+\left(g'(t)H_{g(t)}\right)\circ Q_{\E'} &  \hspace{.2cm} \hbox{(by multiplying by $g'(t)$)} .
 \end{array}
 $$
  The last equation means that $(\Phi_{g(t)} , g'(t)H_{g(t)})$ is a homotopy between $ \Phi_{g(a)}$ and $ \Phi_{g(b)} $. {The backward implication is obvious, it suffices to consider $a=c$, $b=d $ and $g={\mathrm{id}}$}.
\end{proof}
\begin{proposition}
Homotopy between Lie $\infty$-morphisms is an equivalence relation. In addition, it is compatible with composition, that is, if $\Phi,\Psi\colon(S^\bullet_\mathbb{K}(\E')  , Q_{\E'} )\rightarrow(S^\bullet_\mathbb{K}(\E)  , Q_{\E} )$ are homotopic Lie $\infty$-algebroid morphisms and $ \hat{\Phi},\hat{\Psi}\colon(S^\bullet_\mathbb{K}(\E)  , Q_{\E} )\rightarrow (S^\bullet_\mathbb{K}(\E'') , Q_{\E''} ) $ are homotopic Lie $\infty$-algebroid morphisms, then, so are their compositions $\hat{\Phi}\circ\Phi$  and $\hat{\Psi}\circ\Psi$.
\end{proposition}
\begin{proof}We first show that this notion of homotopy is an equivalence relation.
 Let $\Phi,\Psi$ and $\Xi\colon S^\bullet_\mathbb{K}(\E')\longrightarrow S^\bullet_\mathbb{K}(\E)$ be three Lie $\infty$-morphisms of algebroids. \begin{itemize}
	\item [$\bullet$] Reflexivity: The pair $(\Phi_t=\Phi,H_t=0)_{t\in[0,1]}$ defines a homotopy between $\Phi$ and $\Phi$.
	\item [$\bullet$]Symmetry: Let $(\Phi_t,H_t)_{t\in[0,1]}$ be a homotopy between $\Phi$ to $\Psi$. By applying Lemma \ref{Homotopy-lemma} with $g(t)=1-t$, we obtain a homotopy between $\Psi$ and $\Phi$ via the pair $(\Phi_{1-t},-H_{1-t})_{t\in[0,1]}$.
	\item [$\bullet$]Transitivity: Assume $\Phi\mathtt{\sim}\Psi$ and $\Psi\mathtt{\sim}\Xi$ and let  $(\Phi_t,H_{1,t})_{t\in[0,\frac{1}{2}]}$ be homotopy between $\Phi$ and $\Psi$ and let $(\Psi_t,H_{2,t})_{t\in[\frac{1}{2},1]}$ be a homotopy between $\Psi$ and $\Xi$. By gluing $\Phi_t$ and $\Psi_t$, respectively $H_{1t}$ and $H_{2,t}$ we obtain a homotopy $(\Theta_t,H_t)_{t\in[0,1]}$ between $\Phi$ and $\Xi$.
\end{itemize}  

We then show it is compatible with composition. Let us denote by $(\Phi_t , H_t )$ the homotopy between $\Phi$ and $\Psi$, and $(\hat{\Phi}_t , \hat{H}_t )$ the homotopy between $\hat{\Phi}$ and $\hat{\Psi}$ . We obtain, \begin{align*}
    \frac{\dd \hat{\Phi}_t\circ\Phi_t}{\dd t}&=\frac{\dd \hat{\Phi}_t }{\dd t}\circ\Phi_t+\hat{\Phi}_t\circ\frac{\dd \hat{\Phi}_t }{\dd t}\\&=Q_{\E''}\circ\left(\hat{H}_t\circ\Phi_t+\hat{\Phi}_t\circ H_t \right) +\left(\hat{H}_t\circ\Phi_t+\hat{\Phi}_t\circ H_t \right)\circ Q_{\E'}.
\end{align*}Hence, $\hat{\Phi}\circ\Phi$  and $\hat{\Psi}\circ\Psi$ are homotopic via the pair $( \hat{\Phi}_t\circ\Phi_t,\hat{H}_t\circ\Phi_t+\hat{\Phi}_t\circ H_t)$ which is easily checked to satisfy all axioms. This concludes the proof.
\end{proof}

We conclude this section with a lemma that will be useful in the sequel.


\begin{lemma}\label{gluing-lemma}
 Let $(\Phi_t ,  H_t)_{t\in[c,+\infty[}$ be a homotopy such that for all $n\in\mathbb{N}_0$ and for every $t\geq n$, $H_t^{(n)}=0$. Then the $n$-th Taylor coefficient
  $\Phi_t^{(n)}$ is constant on $[n, +\infty [$ and
 the co-algebra morphism $\Phi_\infty$ whose $n$-th Taylor coefficient is  $\Phi_t^{(n)}$ for any $n\in\mathbb{N}_0$ and $t \in [n, +\infty[$  is a Lie $\infty $-algebroid morphism.
 
 {Moreover, for $ g:[a,b[\rightarrow[c,+\infty[$ a rational function with no pole on $[a,b[$ and such that $\displaystyle\lim_{t\to b}g(t)=+\infty$, the pair $(\Phi_{g(t)} , g'(t)H_{g(t)})$ is a homotopy between $\Phi_c$ and $\Phi_\infty$. }
\end{lemma}
\begin{proof}
Since the $n$-th Taylor coefficient of the $\Phi_t $-co-derivation $\frac{\dd \Phi_t^{(n)}}{\dd t}=(Q_\E\circ H_t +H_t\circ Q_{\E'})^{(n)}$ depends only on $H_t^{(i)} $ for $i=0, \dots, n-1 $, 
we have by assumption  $\frac{\dd \Phi_t^{(n)}}{\dd t}=0$ for all $t\geq n$. As a consequence $\Phi_t^{(n)} $ is constant on $[n, +\infty[ $.   
It follows from Proposition \ref{prop:justify} that $\Phi_\infty $ is a Lie $\infty $-algebroid morphism since for every $n\in\mathbb{N}_0$ and $t\in[n,+\infty[$\begin{align*}
    (Q_\E\circ\Phi_\infty-\Phi_\infty\circ Q_{\E'})^{(n)}&=\sum_{i+j=n}(Q_\E^{(i)}\circ\Phi^{(i)}_t-\Phi_t^{(j)}\circ Q_{\E'}^{(j)})\\&=0\hspace{.4cm}\hbox{(since $\Phi_t$ is a Lie $\infty$-algebroid morphism)}.
\end{align*}

Let us prove the last part of the statement. By assumption, there exists $a \leq b_n \leq b $ such that for all $t \in [b_n ,b]$, we have $g(t)\geq n$, so that $\Phi_{g(t)}^{(n)}=\Phi_\infty^{(n)}$ and $ g'(t)H_{g(t)}^{(n)}=0$ on $[b_n,b]$. The function $\Phi_{t}^{(n)}$ (resp. $H_t^{(n)} $) being  piecewise rational continuous (resp. piecewise rational) on $[c,n]$, the same holds for $\Phi_{g(t)}^{(n)}$ (resp. $g'(t)H_{g(t)}^{(n)}$) on $[a,b_n]$. By gluing with a constant function $\Phi_\infty$ (resp. with $0$), we see that all Taylor coefficients of $\Phi_{g(t)} $ (resp. $g'(t)H_{g(t)}$) are piecewise rational continuous (resp. piecewise rational) with finitely many gluing points. This completes the proof.
\end{proof}

\begin{remark}
 Lemma \ref{gluing-lemma} explains how to glue infinitely many homotopies, at least when for a given $n$, only finitely of them affects the $n$-th Taylor coefficient.
\end{remark}
\section{Main results}

\label{sec:main}

\subsection{The two main theorems about the universal Lie $\infty $-algebroids of a Lie-Rinehart algebras}

Let us now extend the main results of \cite{LLS} from locally real analytic finitely generated singular foliations to arbitrary Lie-Rinehart algebras.

 \subsubsection{Existence Theorem \ref{thm:existence}}

Here is our first main result, which states that universal Lie $ \infty$-algebroids over a given Lie-Rinehart algebra exist. 
We are convinced that it may be deduced using the methods of semi-models categories as in Theorem 4.2 in \cite{Fregier}, but does not follow from a simple homotopy transfer argument. It extends Theorem 2.8 in \cite{LLS}.

\begin{theorem}\label{thm:existence}
Let $ \mathcal O$ be an algebra and $\mathcal A $ be a Lie-Rinehart algebra over $ \mathcal O$.
Any resolution  of $\mathcal A $ by free $\mathcal O $-modules 
\begin{equation}
    \label{eq:resolutions}
\cdots \stackrel{\ell_1} \longrightarrow\E_{-3} \stackrel{\ell_1}{\longrightarrow} \E_{-2} \stackrel{\ell_1}{\longrightarrow} \E_{-1} \stackrel{\pi}{\longrightarrow} \mathcal A \end{equation}
  comes equipped with a Lie $\infty $-algebroid structure whose unary bracket is $\ell_1 $ and that terminates in $\mathcal A $  through the hook $ \pi$.
\end{theorem}
Since any module admits free resolutions, Theorem \ref{thm:existence} implies that:

\begin{corollary}\label{cor:existence}
Any Lie-Rinehart algebra $\mathcal A$ admits a universal Lie $\infty$-algebroid.
\end{corollary}

While proving Theorem \ref{thm:existence}, we will see that if $\E_{-1} $ can be equipped with a Lie algebroid bracket (i.e. a bracket whose Jacobiator is zero), then all $k$-ary brackets of the universal Lie $ \infty$-algebroid structure may be chosen to be zero on $\E_{-1} $:

\begin{proposition} \label{prop:lksontNuls}
Let $(\E, \ell_1, \pi)$ be a free resolution of a Lie-Rinehart algebra $\mathcal{A}$. If $\E_{-1}$ admits a Lie algebroid bracket $[\cdot, \cdot] $ such that $\pi: \mathcal E_{-1} \to \mathcal A$ is a Lie-Rinehart morphism, then there exists a structure of universal Lie $\infty$-algebroid $(\mathcal E, (\ell_k)_{k\geq 1}, \rho_\mathcal E,\pi)  $  of $\mathcal{A}$ whose $2$-ary bracket  coincides with $[\cdot, \cdot]$ on $\mathcal E_{-1} $ and such that for every $k\geq 3$ the $k$-ary bracket $ \ell_k$ vanishes on $\bigodot^k\E_{-1}$.\end{proposition}

 \subsubsection{Universality Theorem \ref{th:universal} and corollaries}

Here is our second main result.  It is related to Proposition 2.1.4 in \cite{Fregier} (but morphisms are not the same), and extends Theorem 2.9 in \cite{LLS}. 

\begin{theorem} \label{th:universal}
Let $\mathcal A $ be a Lie-Rinehart algebra over $ \mathcal O$.  Given, 
\begin{enumerate}
    \item[a)] a Lie $ \infty$-algebroid $(\mathcal E', (\ell_k')_{k\geq 1}, \rho_{\mathcal E'}, \pi')$ that terminates in $\mathcal A $ through the hook $\pi'$, and
    \item[b)] any Lie $\infty $-algebroid $(\mathcal E, (\ell_k)_{k\geq 1}, \rho_{\mathcal E},\pi)$ universal for $\mathcal A $ through the hook $\pi$,
\end{enumerate} 
then
\begin{enumerate} 
\item there exists a morphism of Lie $\infty $-algebroids from $(\E', (\ell'_k)_{k \geq 1}, \rho_{\E'},\pi')$ to $(\E, (\ell_k)_{k\geq 1}, \rho_\E,\pi)$ over $\mathcal A $ .
\item and any two such morphisms  are homotopic.  
\end{enumerate}
\end{theorem}

Recall that a Lie $\infty $-algebroid morphism $\Phi $ as above is \textquotedblleft a morphism of Lie $\infty $-algebroid that terminates in $\mathcal A $\textquotedblright\,(or \textquotedblleft over $\mathcal A $\textquotedblright\,for short) if $\pi \circ \Phi^{(0)} = \pi' $, see Definition \ref{def:morph}. Here is an immediate corollary of Theorem \ref{th:universal}.

\begin{corollary}
\label{cor:unique} 
Any two universal Lie $\infty $-algebroids of a given Lie-Rinehart algebra are homotopy equivalent. 
This homotopy equivalence, moreover, is unique up to homotopy.
\end{corollary}

We will prove that the morphism that appears in Theorem \ref{th:universal} can be made trivial upon choosing a \textquotedblleft big enough\textquotedblright\,universal Lie $\infty$-algebroid:
\begin{proposition}\label{univ:precise}
Let $\mathcal A $ be a Lie-Rinehart algebra over $ \mathcal O$.  Given a Lie $ \infty$-algebroid structure $(\mathcal E', (\ell_k')_{k\geq 1}, \rho_{\mathcal E'},\pi')$ that terminates in $\mathcal A $ through a hook $\pi'$, 
then there exist a universal Lie $ \infty$-algebroid $(\mathcal E, (\ell_k)_{k\geq 1}, \rho_{\mathcal E},\pi)$ of $\mathcal A $ through a hook $\pi$ such that
\begin{enumerate}
\item $\mathcal E $ contains $\mathcal E' $ as a subcomplex,  
\item the Lie $ \infty$-algebroid morphism from $\mathcal E' $ to $\mathcal E $ announced in Theorem \ref{th:universal} can be chosen to be the inclusion map  $\mathcal E' \hookrightarrow \mathcal E $ (i.e. a Lie $\infty$-morphism  where the only non-vanishing Taylor coefficient is the inclusion $\mathcal E' \hookrightarrow \mathcal E $).
\end{enumerate}
 \end{proposition}
 
The following Corollary follows immediately from Proposition \ref{univ:precise}: 
\begin{corollary}
Let $\mathcal{A}$ be a Lie-Rinehart algebra over $\mathcal{O}$ and $\mathcal{B}$ be a Lie-Rinehart subalgebra of $\mathcal{A}$. Any universal Lie $\infty$-algebroid of $\mathcal B$ can be contained in a  universal Lie $\infty$-algebroid of $\mathcal A$. \end{corollary}

\subsubsection{Induced Lie $\infty $-algebroids structures on ${\mathrm{Tor}}_{\mathcal O}(\mathcal A, \mathcal O/\mathcal I) $.}\label{ssec:Tor}

 Let $\mathcal A $ be a Lie-Rinehart algebra over $\mathcal O $ with anchor $\rho_\mathcal A $. We say that an ideal $\mathcal I \subset \mathcal O $ is a \emph{Lie-Rinehart ideal} if
  $ \rho_\mathcal A (a) [\mathcal I] \subset  \mathcal I$ for all $a \in \mathcal A $.
  Since this assumption implies $[\mathcal I \mathcal A , \mathcal A] \subset \mathcal I \mathcal A $, the quotient space $ \mathcal A /\mathcal I $ comes equipped with a natural Lie-Rinehart algebra structure over $ \mathcal O/ \mathcal I$.
  
 For $(\mathcal E, (\ell_k)_{k\geq 1}, \rho_\mathcal E, \pi) $ an universal Lie  $\infty$-algebroid of $\mathcal A $, the quotient space $\mathcal E_\bullet / \mathcal I \simeq \mathcal O/\mathcal I \otimes_\mathcal O \mathcal E_\bullet $ comes equipped with an induced Lie $\infty $-algebroid structure: the $n$-ary brackets for $n \neq 2$ go to quotient by linearity, while for $n =2$, the $2$-ary bracket goes to the quotient in view of the relation $\rho_{\mathcal E} (\mathcal E_{-1}) [\mathcal I] \subset \mathcal I$.
 Also, $\pi $ goes to the quotient to a Lie-Rinehart algebra morphism $ \mathcal E_{-1}/\mathcal I \to \mathcal A/\mathcal I  $. 
 
 \begin{definition}
     Let $\mathcal A $ be a Lie-Rinehart algebra over $\mathcal O $. For every Lie-Rinehart ideal $\mathcal I \subset \mathcal O $, we call  \emph{Lie $\infty$-algebroid of $\mathcal I $} the quotient Lie $\infty$-algebroid  $\mathcal E_\bullet/\mathcal I $, with $(\mathcal E, (\ell_k)_{k\geq 1}, \rho_{\mathcal E},\pi)$ a universal Lie $\infty $-algebroid of $\mathcal A $.
 \end{definition}
 
 \begin{remark}
 \label{rmk:Tor}
  \normalfont
  The complexes on which the Lie $\infty$-algebroids of the ideal $\mathcal I $ are defined compute ${\mathrm{Tor}}_\mathcal O (\mathcal A, \mathcal O/\mathcal I) $ by construction. 
 \end{remark}

 Moreover,  for any two universal Lie  $\infty$-algebroids  for $\mathcal A $, defined on $\mathcal E, \mathcal E'$ the homotopy equivalences $\Phi: \mathcal E \to \mathcal E' $ and $\Psi: \mathcal E' \to \mathcal E$,  whose existence is granted by  Corollary \ref{cor:unique}, go to the quotient and induce an homotopy equivalences between $ \mathcal O /\mathcal I \otimes_\mathcal O \mathcal E_\bullet  \simeq \mathcal E_\bullet / \mathcal I$ and $\mathcal O /\mathcal I \otimes_\mathcal O  \mathcal E_\bullet' \simeq  \mathcal E_\bullet' / \mathcal I$.
 The following corollary is then an obvious consequence of Theorem \ref{th:universal}.
 
 \begin{corollary}
 \label{cor:LRideal}
  Let $\mathcal A $ be a Lie-Rinehart algebra over $\mathcal O $. Let $\mathcal I \subset \mathcal O $ be a Lie-Rinehart ideal. Then any two Lie $\infty $-algebroids of $\mathcal I $ are homotopy equivalent, and there is a distinguished class of homotopy equivalences between them.
 \end{corollary}
 
 Taking under account Remark \ref{rmk:Tor}, here is an alternative manner to restate this corollary.
  
 \begin{corollary}
 \label{cor:LRideal2}
  Let $\mathcal A $ be a Lie-Rinehart algebra over $\mathcal O $. Let $\mathcal I \subset \mathcal O $ be a Lie-Rinehart ideal. Then the complex computing $ {\mathrm{Tor}}_\mathcal O^\bullet (\mathcal A, \mathcal O/\mathcal I) $ comes equipped with a natural Lie $\infty $-algebroid structure over $\mathcal O/ \mathcal I $, and any two such structures are homotopy equivalent in a unique up to homotopy manner.
\end{corollary}
 
 When, in addition to being a Lie-Rinehart ideal, $\mathcal I $ is a maximal ideal,  then $\mathbb K:= \mathcal O/\mathcal I $ is a field and Lie $\infty $-algebroids of $\mathcal I $ are a homotopy equivalence class of Lie $\infty $-algebras. 
 In particular their common cohomologies, which is easily seen to be identified to ${\mathrm{Tor}}_\mathcal O^\bullet(\mathcal A, \mathbb K) $ comes equipped with a graded Lie algebra structure.
 In particular,  ${\mathrm{Tor}}_\mathcal O^{-1}(\mathcal A, \mathbb K) $ is a Lie algebra, ${\mathrm{Tor}}_\mathcal O^{-2}(\mathcal A, \mathbb K) $ is a representation of this algebra, and the $3$-ary bracket defines a class in the third Chevalley-Eilenberg cohomology of  ${\mathrm{Tor}}_\mathcal O^{-1}(\mathcal A, \mathbb K) $  valued in ${\mathrm{Tor}}_\mathcal O^{-2}(\mathcal A, \mathbb K) $. This class does not depend on any choice made in its construction by the previous corollaries and trivially extends the class called NMRLA-class in \cite{LLS}. If it is not zero, then there is no Lie algebroid of rank $r$ equipped with a surjective Lie-Rinehart algebra morphism onto $\mathcal A $, where $r$ is the rank of $\mathcal A $ as a module over $\mathcal O $. All these considerations can be obtained by repeating verbatim Section 4.5.1 in \cite{LLS} (where non-trivial examples are given).
 

 \subsubsection{Justification of the title}
 
\noindent 
Let us give a categorical approach of Theorem \ref{thm:existence} and Corollary \ref{cor:existence}.

\begin{definition}
We denote by \emph{Lie-$\infty $-alg-oids/$\mathcal O $}
the category where:
\begin{enumerate}
    \item objects are Lie $\infty $-algebroids over $ \mathcal O$,
    \item arrows are homotopy equivalence classes of morphisms of Lie $\infty$-algebroids over $\mathcal O$.
\end{enumerate}
\end{definition}

Theorem  \ref{th:universal} means that universal Lie $\infty $-algebroids over Lie-Rinehart algebra $\mathcal A $ are terminal objects in the subcategory of Lie-$\infty $-alg-oids/$\mathcal O $ whose objects are Lie $\infty $-algebroids that terminate in $\mathcal A $.

\noindent 
Let us re-state Corollary \ref{cor:unique} differently.
By associating to any Lie $\infty $-algebroid its basic Lie-Rinehart algebra (see Example \ref{ex:LieInftyInduced} and \ref{ex:basic}),
one obtains therefore a natural functor:
 \begin{itemize}
\item  from the category \emph{Lie-$\infty $-alg-oids/$\mathcal O $},
\item to the category of Lie-Rinehart algebras over $\mathcal O $.
 \end{itemize}
 
\noindent 
Theorem   \ref{thm:existence} gives a right inverse of this functor. In particular,  this functor becomes  an equivalence of categories when restricted to homotopy equivalence classes of \underline{acyclic} Lie $\infty $-algebroids over $ \mathcal O$, i.e:

\begin{corollary}
Let $ \mathcal O$ be an unital commutative algebra.
There is an equivalence of categories between:
\begin{enumerate}
    \item[(i)] Lie-Rinehart algebras over $ \mathcal O$, 
    \item[(ii)] acyclic Lie $\infty $-algebroids over $ \mathcal O$.
\end{enumerate}
\end{corollary}



\noindent
The corollary justifies the title of the article.\\

\begin{remark}
 \normalfont
In the language of categories, Corollary \ref{cor:LRideal}  means that there exists a functor from Lie-Rinehart ideals of a Lie-Rinehart algebra over $ \mathcal O$, to the category of Lie $\infty $-algebroids, mapping a Lie-Rinehart ideal $\mathcal I $ to an equivalence class of Lie $\infty $-algebroids over $\mathcal O / \mathcal I $. 
\end{remark}
 
\subsection{An important bi-complex: ${\mathfrak{Page}}^{(n)}(\E',\E)$}

\subsubsection{Description of ${\mathfrak{Page}}^{(n)}(\E',\E)$}

\label{bi-com}
Let $ \mathcal A$ be an $ \mathcal O$-module, and let 
 $ (\E, \dd , \pi) $ and $ (\E', \dd' , \pi') $ be complexes of projective $\mathcal O $-modules that terminates at $ \mathcal A$:
\begin{equation}
  \label{eq:EE'}\cdots\xrightarrow{\dd}\E_{-2}\xrightarrow{\dd}\E_{-1}\xrightarrow{\pi} \mathcal A, \quad \cdots\xrightarrow{\dd'}\E'_{-2}\xrightarrow{\dd'}\E'_{-1}\xrightarrow{\pi'} \mathcal A.
\end{equation}

  For every $k \geq 1$, the $(k+1)$-th graded symmetric power  $ \bigodot^{k+1} \E'$ of $ \mathcal E'$ over $ \mathcal O$ is a projective $\mathcal O $-module, and comes with a natural grading induced by the grading on $\E' $. 

\begin{definition}\label{def:bi-com}
Let $ k \in \mathbb{N}_0$. We call \emph{page number $k$ of $(\E,\dd,\pi) $ and $ (\E, \dd',\pi')$}  the bicomplex of $\mathcal O$-modules on the upper left quadrant $ \mathbb Z^- \times \mathbb{N}_0$ defined by:
\begin{align}
\label{eq:jgeq0}
\text{Page}^{(k)}(\E', \E)_{j,m}&:= \text{Hom}_{\mathcal O} \left( \bigodot^{k+1}  \E' \, _{|_{-k-m-1}} \, , \, \E_{j}\right),& \hspace{0.2cm }\hbox{  for $ m \geq 0  $ and $ j \leq -1 $}
\\
\label{eq:j=1}
\text{Page}^{(k)}(\E',\E)_{0,m}&:=\text{Hom}_{\mathcal O} \left( \bigodot^{k+1}  \E' \, _{|_{-k-m-1}}\, ,\,  \mathcal A \right),  & \hspace{0.2cm }\hbox{ for $ m \geq 0  $,}
\end{align}
together with the vertical differential defined for $\Phi$ in any one  of the two $\mathcal O$-modules \eqref{eq:jgeq0} or \eqref{eq:j=1} by
$$\delta(\Phi) \, (x_1,\ldots,x_{k+1}):=\Phi\circ\dd' \, (x_1\odot\ldots\odot x_{k+1}),\hspace{1cm} \forall \,x_1,\dots, x_{k+1}\in \mathcal E',$$
where $\dd'$ acts as an $\mathcal O$-derivation on $x_1\odot\ldots\odot x_{k+1}\in\bigodot^k\E'$ (and is 0 on  $\E_{-1}'$). The horizontal differential is given by 
 $$\Phi \mapsto \dd \circ \Phi \, \, \hbox{ or }\,\, \Phi \mapsto \pi \circ \Phi$$ 
depending on whether $ \Phi$ is of type \eqref{eq:jgeq0} with $j \leq -2 $ or the type \eqref{eq:jgeq0} with $j=-1$. It is zero on elements of type \eqref{eq:j=1}. We denote by $\left(\mathfrak{Page}^{(k)}_\bullet(\E',\E), D\right) $ its associated total complex. When $\E'=\E$ we shall write $\mathfrak{Page}^{(k)}_\bullet(\E)$ instead of $\mathfrak{Page}^{(k)}_\bullet(\E,\E)$.
\end{definition}

 The following diagram recapitulates the whole picture of $\mathfrak{Page}^{(k)}_\bullet(\E',\E)$:

\begin{equation}\label{recap}
\scalebox{0.8}{ \hbox{$
	\begin{array}{ccccccccccc}
		& & \vdots & & \vdots & & \vdots & & 
		\\ 
		& & \uparrow & & \uparrow & & \uparrow & & 
		\\ 
		\cdots& \rightarrow & \text{Hom}_\mathcal O\left(\bigodot^{k+1} \E'\,_{|_{-k-3}},\E_{-2}\right) & \overset{\dd}{\rightarrow} & \text{Hom}_\mathcal O\left(\bigodot^{k+1} \E'\,_{|_{-k-3}},\E_{-1}\right) 
		& \overset{\pi}{\rightarrow} & \text{Hom}_\mathcal O\left(\bigodot^{k+1} \E'\,_{|_{-k-3}},\mathcal{A}\right) & \rightarrow & 0
			\\ 
		& & \delta\uparrow & & \delta\uparrow & & \delta\uparrow & & 
		\\ 
		\cdots& \rightarrow & \text{Hom}_\mathcal O\left(\bigodot^{k+1} \E'\,_{|_{-k-2}},\E_{-2}\right) & \overset{\dd}{\rightarrow} & \text{Hom}_\mathcal O\left(\bigodot^{k+1} \E'\,_{|_{-k-2}},\E_{-1}\right) 
		& \overset{\pi}{\rightarrow} & \text{Hom}_\mathcal O\left(\bigodot^{k+1}\E'\,_{|_{-k-2}},\mathcal{A}\right) & \rightarrow& 0
		\\ 
		& & \delta\uparrow & & \delta\uparrow & & \delta\uparrow & &  
		\\ 
		\cdots& \rightarrow & \text{Hom}_\mathcal O\left(\bigodot^{k+1} \E'\,_{|_{-k-1}},\E_{-2}\right) & \overset{\dd}{\rightarrow} & \text{Hom}_\mathcal O\left(\bigodot^{k+1}\E'\,_{|_{-k-1}},\E_{-1}\right) 
		& \overset{\pi}{\rightarrow}& \text{Hom}_\mathcal O\left(\bigodot^{k+1}\E'\,_{|_{-k-1}},\mathcal{A}\right) & \rightarrow & 0
	\\
	& & \uparrow & & \uparrow & & \uparrow & & 
	\\ 
	& & 0 & & 0 & & 0 & &  \\
	& & \hbox{\small{\texttt{"-$2$ column"}}} & & \hbox{\small{\texttt{"-$1$ column"}}} & & \hbox{\small{\texttt{"last column"}}} & &  	
\end{array}
$}}
\end{equation}

For later use, we spell out the meaning of being $D$-closed.

\begin{lemma}
\label{lem:beingClosed}
An element $P \in \mathfrak{Page}^{(k)}_j(\E',\E)$  in $ \oplus_{ i \geq 1 }\emph{Hom}_\mathcal O\left(\bigodot^{k+1}\E'\,_{|_{-j-i}},\E_{-i}\right)$ is $D$-closed if and only if:
\begin{enumerate}
    \item the component $ P_{-1}\colon \bigodot^{k+1} \mathcal E'\,_{|_{-j-1}} \to \mathcal E_{-1}$ is valued in the kernel of $ \pi \colon \E_{-1} \to \mathcal A$,
    \item the following diagram commutes:
     $$ \xymatrix{\bigodot^{k+1} \mathcal E' \,_{|_{-j-i}}\ar[rr]^{(-1)^j\dd'} \ar[d]_{P_{-i}} &&\bigodot^{k+1} \mathcal E' \,_{|_{-j-i+1}} \ar[d]^{P_{-i+1}}\\ \E_{-i}  \ar[rr]_\dd && \E_{-i+1}}$$ 
     with $P_{-i} $ being the component of $P$ in $\emph{Hom}_\mathcal O\left(\bigodot^{k+1}\E'\,_{|_{-j-i}},\E_{-i}\right) $. 
\end{enumerate}
For $(\E, \dd) = (\E', \dd')$, the second condition above also reads $ [P,\dd]_{\hbox{\tiny{\emph{RN}}}}=0$. Here is now our main technical result.
\end{lemma}
\begin{proposition}
\label{bicomplex}
Let $ (\E, \dd, \pi)$ be a resolution of $ \mathcal A$ in the category of $ \mathcal O$-modules. Then,
for every $k \geq 0$,
\begin{enumerate}
    \item  the cohomology of the complex $(\mathfrak{Page}^{(k)}_\bullet(\E) , D) $  for the total differential $
D_\bullet:=\dd-(-1)^\bullet\delta $ is zero in all degrees;
    \item\label{2} Moreover, a $D$-closed element whose component on the \textquotedblleft    last column\textquotedblright\,of the diagram above is zero is the image through $D$ of some element whose  two last components are also zero.
   \item\label{3}
    More generally, for all $n \geq 1 $, for a $D$-closed element $P \in \mathfrak{Page}^{(k)}_j(\E)$  of the form $ \oplus_{ i \geq n } \emph{Hom}_\mathcal O\left(\bigodot^{k+1}\E'\,_{|_{-j-i}},\E_{-i}\right)$, one has $P=D(R) $ and $R \in  \mathfrak{Page}^{(k)}_{j-1}(\E) $ can be chosen in $\oplus_{ i \geq n+1 } \emph{Hom}_\mathcal O\left(\bigodot^{k+1}\E'\,_{|_{-j-i+1}},\E_{-i} \right) $.
\end{enumerate}
\end{proposition}
\begin{proof}
Since $ \bigodot^k \mathcal E' |_{j+m-k} $ is a projective $ \mathcal O$-module for all $(j,m)\in\mathbb{Z^-}\times\mathbb{N}_0$, and  $ (\E, \dd, \pi)$ is a resolution, all the lines of the above bicomplex are exact. This proves the first item. The second and the third are obtained by diagram chasing.
\end{proof}
We will need the consequence of the cone construction.
\begin{lemma} \label{lem:inFactInclusion}Let $(\mathcal R,\dd^\mathcal R,\pi^\mathcal R)$ be an arbitrary complex of projective $\mathcal O$-module that terminates in a $\mathcal O$-module $\mathcal A$.
There exists a projective resolution $(\mathcal E,\dd^{\mathcal E}, \pi^{\mathcal E})$ of $\mathcal A $, which contains $(\mathcal R,\dd^\mathcal R,\pi^\mathcal R)$ as a sub-complex. Moreover, we can assume that $\mathcal R$ admits a projective sub-module in $\E$ in direct sum.
\end{lemma}
\begin{proof}Resolutions of an $\mathcal O $-modules $\mathcal A $ are universal objects in the category of complexes of projective $\mathcal O $-modules. In particular, for every projective resolution $(\mathcal F,\dd^\mathcal F, \pi^\mathcal F)$ of $\mathcal A $, there exist a (unique up to homotopy) chain map:
 $$  \phi \colon (\mathcal R,\dd^\mathcal R,\pi^\mathcal R) \to (\mathcal F,\dd^\mathcal F, \pi^\mathcal F)  . $$
We apply the cone construction {(see, e.g. \cite{Weibel}, Section 1.5)} to: 
\begin{enumerate}
    \item the complex  $(\mathcal R,\dd^\mathcal R,\pi^\mathcal R)$
    \item the direct sum of the complexes $(\mathcal R,\dd^\mathcal R)$ and $(\mathcal F,\dd^\mathcal F)$  namely, $\left(\mathcal R\oplus\mathcal F,\dd^\mathcal R\oplus\dd^\mathcal F, \pi^\mathcal R\oplus\pi^\mathcal F \right) $ 
    \item \label{chain:cone} the chain map obtained by mapping any $ x  \in \mathcal R$ to $ (x, \phi(x)) \in \mathcal R\oplus\mathcal F$.
\end{enumerate}
The differential is given by \begin{align}
   \dd^{\E}(x, y, z)=(-\dd^\mathcal Rx , \dd^\mathcal Ry-x , \dd^\mathcal F z-\varphi(x)) 
\end{align} for all $(x,y,z)\in\E_{-i}=\mathcal R_{-i+1}\oplus\mathcal R_{-i}\oplus\mathcal F_{-i}$, $i\geq 2$. Since the chain given in item \ref{chain:cone} is a quasi-isomorphism, its cone is an exact complex. We truncate the latter at degree $-1$ without destroying its exactness by replacing the cone differential at degree $-1$ as follows:  $\pi^{\E}\colon\mathcal R_{-1}\oplus\mathcal F_{-1}\rightarrow \mathcal A,\;(r,e)\mapsto\pi^{\mathcal F}(e)-\pi^{\mathcal R}(r)$. For a visual description, see Equation \eqref{sub:resol} below: the resolution of $\mathcal A $ described in Lemma \ref{lem:inFactInclusion} is defined by:
\begin{equation}\label{sub:resol}
\xymatrix{ \cdots & \mathcal F_{-3} \ar^{\dd^\mathcal F}[rr]&&\mathcal F_{-2} \ar^{\dd^\mathcal F}[rr] && \mathcal F_{-1} \ar^{\pi^\mathcal F}[rr]&&  \mathcal A\\\cdots & \mathcal R_{-3} \ar^{\dd^\mathcal R}[rr]&& \mathcal R_{-2} \ar^{\dd^\mathcal R}[rr]&& \mathcal R_{-1}  \ar_{\pi^\mathcal R}[rru] && \\ \cdots &\mathcal R_{-2} \ar_{\mathrm{id}}[urr] \ar_{\dd^\mathcal R}[rr] \ar@/_/^<<<<<{\phi}[uurr] && \mathcal R_{-1} \ar_{\mathrm{id}}[urr] \ar@/_/^<<<<<{\phi}[uurr]&& &&  \\}
\end{equation}
The proof of the exactness of this complex is left to the reader.


\noindent
The henceforth defined complex $(\E,\dd^{\E}, \pi^{\E})$ is a resolution of $\mathcal A $, and obviously contains $(\mathcal R,\dd^\mathcal R,\pi^{\mathcal R})$  as a sub-chain complex of $\mathcal O $-modules. 
\end{proof}

Let $(\mathcal E,\dd^\mathcal E, \pi^\mathcal E)$ be a free resolution of $\mathcal A$ and $(\mathcal R,\dd^\mathcal R,\pi^\mathcal R)$ a subcomplex of projective $\mathcal O$-modules as in Lemma \ref{lem:inFactInclusion}. We say that  $P \in \mathfrak{Page}^{(k)}_j(\E)$  of the form $ \oplus_{ i \geq n } \text{Hom}_\mathcal O\left(\bigodot^{k+1}\mathcal E\,_{|_{-j-i}},\E_{-i}\right) $ \emph{preserves} $\mathcal R $ if $\bigodot^{k+1}\mathcal R\,_{|_{-j-i}}$ is mapped by $P$  to $\mathcal R_{-i}$ for all possible indices.
In such case, it defines by restriction to $\bigodot^\bullet\mathcal R$ an element $\iota^*_\mathcal R P$ in the graded $\mathcal{O}$-module $\mathfrak{Page}^{(k)}_j(\mathcal R):= \oplus_{ i \geq n } \text{Hom}_\mathcal O\left(\bigodot^{k+1}\mathcal R\,_{|_{-j-i}},\mathcal R_{-i}\right) 
$. For the sake of clarity, let us denote by $D^\E $ and $D^{\mathcal R} $ the respective differentials of the bi-complexes $ \mathfrak{Page}^{(k)}_{j}(\E)$
and $  \mathfrak{Page}^{(k)}_{j}(\mathcal R)$ and by $D^\E_h$, $D^\mathcal R_h $ and $D^{\mathcal R}_v$, $D^\mathcal R_v $ the horizontal differential resp. vertical differential, of their associated bi-complexes. Also, $\iota^*_\mathcal R P$ will stand for the restriction of $ P \in \mathfrak{Page}^{(k)}_j(\E)$ to $\bigodot^\bullet\mathcal R$ (a priori it is not valued in $\mathcal R$ but in $\E$).

\begin{lemma}\label{rest:univ}
Let $(\E,\dd^\E, \pi^\E)$  be a free resolution of $\mathcal A$. Let $\mathcal R \subset \E$ be a subcomplex made of free sub-$\mathcal O$-modules such that there exists a graded free $\mathcal O $-module $\mathcal V $ such that $\E =\mathcal R \oplus \mathcal V$.
\begin{enumerate}
    \item  For every $k\geq 0$, a $D^\E$-cocycle   $P \in \mathfrak{Page}^{(k)}_j(\E)$ which preserves $\mathcal R$ is the image through $D^\E$ of some element $Q \in \mathfrak{Page}^{(k)}_{j-1}(\E) $ which preserves $\mathcal R$ if and only if its restriction $\iota^*_\mathcal R P \in \mathfrak{Page}^{(k)}_j(\mathcal R)$ is a $D^\mathcal R$-coboundary.
    \item In particular, if the restriction of $\dd^\E $ and $\pi^\mathcal E$ to $\mathcal R$  makes it a resolution of $\pi^\E(\mathcal R_{-1}) \subset \mathcal A$, then any $D^\E$-cocycle   $P \in \mathfrak{Page}^{(k)}_j(\E)$ which preserves $\mathcal R$ is the image through $D^\E$ of some element $Q \in \mathfrak{Page}^{(k)}_{j-1}(\E) $ which preserves $\mathcal R $.
\end{enumerate}
\end{lemma}
\begin{proof}
Let us decompose the element $P\in\mathfrak{Page}^{(k)}_{j}(\E)$ as $P=\sum_{i\geq 1}P_i$ with, for all $i \geq 1$, $P_i $ in $\mathrm{Hom}_\mathcal O\left(\bigodot^{k+1}\mathcal E\,_{|_{-j-i}},\E_{-i}\right)$. Assume $P\in \mathfrak{Page}^{(k)}_{j}(\E)$ is a $D^\E$-cocycle which preserves $\mathcal R$.

Let us prove one direction of item 1. If $P$ is the image through $D^\E$ of some element $Q\in\mathfrak{Page}^{(k)}_{j-1}(\E)$ which preserves $\mathcal R$, then $D^{\mathcal R} (\iota^*_\mathcal R Q)=\iota^*_\mathcal R D^\E (Q)= \iota^*_\mathcal R P$, with $\iota^*_\mathcal R Q\in\mathfrak{Page}^{(k)}_{j-1}(\mathcal R)$. Thus, the restriction $\iota^*_\mathcal R P \in \mathfrak{Page}^{(k)}_j(\mathcal R)$ of $P$ is a $D^\mathcal R$-coboundary.

Conversely, let us assume that  $\iota^*_\mathcal R P \in \mathfrak{Page}^{(k)}_j(\mathcal R)$ is a $D^{\mathcal R}$-coboundary, i.e. 
 $ \iota^*_\mathcal R P= D^{\mathcal R} Q_{\mathcal R} $ for some $Q_{\mathcal R}\in\mathfrak{Page}^{(k)}_{j-1}(\mathcal R)$.
Take $\hat{Q}\in \mathfrak{Page}^{(k)}_{j-1}(\E)$ any extension of $Q_{\mathcal R}$ (e.g. define $\hat Q$ to be $0$ as soon as one element in $\mathcal V$ is applied to it).
Then $P- D^\E (\hat Q):\bigodot^{k+1}\E\longrightarrow\E$ is zero on $ \bigodot^{k+1} \mathcal R$. We have to check that it is a $D^\E$-coboundary of a map with the same property.
\noindent
Put $\kappa=P-D^\E (\hat Q)$. By Proposition \ref{bicomplex}, item 1, there exists $\tau\in \mathfrak{Page}^{(k)}_{j-1}(\E)$ such that $D^\E(\tau)= \kappa$. The equation $D^\E(\tau)=\kappa$ is equivalent to the datum of a collection of equations\begin{align}\label{eq:tau}
    D^\E_v(\tau_i)+D^\E_h(\tau_{i+1})=\kappa_{i+1}, i\geq 1,\quad \text{and}\quad D^\E_h(\tau_1)=\kappa_1,
\end{align}
{with, $\tau_i\in\mathrm{Hom}_\mathcal O\left(\bigodot^{k+1}\mathcal E\,_{|_{-j-i+1}},\E_{-i}\right)$ and  $\kappa_i\in\mathrm{Hom}_\mathcal O\left(\bigodot^{k+1}\mathcal E\,_{|_{-j-i}},\E_{-i}\right)$ for every $i \geq 1$.
Since $\iota^*_\mathcal R\kappa_1=0$, we have that $D^\mathcal E_h(\iota^*_\mathcal R\tau_1)=\iota^*_\mathcal R \left(D^\E_h(\tau_1)\right)=0$, (with the understanding that $\iota^*_\mathcal R{\tau_1}_{|_\mathcal V}\equiv 0$). Using the exactness of the horizontal differential $D^\E_h$, there exists $C_1\in \mathfrak{Page}^{(k)}(\mathcal E)$ such that $D^\mathcal E_h(C_1)=\iota^*_\mathcal R\tau_1$. We now change $\tau_1$ to $\tau_1'$ and $\tau_2$ to $\tau'_2$ by putting $\tau_1':=\tau_1-\iota^*_\mathcal R\tau_1$ and $\tau_2':=\tau_2+D^\mathcal E_v(C_1)$. One can easily check that Equation \eqref{eq:tau} still holds  under these changes, i.e.,   $$D^\mathcal E_v( \tau'_1)+D^\mathcal E_h( \tau'_2)=\kappa_2\quad \text{and} \quad D^\mathcal E_h(\tau'_2)+D^\mathcal E_v(\tau_3)=\kappa_3.$$ We can therefore choose $\tau$ such that $\iota^*_\mathcal R\tau_1=0$.} We then iterate this procedure, which allows us to choose $\tau \in \mathfrak{Page}^{(k)}_{j-1}(\E)$ such that $\iota^*_\mathcal R \tau=0$ and $D^\E(\tau)=\kappa$.
By construction, $Q:=\tau+\hat Q$  preserves $\mathcal R$, while $ \iota_{\mathcal R}^* Q =Q_\mathcal R$, and $D^\mathcal E (Q)=P$. The second item follows from the first one.
\end{proof}

\subsubsection{Interpretation of ${\mathfrak{Page}}^{(n)}(\E',\E)$} 
Let $\phi\colon (\mathcal E', \dd') \to (\E, \dd)  $ be a chain map,
and let $\Phi^{(0)} :  \bigodot^\bullet\mathcal E' \to  \bigodot^\bullet  \mathcal E $ be its natural extension to a co-algebra morphism, namely:
 $$ \Phi^{(0)} (x_1 \cdot \dots \cdot x_n):= \phi (x_1) \cdot \dots \cdot \phi(x_n) .$$
We denote by $Q^{(0)}_\E $ and $ Q_{\E'}^{(0)}$ the differentials of arity $0$ on $\bigodot^\bullet\mathcal E$ and $ \bigodot^\bullet\mathcal E'$ induced by $\dd $ and $\dd'$. As in Definition \ref{coder} and Proposition \ref{linearity}, $\Phi^{(0)} $-co-derivations of a given arity $n$ form a complex when equipped with
 $$  H \mapsto Q_\E^{(0)} \circ H + (-1)^{|H|}  H \circ Q_{\E'}^{(0)}.$$

\begin{proposition}
\label{prop:interpretation}
For every $k \in \mathbb{N}_0$, and $\phi\colon (\E', \dd') \to (\E,\dd) $ be a chain map as above. The complex of $\Phi^{(0)}$-co-derivations of arity $k$ is isomorphic to the complex $\widehat{\mathfrak{Page}}^{(k)}(\E',\E)$ obtained from  $\mathfrak{Page}^{(k)}(\E',\E)$ by crossing its \textquotedblleft last column\textquotedblright, see diagram \eqref{recap}.
\end{proposition}
\begin{proof}
The chain isomorphism consists in mapping a $\Phi^{(0)} $-co-derivation $H $ of arity $k$ and degree $j$ to its Taylor coefficient, which is an element of degree $j$ of ${\mathfrak{Page}}^{(k)}(\E',\E)$. It is routine to check that this map is a chain map.
\end{proof}

Here is an other type of interpretation for $\widehat{\mathfrak{Page}}^{\bullet}(\E',\E)$ involving the Richardson-Nijenhuis bracket.

\begin{prop}\label{rem:bicom-rch}
\cite{zbMATH03130695}
For $\E = \E' $, $\widehat{\mathfrak{Page}}^{\bullet}(\E',\E)$ is the the bi-graded complex  of \emph{exterior forms on $\E $} and the differential $D$ of $\widehat{\mathfrak{Page}}^{\bullet}(\E',\E)$ is  $D(\cdot) =[\mathrm d ,\cdot ]_{\hbox{\tiny{\emph{RN}}}}$.\end{prop}

\subsection{Existence: The Lie $\infty $-algebroid on a free $ \mathcal O$-resolution}

\subsubsection{Proof of Theorem \ref{thm:existence}}
\label{main:section}
In this section, we prove Theorem \ref{thm:existence}.

Consider $ (\E, \dd= \ell_1, \pi) $ a  resolution of $ \mathcal A$ by free $ \mathcal O$-modules: such resolutions always exist\cite{Weibel}.
To start with, we define a binary bracket $\ell_2$. The pair $ (\dd, \ell_2)$ will obey to the axioms of the object that we now introduce.

\begin{definition}\cite{LavauSylvain}
\label{def:al-oid}\label{def:almost}
An almost differential  graded Lie algebroid of a Lie-Rinehart algebra $ (\mathcal A, \rho_\mathcal A,  [\cdot\,, \cdot]_\mathcal A)$ is a complex 
$$    \cdots\stackrel{\dd}{\longrightarrow} \mathcal E_{-3}\stackrel{\dd}{\longrightarrow} \mathcal E_{ -2}  \stackrel{\dd}{\longrightarrow}  \mathcal E_{ -1} \stackrel{\pi}{\longrightarrow}    \mathcal A.$$
of projective $\mathcal O $-modules equipped with a graded symmetric degree $ +1$ $\mathbb{K}$-bilinear bracket $ \ell_2 = [\cdot\,, \cdot] $ such that:
\begin{enumerate}
\item  $ \ell_2$ satisfies the Leibniz identity with respect to the anchor $\rho_\mathcal E := \rho_{\mathcal A} \circ \pi \colon \mathcal E_{-1} \longrightarrow {\mathrm{Der}}(\mathcal O) $, 
\item $\dd$ is degree $+1 $-derivation of $\ell_2 $, i.e. for all  $x \in \mathcal E_{i}, y \in \mathcal E$:
$$ \dd \ell_2(x,y) +  \ell_2( \dd x,y) + (-1)^{i}  \ell_2( x  , \dd y) = 0,$$ 
\item $\pi $ is a morphism, i.e. for all  $x,y \in \mathcal E_{-1}$
$$ \pi (\ell_2 (x,y)) = [\pi (x), \pi (y)]_\mathcal A .$$
 \end{enumerate}
\end{definition}

\begin{lemma}\label{lem1}
Every free resolution $ (\E, \dd, \pi )$ of a Lie-Rinehart algebra $(\mathcal A , [ \cdot\,, \cdot]_\mathcal{A}, \rho_\mathcal A) $ comes equipped with a binary bracket $\ell_2$ that makes it an almost differential  graded Lie algebroid of $\mathcal A $.
\end{lemma}
\begin{proof}
For all $ k \geq 1$, let us denote by 
$(e_i^{(-k)})_{i \in I_k}$
a family of generators of the free $ \mathcal O$-module~$\E_{-k} $. 
By construction $\{a_i=\pi(e_i^{(-1)})\in \mathcal{A}\mid i\in I_{1}\}$ is a set of generators of $\mathcal{A}$. In particular, there exists elements $u^k_{ij}\in\mathcal{O}$, such that for given indices $i,j$, the coefficient $u^k_{ij}$ is zero except for finitely many indices $k$, 
and satisfying the skew-symmetry condition $u^k_{ij}=-u^k_{ji}$ together with
\begin{equation} \label{def:uijk}
\left[ a_i,a_j\right] _{\mathcal{A}}=\sum_{k\in I}u^k_{ij} a_k \hspace{.5cm} \forall i,j \in I_{1} 
\end{equation}
We now define: 
\begin{enumerate}
\item an anchor map by $\rho_{\mathcal{E}}(e_i^{(-1)})=\rho_{\mathcal{A}}(a_i$) for all $i\in I$,
\item a degree $ +1$ graded symmetric operation $\tilde{\ell}_2 $ on $ \E$ as follows:
\begin{enumerate}
    \item $\tilde{\ell}_2\left(e_i^{(-1)},e_j^{(-1)} \right) =\sum_{k\in I}u^k_{ij}e_k^{(-1)}$ for all $i,j\in I_{-1}$.
    \item 
    $\tilde{\ell}_2 \left(e_i^{(-k)},e_j^{(-l)} \right) =0$  for all $ i \in I_{k}, j \in I_l$ with $ k \geq 2$ or  $ l \geq 2$.  
\item we extend $\tilde{\ell}_2$ to $\E $ using $\mathcal O$-bilinearity and Leibniz identity with respect to the anchor $\rho_\E $.
\end{enumerate} 
\end{enumerate} 

By construction, $\tilde{\ell}_2$ satisfies the Leibniz identity with respect to the anchor $ \rho_\mathcal E$. Also, $\rho_\E\circ\dd=0$ on $\E_{-2}$. The map defined for all homogeneous $x,y\in \mathcal{E}$ by$$
 [\dd, \tilde{\ell}_2]_{\hbox{\tiny{RN}}}(x,y) = \dd \circ \tilde \ell_2 \left( x,y\right)+ \tilde \ell_2 \left(\dd x,y \right) +(-1)^{\lvert x\rvert} \tilde \ell_2 \left(x,\dd y \right),
$$
is a graded symmetric degree $ +2$ operation $ (\E \otimes \E)_\bullet \longrightarrow \E_{\bullet +2} $, and $[\dd, \tilde{\ell}_2]_{{\hbox{\tiny{RN}}}_{|_{\E_{-1}}}}=0$.
Let us check that it is $\mathcal O$-bilinear, i.e. for all $f \in \mathcal O, x,y \in \E$:
 $$  [\dd, \tilde{\ell}_2]_{\hbox{\tiny{RN}}}(x,f y )- f  [\dd, \tilde{\ell}_2]_{\hbox{\tiny{\text{RN}}}}(x,y )=0.$$
 
\begin{enumerate}
    \item if $x \in \E_{-1}$, this quantity is zero in view of
     \begin{align*}
     [ \dd, \tilde{\ell}_2]_{\hbox{\tiny{RN}}}(x,f y )&=f[\dd,\tilde{\ell}_2](x,y)+\underbrace{\dd\rho_{\mathcal{E}}(x)[f]\, y-\rho_{\mathcal{E}}(x)[f] \, \dd y}_{=0} 
\end{align*}
    \item if $x \in \E_{-2}$, one has
     \begin{align*}
      [ \dd, \tilde{\ell}_2]_{\hbox{\tiny{RN}}}(x,f y ) - f [ \dd, \tilde{\ell}_2](x, y )  = \tilde \ell_2 (\dd x , fy) - f \tilde \ell_2 (\dd x , y) 
      = \rho_\E  (\dd x)(f) \, y = 0 
     \end{align*}
     since $ \rho_\E \circ \dd =\rho_{\mathcal A} \circ \pi \circ \dd =0 $,
    \item if $ x \in \E_{-i} $ with $ i \geq 3$, it is obvious by $\mathcal O $-linearity of $\tilde \ell_2 $ on the involved spaces.
\end{enumerate}
\noindent
As a consequence $ [\dd, \tilde{\ell}_2]_{\hbox{\tiny{RN}}}$ is a degree $+2$ element in the total complex $\mathfrak{Page}^{(1)}(\E)$. By construction $ [\dd, \tilde{\ell}_2]_{\hbox{\tiny{RN}}}$ has no component on the last column. Since $ \pi ( [\dd, \tilde{\ell}_2]_{{\hbox{\tiny{RN}}} _{|_{\E_{-1}}}}) =0$ and also $[ \dd, [\dd, \tilde{\ell}_2]_{\hbox{\tiny{RN}}}]_{{\hbox{\tiny{RN}}}_{|_{\E_{\leq -2}}}}=0$, the $\mathcal O$-bilinear operator $ [\dd, \tilde{\ell}_2]_{\hbox{\tiny{RN}}}$ is $D$-closed in  $\mathfrak{Page}^{(1)}(\E)$.

By virtue of the first item of Proposition \ref{bicomplex}, the operator $ [\dd, \tilde{\ell}_2]_{\hbox{\tiny{RN}}}$ is then a $D$-coboundary so there exists  $\tau_2 \in \oplus_{j\geq 2} \text{Hom}_\mathcal O\left(\bigodot^2\E_{-j-1},\E_{-j}\right)$ such as $D(\tau_2)= -[\dd, \tilde{\ell}_2]_{\hbox{\tiny{RN}}}.$  Upon replacing  $ \tilde{\ell}_2$  by $\tilde{\ell}_2+\tau_2$ we obtain a 2-ary bracket $\ell_2$ of degree +1 which satisfies all items of Definition \ref{def:al-oid}.
\end{proof} 
\begin{proof}[Proof (of Theorem \ref{thm:existence})]
Lemma \ref{lem1} gives the existence of an almost differential graded Lie algebroid with differential $\ell_1=\dd$ and binary bracket $\ell_2$. We have to construct now the higher brackets $\ell_k$ for $k\geq 3$. 

\noindent
{\textbf{Step 1:
Construction of the 3-ary bracket $\ell_3 $.}} (Its construction being different from the one of the higher brackets, we put it apart).  We first notice that the graded Jacobiator defined for all $x,y,z \in \E$ by 
$$ \text{Jac}(x,y,z):= \ell_2(\ell_2(x,y),z)+(-1)^{\lvert y\rvert\lvert z\rvert}\ell_2(\ell_2(x,z),y)+(-1)^{\lvert x\rvert\lvert y\rvert+\lvert x\rvert\lvert z\rvert}\ell_2(\ell_2(y,z),x)  $$is $ \mathcal O$-linear in each variable, hence is a degree $ +2$ element in $\bigoplus_{j \geq 1} \text{Hom}_\mathcal O(\bigodot^3\E\, _{|_{-j-2}},\E_{-j}) \subset  \mathfrak{Page}^{(2)}(\E)$. For degree reason, its component on the last column of diagram \eqref{recap} is zero, i.e. it belongs to $\widehat{\mathfrak{Page}}^{(1)}(\E)$.

Let us check that it is $D$-closed: for this purpose we have to check that both conditions in Lemma \ref{lem:beingClosed} hold:
\begin{enumerate}
\item Since $\pi $ is a morphism from $(\E_{-1}, \ell_2)  $ to $(\mathcal A , [\cdot, \cdot]_{\mathcal A}) $, and since $ [\cdot, \cdot]_{\mathcal A} $ satisfies the Jacobi identity, one has for all $x,y,z\in\E_{-1}$: $$\text{Jac}(x,y,z)\in \ker\pi .$$ 
\item Furthermore, a direct computation of $[\text{Jac}, \dd]_{\hbox{\tiny{RN}}} $  gives in view of item $2$ of Definition \ref{def:almost}: \begin{equation*}
\dd \text{Jac}(x,y,z)=\text{Jac}(\dd x,y,z)+(-1)^{\lvert x\rvert}\text{Jac}(x,\dd y,z)+(-1)^{\lvert x\rvert+\lvert y\rvert}\text{Jac}(x,y,\dd z)
\end{equation*}\text{for all}\;$x,y,z \in \E$. 
\end{enumerate}

\noindent
Thus, $D(\text{Jac})=0$. By Proposition \ref{bicomplex}, item 2, $\text{Jac}$ is a $D$-coboundary, and, more precisely, there exists an element $\ell_3=\sum_{j\geqslant 2} \ell_3^{j}\in\widehat{\mathfrak{Page}}_1^{(2)}(\E)$ with $\ell_3^{j}\in\text{Hom}(\bigodot^3\E\, _{|_{-j-1}},\E_{-j})$ such that \begin{equation}
D(\ell_3) =-\text{Jac}\quad\text{i.e,}\quad\left[\dd,\ell_3 \right]_{\hbox{\tiny{RN}}} =-\text{Jac}.
\end{equation}
We choose the $3$-ary bracket to be $\ell_3$.

\noindent
{\textbf{Step 2:
Recursive construction of the $k$-ary brackets $\ell_k $ for $k \geq 4$.}}
Let us recapitulate: $ \ell_1=\dd$
, $\ell_2 $ and $\ell_3 $ are constructed and
the lowest arity terms of 
 $ [\ell_1 + \ell_2 +\ell_3, \ell_1 + \ell_2 +\ell_3]_{\hbox{\tiny{RN}}}$
satisfy
\begin{enumerate}
    \item  $[\ell_1, \ell_1]_{\hbox{\tiny{RN}}} =0 $ (since $\dd^2=0$),
    \item  $[\ell_1, \ell_2]_{\hbox{\tiny{RN}}} =0 $ (since $\dd=\ell_1$ and $\ell_2$ define an almost Lie algebroid structure) .
    \item $ [\ell_2, \ell_2]_{\hbox{\tiny{RN}}} + 2 [\ell_3, \ell_1]_{\hbox{\tiny{RN}}} = 2( \mathrm{Jac} +  [\ell_3 ,\ell_1]_{\hbox{\tiny{RN}}}) =0$ 
    by definition of $ \ell_3$, and because $[\ell_2, \ell_2]_{\hbox{\tiny{RN}}} = 2 \mathrm{Jac}$.
\end{enumerate}
However, the following term of degree $+2$ and arity $3 $ may not be equal to zero:
\begin{equation}
[\ell_3 , \ell_2]_{\hbox{\tiny{RN}}} \in\bigoplus\text{Hom}_\mathcal O\left(\bigodot^4\E_{j+1},\E_{-j}\right)=\widehat{\mathfrak{Page}}_1^{(3)}(\E).\end{equation}
Let us check that this term is indeed a $\mathcal{O}$-multilinear map: For $x_1\in \E_{-1},x_2,x_3,x_4\in\E_{\leq -2}$ and $f\in\mathcal O$, the only terms of $(\ell_3\circ \ell_2+\ell_2\circ \ell_3)(x_1,fx_2,x_3,x_4)$ where the anchor shows up are: 
$$\begin{cases}
\ell_3(\ell_2(x_1,fx_2),x_3,x_4)&=\rho_\E(x_1)[f]\ell_3(x_2,x_3,x_4)+f(\ell_3(\ell_2(x_1,x_2),x_3,x_4))\\ (-1)^{|x_2|+|x_3|+|x_4|}\ell_2(f\ell_3(x_2,x_3,x_4),x_1)&=-\rho_\E(x_1)[f]\ell_3(x_2,x_3,x_4) \\ &+f((-1)^{|x_2|+|x_3|+|x_4|}\ell_2(f\ell_3(x_2,x_3,x_4),x_1))
\end{cases}$$
The terms containing the anchor map add up to zero. When there is  more elements in $\E_{-1}$, the computation follows the same line. Moreover, by graded Jacobi identity of the Richardson-Nijenhuis bracket:
 $$ [[\ell_1 + \ell_2 +\ell_3, \ell_1 + \ell_2 +\ell_3]_{\hbox{\tiny{RN}}}, \ell_1 + \ell_2 +\ell_3]_{\hbox{\tiny{RN}}}  =0 $$ 
The term of arity $4$ in the previous expression gives
$[[\ell_3, \ell_2]_{\hbox{\tiny{RN}}} , \ell_1]_{\hbox{\tiny{RN}}}=0$. Hence, by Proposition \ref{rem:bicom-rch}, $[\ell_3, \ell_2]_{\hbox{\tiny{RN}}}$ is a $D$-cocycle in the complex $\mathfrak{Page}^{(3)}(\E)$, whose components on the last column and the column $-1$ are zero. 
It is therefore a coboundary by Proposition  \ref{bicomplex} item 3: we can continue a step further and define $\ell_4\in \oplus_{j \geq 3} \text{Hom}\left(\bigodot^4\E _{|_{-j-1}},\E_{-j}\right)$ such  that: \begin{equation}
-\left[\ell_2,\ell_3 \right]_{\hbox{\tiny{RN}}}
=
\left[\ell_1,\ell_4 \right]_{\hbox{\tiny{RN}}}
=
\left[\dd,\ell_4 \right]_{\hbox{\tiny{RN}}} .
\end{equation}

We choose the $4$-ary bracket to be $\ell_4$. We now proceed by recursion. We assume that we have constructed all the $k$-ary brackets $\ell_k$ such as : \begin{equation}\label{Maurer}
\left[\dd,\ell_k \right]_{\hbox{\tiny{RN}}} =-\sum_{\overset{i+j=k+1}{i\leq j}}\left[\ell_i,\ell_j \right]_{\hbox{\tiny{RN}}}=-\frac{1}{2}\sum_{\overset{i+j=k+1}{i,j\geq 1}}\left[\ell_i,\ell_j \right]_{\hbox{\tiny{RN}}}
\end{equation} for every $k=1,\ldots,n$ with $n\geq 4$. The ($n+1$)-ary bracket is constructed as follows. First the operator $\sum_{\overset{i+j=k+1}{i,j\geq 1}}\left[\ell_i,\ell_j \right]_{\hbox{\tiny{RN}}}$ is checked to be $\mathcal O$-linear as before. Now, we have \begin{align*}
\sum_{\overset{i+j=n+2}{i,j\geq 1}}\left[\dd,\left[ \ell_i,\ell_j \right]_{\hbox{\tiny{RN}}}\right]_{\hbox{\tiny{RN}}} 
&= -2\sum_{\overset{i+j=n+2}{i,j\geq 1}}\left[\ell_i,\left[ \dd,\ell_j \right]_{\hbox{\tiny{RN}}}\right]_{\hbox{\tiny{RN}}}\quad\text{(by graded Jacobi identity)}
.\end{align*}
Since $\ell_j$ satisfies Equation \eqref{Maurer} up to order $n$, we obtain\begin{equation*}
\sum_{\overset{i+j=n+2}{i,j\geq 1}}\left[\dd,\left[ \ell_i,\ell_j \right]_{\hbox{\tiny{RN}}}\right]_{\hbox{\tiny{RN}}}=\sum_{\overset{i+j+k=n+3}{i,j,k\geq 1}}\left[  \ell_i,\left[ \ell_j,\ell_k\right]_{\hbox{\tiny{RN}}} \right]_{\hbox{\tiny{RN}}} =0,
\end{equation*}
\noindent
where we used the graded Jacobi identity of the  Nijenhuis-Richardson bracket in the last step. Therefore, $\sum_{\overset{i+j=n+2}{i,j\geq 1}}\left[\ell_i,\ell_j \right]_{\hbox{\tiny{RN}}}$, seen as an element in $\mathfrak{Page}^{(i+j-2)}(\E)$ by Remark \ref{rem:bicom-rch}, is a cocycle  and for degree reason it has no element on the last column, and the columns $-1,\ldots,3-n$ in \ref{recap}. The third item of Proposition \ref{bicomplex} gives the existence of an $(n+1)$-ary bracket $\ell_{n+1}$ such as\begin{equation*}
\left[\dd,\ell_{n+1} \right]_{\hbox{\tiny{RN}}} =-\sum_{\overset{i+j=n+2}{i\leq j}}\left[\ell_i,\ell_j \right]_{\hbox{\tiny{RN}}}.
\end{equation*}
{This completes the proof.}
\end{proof}

\subsubsection{Proof of Proposition \ref{prop:lksontNuls} and Proposition \ref{univ:precise}}

\begin{proof}[Proof (of Proposition \ref{prop:lksontNuls})]
This is a consequence of Proposition \ref{thm:existence} and the third item of the Proposition \ref{bicomplex}: If the component of $\text{Jac}$ on the column $-1$ is zero, we can choose $\ell_3$ with no component on the last column and in column $-1$ (see Proposition \ref{bicomplex}), i.e. the restriction of $\ell_3$ to $\bigodot^3\E_{-1}$ is zero. Then $\ell_3$ has no component on the last column, the column $-1$ and the column $-2$. so $[\ell_2,\ell_3]_{\hbox{\tiny{RN}}}$ has no component in the last column, $-1$ and $-2$ columns as well. Hence $\ell_4$ can be chosen with no component on column $-1$, $-2$ and $-3$ by the third item of Proposition \ref{bicomplex}. The proof continues by recursion.
\end{proof}
We finish this section with a proof of Proposition \ref{univ:precise}.
\begin{proof}[Proof (of Proposition \ref{univ:precise})]We prove this Proposition in two steps.
\begin{enumerate}
    \item Lemma \ref{lem:inFactInclusion} guarantees the existence a free resolution $(\E,\dd,\pi)$ of the Lie-Rinehart algebra $\mathcal A$ such that $\mathcal E$ contains $\mathcal E'$ and such that there exists a graded free module $\mathcal V$ with $\E ' \oplus \mathcal V = \E $.
    \item Let $D^\E$ and $D^{\E'}$ be as in the proof of Lemma \ref{rest:univ}. We construct the $n$-ary brackets on $\E$ by extending the ones of $(\mathcal E', (\ell_k')_{k\geq 1}, \rho_{\mathcal E'}, \pi')$ in the following way:
    \begin{enumerate}
        \item We first construct an almost Lie algebroid bracket $\tilde{\ell_2}$ on $ \E_{-1}$ that extends the $2$-ary bracket of $ \E_{-1}'$. Since the $2$-ary bracket is determined by its value on a basis, the existence of a free module $\mathcal V_{-1}$ such that $\mathcal E_{-1} ' \oplus \mathcal V_{-1} = \mathcal E_{-1} $ allows to construct $\tilde{\ell_2} $ on $\E $ such that its restriction to $ \mathcal E' $ is $\ell_2' $ and such that it satisfies the Leibniz identity.
        
        As in the proof of Theorem \ref{thm:existence} (to be more precise: Lemma \ref{lem1}), we see that $[ \tilde \ell_2 , \dd^\E ]_{\hbox{\tiny{RN}}}$  is $\mathcal O $-linear, hence belongs to ${\mathfrak{Page}}^{(2)}_{2} (\E) $ and is a $D^\E$-cocycle. Since $\E' $ is a Lie $ \infty$-algebroid, its restriction to $\bigodot^2\mathcal E' $ is zero. Lemma  \ref{rest:univ} allows to change $\tilde{\ell}_2 $ to an $2$-ary bracket $\ell_2:=\tilde{\ell}_2+\tau_2$ with $\tau_2=0$ on $\bigodot^2 \E'$. Hence $\ell_2$ defines a graded almost Lie algebroid bracket, whose restriction to $\E'$ is still $\ell_2' $.  
        
        \item Since $\ell_2$  is an extension of  $\ell_2'$, its Jacobiator $\text{Jac}\in \mathfrak{Page}^{(2)}_2(\E)$ of the $2$-ary bracket $\ell_2$ preserves $\E'$. Also, its restriction $\iota^*_{\E'}\text{Jac}\in \mathfrak{Page}^{(2)}_2({\E'})$ is the Jacobiator of $\ell_2' $, and the latter is the $D^{\E'}$-coboundary of $\ell_3'$ in view of the higher Jacobi identity of $\E'$. Since $\text{Jac}\in \mathfrak{Page}^{(2)}_2(\E)$ is a $D^\E$-cocycle, Lemma \ref{rest:univ} assures that $\text{Jac}$ is the image through $D^\E$ of some element $\ell_3\in\mathfrak{Page}^{(2)}_1(\E)$  which preserves $\E'$ and whose restriction to $ \bigodot^3\mathcal E'$ is $ \ell_3'$. The proof continues by recursion: at the $n$-th step, we use Lemma  \ref{rest:univ} to construct an $n$-ary bracket for $\E $ that extends the $n$-ary bracket of $\E' $.
    \end{enumerate}
\end{enumerate}
By construction, the inclusion map $\iota\colon \E'\hookrightarrow \E$ is a morphism for the $n$-ary brackets for all $n \geq 1$.
\end{proof}

\subsection{Universality: Proof of Theorem \ref{th:universal} }

Let $\mathcal{A}$ be a Lie-Rinehart algebra. We consider $(\E,(\ell_k)_{k\geq 1},\rho_\E)$ a universal Lie $\infty$-algebroid of $\mathcal{A}$: its existence is granted by Theorem \ref{thm:existence}, proved in Section \ref{main:section}. Let $(\E', (\ell_k')_{k\geq 1}, \pi', \rho_{\E'})$ be an arbitrary Lie $\infty$-algebroid that terminates in $\mathcal{A}$ through a hook $\pi'$. Let $Q_\E$ (resp. $Q_{\E'})$ be the co-derivations of $S^\bullet_\mathbb{K} (\mathcal E) $ (resp. of $S^\bullet_\mathbb{K} (\mathcal E') $) associated to the Lie $\infty$-algebroid structures $(\E, (\ell_k)_{k\geq 1}, \rho_\E,\pi)$ (resp. $(\E', \rho_{\E'}, (\ell_k')_{k\geq 1}),\pi'$) that terminate in $\mathcal A $. 

\noindent
Let us show that:
\begin{enumerate}
    \item there is a Lie $\infty$-algebroid morphism from $\E'$ to $\E$,
    \item Any such two Lie $\infty$-morphisms are homotopic.
\end{enumerate}
Altogether, these two points prove Theorem \ref{th:universal}.
The Taylor coefficients of the required Lie $\infty$-algebroid morphisms and homotopies will be constructed by induction. These inductions rely on Lemmas \ref{O-linearity-lemma} and \ref{imp-lemma} below.

\begin{lemma}\label{O-linearity-lemma}
Let $\Phi\colon S_{\mathbb K}^\bullet ( \mathcal E') \longrightarrow S_{\mathbb K}^\bullet ( \mathcal E )$ be a co-algebra morphism such that
\begin{enumerate}
     \item $\Phi$ is $\mathcal O $-multilinear,
    \item $\pi \circ \Phi^{(0)} = \pi'$ on $\E'$,
\end{enumerate}
For every $n \in \mathbb{N}_0$ such that\footnote{$\Phi\circ Q_{\E'}-Q_{\E}\circ\Phi$ being a $\Phi$-co-derivation, its component of arity $i$ is zero for $0\leq i\leq n$ if only if its $i$-th Taylor coefficient is zero for $i=1 \leq  \dots \leq n $.} $(\Phi\circ Q_{\E'}-Q_{\E}\circ\Phi)^{(i)}=0$  for every $0 \leq i \leq n$, then the map $S_{\mathbb K} ( \mathcal E' )\longrightarrow  S_{\mathbb K} ( \mathcal E)$ given by:
$$ (\Phi\circ Q_{\E'}-Q_{\E}\circ\Phi)^{(n+1)}  $$ 
\begin{enumerate}
    \item is a $\Phi^{(0)} $-co-derivation of degree $+1$,
    \item is $\mathcal O $-multilinear,
    \item and the induced  $\Phi^{(0)} $-co-derivation $ \left(\bigodot^\bullet \mathcal E', Q_{\E'}^{(0)}\right) \longrightarrow \left(\bigodot^\bullet\mathcal E,Q_{\E}^{(0)}\right)$ satisfies:
     $$    Q^{(0)}_\E\circ(\Phi\circ Q_{\E'}-Q_{\E}\circ\Phi)^{(n+1)}  =
          (Q_{\E}\circ\Phi-\Phi\circ Q_{\E'})^{(n+1)}\circ Q_{\E'}^{(0)} . $$
\end{enumerate}
\end{lemma}

\begin{remark}
The following remark will be crucial. 
Under the assumptions of Lemma \ref{lem:beingClosed}, $(\Phi\circ Q_{\E'}-Q_{\E}\circ\Phi)^{(n+1)} $ corresponds to a $D$-closed element of degree $+1$ in  the bi-complex $\mathfrak{Page}^{(n+1)}(\E',\E)$ through the chain isomorphism described in Proposition \ref{prop:interpretation}.
Here, $\E,\E' $ are equipped with the differentials $\ell_1, \ell_1' $ which are dual to the arity $0$ components $Q^{(0)}_\E,Q^{(0)}_{\E'}  $.
\end{remark}

\begin{proof}
A straightforward computation yields:
\begin{align*}
    \Delta(\Phi\circ Q_{\E'}-Q_{\E}\circ\Phi)&=(\Phi\otimes\Phi)\circ\Delta'\circ Q_{\E'}-(Q_{\E}\otimes \text{id}+\text{id}\otimes Q_{\E})\circ\Delta\circ\Phi\\&=\left((\Phi\circ Q_{\E'}-Q_{\E}\circ\Phi)\otimes\Phi+\Phi\otimes(\Phi\circ Q_{\E'}-Q_{\E}\circ\Phi) \right)\circ \Delta'.
\end{align*}
Now, $\Delta$ preserves arity i.e. $\Delta:S^n_\mathbb K(\E)\longrightarrow \oplus_{i+j=n} S^i_\mathbb K(\E)\otimes S^j_\mathbb K(\E) $ and so does $ \Delta'$. 
Taking into account the assumption  $(\Phi\circ Q_{\E'}-Q_{\E}\circ\Phi)^{(i)}=0$  for every $0 \leq i \leq n$, we obtain:
  \begin{eqnarray*}\Delta \circ (\Phi\circ Q_{\E'}-Q_{\E}\circ\Phi)^{(n+1)}\\ = 
   \left((\Phi\circ Q_{\E'}-Q_{\E}\circ\Phi)^{(n+1)}\otimes\Phi^{(0)}+\Phi^{(0)}\otimes(\Phi\circ Q_{\E'}-Q_{\E}\circ\Phi)^{(n+1)}\right)\circ \Delta'.\end{eqnarray*} 
   All the other terms disappear for arity reasons.    Hence $(\Phi\circ Q_{\E'}-Q_{\E}\circ\Phi)^{(n+1)}$ is a $\Phi^{(0)}$-co-derivation. 

Let us prove that it is $\mathcal{O}$-linear. It suffices to check  $\mathcal{O}$-linearity of\, $T_\Phi:= \Phi\circ Q_{\E'}-Q_{\E}\circ\Phi$. Let us choose homogeneous elements $x_1,\ldots, x_{N}\in\E'$ and let us assume that $x_i\in\E'_{-1}$ is the only term of degree $-1$: The proof in the case where there is more than one such an homogeneous element of degree $-1$ is identical. We choose $j \neq i$ and we compute $T_\Phi (x_1,\ldots,x_i,\ldots,fx_j,\ldots, x_{N})$ for some $f\in\mathcal{O}$. The only terms in the previous expression which are maybe non-linear in $f$ are those for which the $2$-ary brackets of a term containing $fx_j$ with $x_i$ or $\Phi^{(0)}(x_i)$ appear (since $\Phi $ and all other brackets are $\mathcal O$-linear). There are two such terms. The first one appears when we apply $Q_{\E'} $  first, and then $\Phi $: this forces
$\Phi\left(\ell'_2(x_i,fx_j),x_{I^{ij}}\right)$ to appear, and the non-linear term is then:
\begin{align}\label{term1}
 \epsilon (x,\sigma_i) \rho_{\E'}(x_i)[f] \,   \Phi(x_{I^i})
\end{align}
with $ \sigma_i$ the permutation that let $i$ goes in front and leave the remaining terms unchanged. There is a second term that appears when one applies $ \Phi$ first, then $ Q_\E$. Since it is a co-morphism, $\Phi ( x_1 \ldots x_i \ldots,fx_j \ldots x_{N} )  $ is the product of several terms among which only one is of degree $-1$, namely the term 
 $$ \epsilon (x,\sigma_i)  \Phi^{(0)}(x_i ) \Phi(f x_{I^i} ).$$
 Applying $Q_\E $ to this term yields the non-linear term 
  \begin{align}\label{term2}
      \epsilon (x,\sigma_i)\rho_\E ( \Phi^{(0)} (x_i))  [f] \, \Phi(x_{I^i}),
  \end{align} 

\noindent
where $I^i$ and $I^{ij}$ are as in Proposition \ref{linearity}. Since $\rho_\E\circ\Phi^{(0)}=\rho_{\E'}$, we see that the terms \eqref{term1} and \eqref{term2} containing an anchor add up to zero. 

Let us check that $(\Phi\circ Q_{\E'}-Q_{\E}\circ\Phi)^{(n+1)}$ is a chain map, in the sense that it satisfies item 3).
Considering again $T_\Phi:= \Phi\circ Q_{\E'}-Q_{\E}\circ\Phi$, we have that $T_\Phi^{(k)}=0$, for all $k=0,\ldots,n$. Since $T_\Phi\circ Q_{\E'}=Q_\E\circ T_\Phi$, one has
\begin{align*}
0=\left( T_\Phi\circ Q_{\E'}+Q_{\E}\circ T_\Phi\right)^{(n+1)}= T_\Phi^{(n+1)}\circ Q_{\E'}^{(0)}+Q_{\E}^{(0)}\circ T_\Phi^{(n+1)}+\sum_{\overset{i+j=n+1}{i,j\geq 1}}\underbrace{\left(T_\Phi^{(i)}\circ Q_{\E'}^{(j)}+Q_{\E}^{(j)}\circ T_\Phi^{(i)}\right)}_0
\end{align*}
By consequent, the $\mathcal O$-linear map $(\Phi\circ Q_{\E'}-Q_{\E}\circ\Phi)^{(n+1)}$ satisfies item 3).
\end{proof}

\begin{lemma}\label{imp-lemma}
Let $\Phi,\tilde{\Phi}: (S^\bullet_\mathbb{K} (\E'),Q_{\E'}) \longrightarrow (S^\bullet_\mathbb{K}(\E),Q_{\E})$ be $\mathcal O $-linear Lie $\infty$-algebroid morphisms and let $n \in \mathbb{N}_0$.
If $\tilde{\Phi}^{(i)}= \Phi^{(i)}$ for every $0 \leq i \leq n$, then
$(\tilde{\Phi}-\Phi)^{(n+1)} \colon S^\bullet_\mathbb{K} (\E') \longrightarrow S^\bullet_\mathbb{K} (\E )$
 \begin{enumerate}
     \item  is a $\Phi^{(0)}$-co-derivation
     \item is $\mathcal O $-multilinear
     \item  and the induced $\Phi^{(0)}$-co-derivation $\left(\bigodot^\bullet\E',Q_{\E'}^{(0)}\right) \longrightarrow \left(\bigodot^\bullet\E,Q_\E^{(0)}\right)$ satisfies:
      $$   Q^{(0)}_\E \circ(\tilde{\Phi}-\Phi)^{(n+1)}= (\tilde{\Phi}-\Phi)^{(n+1)} \circ Q_{\E'}^{(0)} .$$
 \end{enumerate}
\end{lemma}

\begin{remark}
Proposition \ref{prop:interpretation} means that the map
$(\tilde{\Phi}-\Phi)^{(n+1)} $ as in Lemma \ref{imp-lemma} corresponds to a closed element in $\mathfrak{Page}^{(n+1)}(\E',\E) $, equipped with differentials $\ell_1,\ell_1' $. 
\end{remark}

 \begin{proof}
  For all $x_1,\ldots,x_k\in\E'$, one has: \begin{align*}
\Delta(\tilde{\Phi}-\Phi)(x_1\odot\cdots\odot x_k)&=\sum^k_{j=1}\sum_{\sigma\in\mathfrak{S}_{j,k-j}} \epsilon(x,\sigma)  \tilde{\Phi}(x_{\sigma(1)}\odot\cdots\odot x_{\sigma(j)})\otimes\tilde{\Phi})(x_{\sigma(j+1)}\odot\cdots\odot x_{\sigma(k)})\\&-\sum^k_{j=1}\sum_{\sigma\in\mathfrak{S}_{j,k-j}}  \epsilon(x,\sigma) ( \Phi(x_{\sigma(1)}\odot\cdots\odot x_{\sigma(j)})\otimes\Phi(x_{\sigma(j+1)}\odot\cdots\odot x_{\sigma(k)})\\&=\left( (\tilde{\Phi}-\Phi)\otimes\Phi+\tilde{\Phi}\otimes(\tilde{\Phi}-\Phi)\right)\circ\Delta'(x_1\odot\cdots\odot x_k).
\end{align*}
Since $\Delta$ has arity $0$ and $(\tilde{\Phi}-\Phi)^{(i)}=0$ for all $0\leq i\leq n$, we obtain\begin{equation}
\Delta(\tilde{\Phi}-\Phi)^{(n+1)}(x_1\odot\cdots\odot x_k)=\left( (\tilde{\Phi}-\Phi)^{(n+1)}\otimes\Phi^{(0)}+\Phi^{(0)}\otimes(\tilde{\Phi}-\Phi)^{(n+1)}\right)\circ\Delta'(x_1\odot\cdots\odot x_k).
\end{equation} This proves the first item. 
Since both $\Phi $ and $\tilde{\Phi} $ are $\mathcal O$-multilinear, $(\tilde{\Phi}-\Phi)^{(n+1)}$ is $\mathcal O$-multilinear. Which proves the second item.  Since, $\Phi$ and $\tilde{\Phi}$ are  Lie ${\infty}$-morphisms:\begin{equation}
(\tilde{\Phi}-\Phi)\circ Q_{\E'}-Q_{\E}\circ(\tilde{\Phi}-\Phi)=0.
\end{equation}
By looking at the component of arity $n+1$, one obtains, $(\tilde{\Phi}-\Phi)^{(n+1)}\circ Q_{\E'}^{(0)}-Q_{\E}^{(0)}\circ(\tilde{\Phi}-\Phi)^{(n+1)}=0.$ This proves the third item.
 \end{proof}
 
\begin{lemma} \label{imp-lemma2}
Under the assumptions of Lemma \ref{imp-lemma},
 there exists 
 \begin{enumerate}
     \item a Lie $\infty$-morphism of algebroids $\tilde{\Phi}_1\colon  S_{\mathbb K}(\E')\rightarrow S_{\mathbb K}(\E)$ 
     \item and a homotopy $(\tilde{\Phi}_t,H_t)$ joining $\tilde{\Phi}_1$ to $\Phi$, compatible with the hooks,
 \end{enumerate}
 such that 
 \begin{enumerate}
     \item the components of arity less or equal to $n$ of $H_t$ vanish, 
     \item $\tilde{\Phi}_1^{(i)}=\tilde{\Phi}^{(i)}$ for every $0 \leq i \leq n+1$.
 \end{enumerate}
 \end{lemma}
\begin{proof} Let us consider $(\tilde{\Phi}-\Phi)^{(n+1)}: \bigodot^\bullet\mathcal E' \longrightarrow  \bigodot^\bullet\mathcal E $. 
By assumption, $ (\tilde{\Phi}-\Phi)^{(i)}=0 $ for all $i \leq n$, so that in view of Lemma \ref{imp-lemma}
\begin{enumerate}
    \item 
$(\tilde{\Phi}-\Phi)^{(n+1)}$ is a $\Phi^{(0)} $-co-derivation. 
    \item The restriction of
$(\tilde{\Phi}-\Phi)^{(n+1)}$
to $\bigodot^{n+2}\mathcal E'$ corresponds to a closed element of degree $0$ in ${\mathfrak{Page}}^{(n+1)} (\E',\E)$.
\end{enumerate}
 Proposition \ref{bicomplex} implies that there exists a degree $ -1$ $\mathcal O$-linear map $H_{n+1}\colon \bigodot^{n+2}(\E')\longrightarrow \E$, of arity $n+1$, hence an element in $\mathfrak{Page}^{(n+1)}(\E',\E)$, such that, 
\begin{equation}
(\tilde{\Phi}-\Phi)^{(n+1)}=Q^{(0)}_{\E}\circ H_{n+1}+H_{n+1}\circ Q_{\E'}^{(0)}.
\end{equation}
We denote its extension
to a $\Phi^{(0)}$-co-derivation of degree $-1$ by $H^{(n+1)}$. We now consider the following differential equation for $t \in [0,1]$: 
\begin{equation}\label{homotpy}
\frac{\dd\tilde{\Phi}_t}{\dd t}=Q_\E\circ H_t+H_t\circ Q_{\E'},\quad \text{and}\quad\tilde{\Phi}_0=\Phi,
\end{equation}where $H_t$ is the unique $\tilde{\Phi}_t$-co-derivation of degree $-1$  whose unique non zero Taylor coefficient is $H_{n+1}$. The existence of a solution for the differential equation \eqref{homotpy} is granted by Proposition \ref{prop:justify}. By considering the component of arity $1, \dots, n, n+1$ in Equation \eqref{homotpy}, we find
 $$ \left\{ \begin{array}{rcl}
 \frac{\dd\tilde{\Phi}_t^{(i)}}{\dd t} &  = &0 \, \, \, \hbox{ for $i=0, \dots, n$ ,} \\
 \frac{\dd\tilde{\Phi}_t^{(n+1)}}{\dd t}&=&Q_\E^{(0)}\circ H^{(n+1)}+H^{(n+1)}\circ Q_{\E'}^{(0)}  =   (\tilde{\Phi}-\Phi)^{(n+1)} 
 \end{array}\right.$$
Hence:
\begin{align*}
 \left\{ \begin{array}{rcl}
  {\tilde{\Phi}_t^{(i)}} &  = & \Phi^{(i)} \, \, \, \hbox{ for $i=0, \dots, n$ ,} \\
\tilde{\Phi}_t^{(n+1)}&=&
\Phi^{(n+1)}+t(\tilde{\Phi}-\Phi)^{(n+1)}. 
\end{array} \right.
\end{align*}
Therefore, applying $t=1$ to the previous relation, one finds $$
 \left\{ \begin{array}{rcl}
{\tilde{\Phi}_t^{(i)}} &  = & \Phi^{(i)} \, \, \, \hbox{ for $i=0, \dots, n$ ,} \\
 \tilde{\Phi}^{(n+1)}_1&=&\Phi^{(n+1)}+(\tilde{\Phi}-\Phi)^{(n+1)}= \tilde{\Phi}^{(n+1)} 
\end{array} \right. 
$$
This completes the proof.
\end{proof}
\begin{proof}[Proof (of Theorem \ref{th:universal})]	
Let us prove item 1. 
We construct  the Taylor coefficients of the Lie $\infty $-algebroid $ \Phi$ by recursion.

The Taylor coefficient of arity $0$ is obtained out of classical properties of projective resolutions of $\mathcal O $-modules.
Given any complex $(\E', \rho_{\E'}, \ell'_{1},\pi')$ which terminates in $\mathcal A$ through $\pi $, for every free resolution $(\E, \rho_\E, \ell_{1},\pi)$ of $\mathcal A$, there exists   a chain map  $\Phi^{(0)}\colon (\E',\ell_1') \to (\E, \ell_1)$ as in Equation (\ref{eq:EE'}), and any two such chain maps are homotopic. We still denote by $ \Phi^{(0)} $ its extension to an arity $ 0$ co-morphism $\bigodot^\bullet \E' \to \bigodot^\bullet \E $.

To construct the second Taylor coefficient,  let us consider the map:
\begin{equation} \label{degre2Obs} \begin{array}{rcl}   
  S^2_\mathbb K (\E')& \to  & \E \\
 (x,y)&\mapsto&\Phi^{(0)} \circ \ell_2'(x,y)-\ell_2(\Phi^{(0)}(x),\Phi^{(0)}(y)).\end{array}\end{equation}
 This map is in fact $\mathcal O $-bililinear, i.e. belongs to $\text{Hom}_\mathcal{O}(\bigodot^2\E',\E)$, hence to ${\mathfrak{Page}}^{(1)}(\E',\E)$, see Equation \eqref{recap}. Let us check that it is a $D$-cocycle:
 \begin{enumerate}
     \item[A.]  If either one of the homogeneous elements $x \in \E'$ or $y \in \E'$ is not of degree $-1$, a straightforward computation  gives:
     \begin{align*}
          \ell_1\circ\left( \Phi^{(0)}\circ\ell'_2(x,y)-\ell_2\left(\Phi^{(0)}(x),\Phi^{(0)}(y)\right)\right)&=\Phi^{(0)}\circ\ell_1'\circ\ell'_2( x,y)+\ell_2\left(\Phi^{(0)}\circ\ell_1'(x),\Phi^{(0)}(y)\right)\\&\hspace{2.1cm}+(-1)^{\lvert x\rvert}\ell_2\left(\Phi^{(0)}(x),\Phi^{(0)}\circ\ell_1'(y)\right)\\&=\left(\Phi^{(0)}\circ\ell'_2-\ell_2\left(\Phi^{(0)},\Phi^{(0)}\right)\right)\circ\ell_1'(x\odot y).
     \end{align*}
     \item[B.] If both $x,y \in \E'$ are of degree $-1 $:
     \begin{eqnarray*}\pi \left( \Phi^{(0)}\ell'_2(x,y)-\ell_2(\Phi^{(0)}x,\Phi^{(0)}y)\right) &=&  \pi'  \circ \ell_2'(x,y)- \pi \circ \ell_2\left(\Phi^{(0)}x,\Phi^{(0)}y\right)\\
     & =&   [\pi'(x), \pi'(y)]- \left[\pi \left(\Phi^{(0)}x\right), \pi\left(\Phi^{(0)}x\right) \right] \\&=& [\pi'(x), \pi'(y)]-[\pi'(x), \pi'(y)]\\ &=& 0 .\end{eqnarray*}
 \end{enumerate}
By Proposition \ref{bicomplex} item 2), there exists $\Phi^{(1)}\in \text{Hom}_\mathcal{O}\left(\bigodot^2 \E',\E\right)$, of degree $0$, so that \begin{equation}\label{phi1}
\Phi^{(0)}\circ\ell'_2(x,y)+\Phi^{(1)}\circ\ell_1'(x\odot y)=\ell_1\circ \Phi^{(1)}(x,y)+\ell_2(\Phi^{(0)}(x),\Phi^{(0)}(y)) \hspace{1cm} \hbox{for all $x,y \in \E'$.}
\end{equation}	
Since, $\Phi^{(0)}$ is a chain map, Relation \eqref{phi1} can be rewritten in terms of $Q_\E$ and $Q_{\E'}$ as follows\begin{align}  \left\{ \begin{array}{rcl} Q^{(0)}_{\E}\circ\Phi^{(0)}&=&\Phi^{(0)}\circ Q^{(0)}_{\E'}\\Q^{(0)}_{\E}\circ\Phi^{(1)}-\Phi^{(1)}\circ Q^{(0)}_{\E'}&=&\Phi^{(0)}\circ Q^{(1)}_{\E'}-Q^{(1)}_{\E}\circ\Phi^{(0)} \end{array} \right.\end{align}

The construction of the morphism $\Phi$ announced in Theorem \ref{th:universal} is then done by recursion. The recursion assumption is that we have already defined a $\mathcal{O}$-multilinear co-morphism  $\Phi:S_\mathbb{K}^\bullet(\E')\rightarrow S^\bullet_\mathbb{K}(\E)$ with $$\left( \Phi\circ Q_{\E'}-Q_{\E}\circ\Phi\right)^{(k)}=0 \quad\text{for all}\quad 0\leq k\leq n.$$
The co-morphism $\Phi:S_\mathbb{K}^\bullet(\E')\rightarrow S^\bullet_\mathbb{K}(\E)$ with Taylor coefficients $\Phi^{(0)} $ and $\Phi^{(1)} $ satisfies the recursion assumption for $n=1$.

Assume now that we have a co-morphism $\Phi $ that satisfies this assumption for some $n \in \mathbb{N}_0$, and
 consider the map $T_\Phi:= \Phi\circ Q_{\E'}-Q_{\E}\circ\Phi$.  Lemma \ref{O-linearity-lemma}  implies that it is a $\mathcal{O}$-multilinear $ \Phi^{(0)}$-coderivation, and that it corresponds to a $D$-closed element in $\mathfrak{Page}^{(n+1)}(\E',\E) $. 
 Since it has no component on the last column for degree reason, Proposition \ref{bicomplex} implies that  $T_\Phi^{(n+1)}$ is a coboundary: That is to say that there is a $\Phi^{(0)} $-co-derivation $\Theta\in \mathfrak{Page}^{(n+1)}(\E',\E)$ (of arity $n+1$ and degree $0$) which can be seen as a map $\Theta:\bigodot^{n+2} (\E')\rightarrow\E$ such that: $$T_\Phi^{(n+1)}=Q^{(0)}_{\E}\circ\Theta-\Theta\circ Q_{\E'}^{(0)}.$$
 Consider now the co-morphism $\tilde{\Phi}$ whose Taylor coefficients are those of $\Phi $ in arity $0, \dots,n$ and  $ \Phi^{(n+1)}+\Theta$ in arity $n+1$:  \begin{equation}\tilde{\Phi}^{(i)}:=\begin{cases} \Phi^{(i)}&\text{if $0\leq i\leq n$,} \\\Phi^{(n+1)}+\Theta &\text{if $i=n+1$}\end{cases}
\end{equation}
This is easily seen to satisfy the recursion relation for $n+1$. This concludes the recursion. The Taylor coefficients obtained by recursion define a Lie $\infty$-algebroid $\Phi\colon S^\bullet_{\mathbb K}(\E')\longrightarrow S^\bullet_{\mathbb K}(\E)$ which is compatible by construction with the hooks $\pi,\pi'$.

By continuing this procedure we construct a Lie $\infty$-morphism from $S^\bullet_{\mathbb K}(\E')$ to $S^\bullet_{\mathbb K}(\E)$. This proves the  first item of Theorem \ref{th:universal}.

Let us prove the second item in Theorem \ref{th:universal}.
Notice that in the proof of the existence of the Lie $\infty$-morphism between $S^\bullet_{\mathbb K}(\E')$ and $S^\bullet_{\mathbb K}(\E)$ obtained in the first item, we made many choices, since we have chosen a coboundary at each step of the recursion. 

Let $\Phi,\Psi$ be two Lie $\infty$-morphisms between $S^\bullet_{\mathbb K}(\E')$ and $S^\bullet_{\mathbb K}(\E)$. The arity $0$ component of the co-morphisms $\Phi$ and $\Psi$ restricted to $\E'$ are chain maps:
$$
\xymatrix{ \cdots\ar[rr]& & \ar[rr]  \mathcal E_{-2}'\ar@{-->}[lldd]_h\ar@<3pt>[dd]^{\Phi^{(0)}} \ar@<-3pt>[dd]_{\Psi^{(0)}}  & & \mathcal E_{-1} '\ar@<3pt>[dd]^{\Phi^{(0)}} \ar@<-3pt>[dd]_{\Psi^{(0)}}  \ar@{-->}[lldd]_h \ar[dr]^{\pi'} \\ & & & & & \mathcal A \\ \cdots\ar[rr] & & \mathcal E_{-2}  \ar[rr] & &\mathcal E_{-1} \ar@{->>}[ru]_\pi & } 
$$
which are homotopy equivalent in the usual sense because $(\E,\ell_1) $ is a projective resolution of $\mathcal A $:
said differently, there exists a degree $ -1$ $\mathcal O$-linear map $h\colon \E' \to \E$ such that \begin{equation}\Psi^{(0)} -\Phi^{(0)} =\ell_1 \circ h+h\circ\ell_1'\quad\text{on}\;\,\E'.\end{equation}
Let us consider the following differential equation:\begin{equation}\label{eq-diff1}
    \begin{cases}
    \frac{\dd\Xi_t}{\dd t}=Q_\E\circ H_t(\Xi_t)+H_t(\Xi_t)\circ Q_{\E'},&\text{for $t\in[0,1]$}\\\\\Xi_0=\Phi.
\end{cases}\end{equation}
with $H_t(\Xi_t)$ being a $\Xi_t$-co-derivation of degree $-1$ whose Taylor coefficient of arity $0$ is $h$. This equation does admit solutions in view of Proposition~\ref{prop:justify}.

By looking at the the component arity $0$ of Equation \eqref{eq-diff1} on $\E'$, one has: \begin{align*}
\frac{\dd\Xi^{(0)}_t}{\dd t}&=\ell_1 \circ h+h\circ \ell_1'\\&=\Psi^{(0)}-\Phi^{(0)}.
\end{align*}
Hence, $\Xi^{(0)}_t=\Phi^{(0)}+t\left( \Psi^{(0)}-\Phi^{(0)}\right)$ is a solution such that $\Xi^{(0)}_1=\Psi^{(0)}$. By construction, $\Xi_1$ is homotopic to $\Phi$ via the pair $\left( \Xi_t,H_t \right) $ over $[0,1]$, and its arity $0$ Taylor coefficient coincides with the Taylor coefficient of $ \Psi$. 

From there, the construction goes by recursion using Lemma \ref{imp-lemma2}. 
Indeed, this lemma allows to construct
recursively a sequence of Lie $\infty $-algebroids morphism $ (\Psi_n)_{n \geq 0} $ and homotopies  $(\Xi_{n,t},H_{n,t})$
(with $t \in [n,n+1] $) between $\Psi_n $ and $\Psi_{n+1} $ such that:  $H_{n,t}^{(i)}$ is zero for $t\geq n$ and $i\neq n+1$. By Lemma \ref{imp-lemma2}, all these homotopies are compatible with the hooks. These homotopies are glued in a homotopy $(\Xi_t,H_t)_{[0,+\infty[}$ such that for every $n\in\mathbb{N}_0$, the components of arity $n$ of the Lie $\infty$-algebroids morphism $\Xi_t^{(n)}$ are constant and equal to $\Psi^{(n)}$ for $t\geq n$. By Lemma \ref{gluing-lemma}, these homotopies can be glued to a homotopy on $[0,1]$. Explicitly, since $t\mapsto\frac{t}{1-t}$ maps $[0,1[$ to $[0,+\infty[$ and by Lemma \ref{gluing-lemma}, the pair $\left(\Xi_{\frac{t}{1-t}},\frac{1}{(1-t)^2}H_{k,\frac{t}{1-t}}\right)$ is a homotopy between $\Phi$ and $\Psi$.
This proves the second item of the Theorem \ref{th:universal}.

\end{proof}

\section{Examples of universal Lie $\infty$-algebroid structures of a Lie-Rinehart algebra}
\label{sec:examples}
\subsection{New constructions from old ones}

In this section, we explain how to construct universal Lie $\infty$-algebroids of some Lie-Rinehart algebra which is derived from a second one through one of natural constructions as in Section \ref{LR-Mor} (localization, germification, restriction), when a universal Lie $\infty$-algebroid of the latter is already known.\\
\label{constructions}

\subsubsection{Localization}\label{ssec:loc}

Localisation is an useful algebraic tool. When $\mathcal O $ is an algebra of functions, it corresponds to study local properties of a space, or germs of functions. 

Let $(\mathcal A,\lb_\mathcal A,\rho_\mathcal A)$ be a Lie-Rinehart algebra over a unital algebra $\mathcal O$. Let $S\subset\mathcal O$ be a multiplicative closed subset containing no zero divisor. We recall from item 2, Remark \ref{loc:res} that the localization $S^{-1}\mathcal A\cong\mathcal A\otimes_\mathcal O S^{-1}\mathcal O$ of $\mathcal A$ at $S$ comes equipped with  a natural structure of Lie-Rinehart algebra over the localization algebra $S^{-1}\mathcal O$. Recall that for $\varphi\colon \E \longrightarrow \mathcal T$ a homomorphism of $\mathcal O$-modules, there is a well-defined homomorphism of $\mathcal O$-modules, $$\varphi\otimes\text{id}\colon \mathcal E\otimes_\mathcal O S^{-1}\mathcal O\longrightarrow \mathcal T\otimes_\mathcal O S^{-1}\mathcal O,\hspace{0.2cm}\varphi\otimes\text{id} \, (x\otimes\frac{f}{s}):=\varphi(x)\otimes\frac{f}{s}$$ that can be considered as a $S^{-1}\mathcal O$-module homomorphism $$S^{-1}\varphi\colon S^{-1}\E\longrightarrow S^{-1}\mathcal T\hspace{0.2cm}\hbox{with}\hspace{0.2cm}S^{-1}\varphi\left(\frac{x}{s}\right):=\frac{\varphi(x)}{s},\,\;x\in\E,\,(f,s)\in\mathcal O\times S,$$ called the \emph{localization of $\varphi$}.
	
\noindent	
Given a Lie $\infty$-algebroid structure $(\E,(\ell_k)_{k\geq 1},\pi,\rho_\E)$ of $\mathcal A$. The triplet $(\E',(\ell_k')_{k\geq 1},\rho_{\E'},\pi')$ is a Lie $\infty$-algebroid structure that terminates at $S^{-1}\mathcal A$ through the hook $\pi'$ where
\begin{enumerate}
\item $\E'=S^{-1}\E$;
\item  The anchor map $\rho_{\E'}$ is defined by
\begin{align*}
  \begin{array}{rcrrcl} \rho_{\E'} \colon S^{-1}\E_{-1}&\longrightarrow &  \text{Der}(S^{-1}\mathcal O) & & & \\\label{der_ext} 
    \frac{x}{s}&\longmapsto & \rho_{\E'}\left(\frac{x}{s}\right): &S^{-1}\mathcal O&\longrightarrow &S^{-1}\mathcal O\\
&  & & \frac{f}{u}&\longmapsto&\frac{1}{s}\cdot\left(\frac{\rho_{\E}(x)[f]u-f\rho_{\E}(x)[u]}{u^2}\right) \end{array}
\end{align*}for $x\in\E,f\in\mathcal O,(s,u)\in S\times S$;
\item $\ell'_k=S^{-1}\ell_k$, for all $k\in\mathbb N\setminus\{2\}$;
\item The binary bracket is more complicated because of the anchor map: we
set $$ \ell_2'\left(\frac{1}{s} x,\frac{1}{u} y\right) =\frac{1}{su}\ell_2(x, y)- \frac{\rho_{\E}(x) [u]}{su^2} \, y + \frac{\rho_{\E}(y) [s]}{s^2u} \, x$$
for $x,y\in\E, (s,u)\in S^2$
(with the understanding that $\rho_\E\equiv0 $ on $\E_{-i}$ with $ i \geq 2 $);
\item  $\pi'=S^{-1}\pi$.
\end{enumerate}
 One can check that these operations above are well-defined and for all $z\in S^{-1}\E_{-1}$ the map $\rho_{\E'}(z)$ is  indeed a derivation on $S^{-1}\mathcal O$. The previously defined structure is also a Lie $ \infty$-algebroid that we call \emph{localization of the  Lie $\infty$-algebroid  $(\E,(\ell_k)_{k\geq 1},\pi,\rho_\E)$ with respect to $ S$}. 

\begin{proposition}
\label{prop:localisation}
Let $\mathcal O $ be a unital algebra and $S \subset \mathcal O$ a multiplicative subset containing no zero divisor. 
The localization of a universal Lie $\infty $-algebroid of a Lie-Rinehart algebra $ \mathcal A$ is a  universal Lie $\infty $-algebroid of $ S^{-1}\mathcal A$.
\end{proposition}
\begin{proof}
The object $(\E',(\ell_k')_{k\geq 1},\rho_{\E'},\pi')$ described above is also a Lie $\infty$-algebroid terminating in $S^{-1}\mathcal A$. It is universal because localization preserves exact sequences \cite{A.Gathmann}.
\end{proof}

\subsubsection{Restriction}\label{ssec:res}
When $\mathcal O_Y$ is the ring of functions of an affine variety $Y$, to every subvariety $X\subset Y$ corresponds its zero locus, which is an ideal $\mathcal I_X \subset \mathcal O_Y $. A Lie $\infty$-algebroid or a Lie-Rinehart algebra over $\mathcal O_Y$ may not restrict to a Lie-Rinehart algebra over $\mathcal O_X$: it only does so when one can quotient all brackets by $\mathcal I_X$, which geometrically means that the anchor map takes values in vector fields tangent to $X$. We can then \textquotedblleft restrict\textquotedblright, i.e. replace $\mathcal O_Y $ by $\mathcal O_Y / \mathcal I_X $. This operation has already been defined in Section \ref{ssec:Tor},
and here is an immediate consequence of Corollary \ref{cor:LRideal2}:

\begin{proposition}\label{prop:restriction}
Let $ \mathcal I \subset \mathcal O$ be Lie-Rinehart ideal, i.e. an ideal such that $ \rho_{\mathcal A}(\mathcal A)[\mathcal I] \subset \mathcal I$.
The quotient of a universal Lie $\infty$-algebroid of $ \mathcal A$ with respect to an ideal $ \mathcal I$ is a   Lie $\infty $-algebroid that terminates in $ \mathcal A/ \mathcal I \mathcal A$. It is universal if and only if
$((\frac{\E_{-i}}{\mathcal I\E_{-i}})_{i\geq 1},\Bar{\ell_1},\overline{\pi})$ is exact, i.e. if
$ {\mathrm{Tor}}_{\mathcal O}^\bullet (\mathcal A , \mathcal O/ \mathcal I)=0$.
\end{proposition}

	\subsubsection[]{Germification}\label{sec:germification}
	Let $W\subseteq\mathbb C^{N} $ be an affine variety and $\mathcal O_W$ its coordinates ring. Denote by $\mathcal O_{W,x_0}$ the germs of regular functions at $x_0$. Since $\mathcal O_{W,x_0}$ is a local ring, and since $\mathcal O_{W,x_0}\simeq(\mathcal O_W)_{\mathfrak m_{x_0}}$ \cite{Hartshorne},  where $\mathfrak m_{x_0}=\{f\in \mathcal O_W\mid f(x_0)=0 \}$ and $(\mathcal O_W)_{\mathfrak m_{x_0}}$ is the localization w.r.t the complement of $\mathfrak m_{x_0}$,	Proposition \ref{prop:localisation} implies the following statement:

		\begin{prop}
		\label{prop:germ}
		Let $W$ be an affine variety with functions $ \mathcal O_W$. For every point $x_0\in W$ and any Lie-Rinehart algebra $\mathcal A $ over $\mathcal O_W $, the germ at $x_0 $ of the universal Lie $\infty$-algebroid of $\mathcal A $ is the universal Lie $\infty$-algebroid of the germ of $\mathcal A $ at $x_0$. 
	\end{prop}
	
	\noindent
	Here, the germ at $x_0$ of a Lie-Rinehart algebra or a Lie $\infty$-algebroid is its  localization w.r.t the complement of $\mathfrak m_{x_0}$.

\subsubsection{Sections vanishing on a codimension $1$ subvariety} 

\label{sec:codim1} 

Let $(\mathcal{A},\lb,\rho)$ be an arbitrary Lie-Rinehart algebra over $ \mathcal O$. For any ideal $\mathcal I  \subset \mathcal O$, $\mathcal I \mathcal A $ is also a Lie-Rinehart algebra (see Example \ref{ex:vanishing}). When $ \mathcal O$ are functions on a variety $M$, $\mathcal I $ are functions vanishing on a subvariety $X$ and $\mathcal A $ is a $\mathcal O $-module of sections over $M$, $\mathcal I \mathcal A$ corresponds geometrically to sections vanishing along $X$. It is not an easy task.
In codimension $1$, i.e. when $\mathcal I $ is generated by one element, the construction can be done by hand.

\begin{prop}Let $(\mathcal A,\lb_\mathcal A,\rho_\mathcal A)$ be a Lie-Rinehart algebra over a commutative algebra $\mathcal O$. Let $(\E,\ell_k=\{\cdots\}_{k\geq 1},\rho_\E)$ be a Lie $\infty$-algebroid that terminates in $\mathcal A$ through a hook $\pi$. For any element $\chi\in \mathcal O$, the $\mathcal O$-module $\mathcal A'=\chi\mathcal A\subseteq\mathcal A$ is closed under the Lie bracket, so the triple  $(\chi \mathcal A, [\cdot, \cdot]_\mathcal A, \rho_\mathcal A) $ is a Lie-Rinehart algebra over $\mathcal O$. A Lie $\infty$-algebroid $(\mathcal E'=\E,\ell_k'=\{\cdots\}_{k \geq 1}', \rho_\mathcal E') $ hooked in $\chi \mathcal A $ through $\pi'$ can be defined as follows:
	\begin{enumerate}
	   \item The brackets are given by
	 \begin{enumerate}
	    \item $\{\cdot\}_1' = \{\cdot\}_1 $,
	    \item the $2$-ary bracket: \begin{equation}
	\{x,y\}'_{2}:=\chi\{x,y\}_{2}+\rho_\E(x)[\chi]\,y+(-1)^{|x||y|}\rho_\E(y)[\chi]\, x,
	\end{equation}
	for all $x,y \in \mathcal E_\bullet$, with the understanding that $\rho_\E=0 $ on $\mathcal E_{\leq -2} $, 
	    \item $\{\cdots\}_k' = \chi^{k-1}\{\cdots\}_k$ for all $k \geq 2$,
	    \end{enumerate}
        \item $\rho_{\E '}= \chi \rho_\E $,
	    \item $\pi' = \chi \pi $.
	    \end{enumerate}
\end{prop}	
\begin{proof}
We leave it to the reader.
 \end{proof}
 	\begin{prop}
 	If $\chi $ is not a zero-divisor in $\mathcal O$, and $(\E,\{\cdots\}_{k\geq 1},\rho_\E,\pi)$ is a universal Lie $\infty$-algebroid of $\mathcal A$, then the Lie $\infty $-structure described in the four items of \ref{constructions} is the universal Lie $\infty $-algebroid of $\chi \mathcal A $. 
 \end{prop}
\begin{proof}
		If $\chi $ is not a zero-divisor in $\mathcal A $ (i.e if $a \mapsto \chi a $ is an injective endomorphism of $\mathcal A $), then the kernel of $\pi' $ coincides with the kernel of $\pi $, i.e. with the image of $\{\cdot\}_1=\{\cdot\}_1$, so that $(\E,\ell_1,\chi\pi)$ is a resolution of $\chi\mathcal A$.
\end{proof}

\subsubsection{Algebra extension and blow-up}\label{blow-up-algebra}

 Recall that for $\mathcal O $ a unital algebra with no zero divisor,  
derivations of $\mathcal O$ induce derivations of its field of fractions $\mathbb O $.

 \begin{prop}\label{prop:blwoup}
 Let $\mathcal O $ be an unital algebra with no zero divisor, $\mathbb O $ its field of fractions, and  $\tilde{\mathcal O} $ an algebra with $\mathcal O \subset \tilde{\mathcal O} \subset \mathbb O$.   
 For every Lie-Rinehart algebra $\mathcal A $ over $\mathcal O$ whose anchor map takes values in derivations of $\mathcal O$ preserving $\tilde{\mathcal O} $, then
 \begin{enumerate}
     \item  any Lie $ \infty$-algebroid structure $ (\E, (\ell_k)_{k \geq 1}, \rho_\E, \pi)$ that terminates at $\mathcal A $ extends for all $i =0, \dots, n $ to a Lie $ \infty$-algebroid structure on $\tilde{\mathcal O} \otimes_{ {\mathcal O}} \mathcal E  $, 
     \item and this extension $\tilde{\mathcal O} \otimes_\mathcal O \mathcal E  $ is a Lie $\infty $-algebroid that terminates at the Lie-Rinehart algebra $\tilde{\mathcal O} \otimes_\mathcal O \mathcal A $.
 \end{enumerate}
 \end{prop}
 \begin{proof}
Since they are $ \mathcal O$-linear, the  hook $\pi$, the anchor $ \rho_\E$, and the brackets $\ell_k $  for $k \neq 2 $ are extended to $\tilde{\mathbb O} $-linear maps. Since the image of $ \rho_\E$  is the image of $\rho_\mathcal A $,  it is made of derivations preserving $\mathbb O $, which is easily seen to allow an extension of $\ell_2 $ to $ \tilde{\mathcal O} \otimes_\mathcal O \mathcal E $ using the Leibniz identity. 
 \end{proof}

\begin{remark}
 Of course, the Lie $\infty $-algebroid structure obtained on $ \tilde{\mathcal O} \otimes_\mathcal O \mathcal E $ is not in general the universal Lie $\infty $-algebroid of $ \tilde{\mathcal O} \otimes_\mathcal O \mathcal A $, because the complex $ (\tilde{\mathcal O} \otimes_\mathcal O \mathcal E , \ell_1, \pi)$
 may not be a resolution of $ \tilde{\mathcal O} \otimes_\mathcal O \mathcal A $ (see  Example \ref{ex:counter}). 
\end{remark}

\begin{remark}
\normalfont
Since any module over a field is projective, any Lie-Rinehart algebra over a field is a Lie algebroid. If we choose $\tilde{\mathcal O}  = \mathbb O $ therefore,
the Lie-Rinehart algebra  $ {\mathbb O} \otimes_\mathcal O \mathcal A $  is a Lie algebroid, so is homotopy equivalent to any of its universal Lie $\infty $-algebroid. 
 Unless $\mathcal A $ is a Lie algebroid itself, the Lie $\infty $-algebroid in Proposition \ref{prop:blwoup} will not be homotopy equivalent to a Lie $\infty $-algebroid whose underlying complex is of length one, and is therefore not a universal Lie $ \infty$-algebroid of $ {\mathbb O} \otimes_\mathcal O \mathcal A $.
 \end{remark}

\begin{example}
For $\mathcal O=\mathbb{C}[z_0,\ldots,z_N]$ the coordinate ring of $\mathbb{C}^{N+1}$, the blow-up of $\mathbb{C}^{N+1}$ at the origin is covered by affine charts: in the $i$-th affine chart $U_i$, the coordinate ring is $$\mathcal O_{U_i}=\mathbb{C}[z_0/z_{i}, \dots, z_i , \dots, z_{N}/z_{i} ].$$ 
 By Remark \ref{loc:res}, 
and Proposition \ref{prop:blwoup} 
for any
 Lie-Rinehart algebra $\mathcal A $  whose anchor map takes values in vector fields vanishing at $0 \in \mathbb C^{N+1}$,
 we obtain a Lie $\infty $-algebroid of $\mathcal O_{U_i} \otimes_\mathcal O \mathcal A $ that we call \emph{blow-up of at $0$ in the chart $ U_i$}. Proposition \ref{prop:blwoup} then says that the blow-up at $0$ of the universal Lie $\infty $-algebroid of $\mathcal A $, in each chart, is a Lie $\infty $-algebroid that terminates in the blow-up of $\mathcal A $  (as defined in remark \ref{loc:res}). It may not be the universal one, see Example \ref{ex:counter}.
\end{example}

\begin{example}[Universal Lie-$ \infty$-algebroids and blow-up: a counter example]\label{ex:counter}
Consider
the polynomial function in $N+1$ variables $\varphi=\sum_{i=0}^{N}z_i^3$. 
Let us consider  the singular foliation $\mathcal{F}_\varphi$  as in Example \ref{koszul}. Its generators are $\Delta_{ij}:= z_i^2 \tfrac{\partial}{\partial z_j}-z_j^2 \tfrac{\partial}{\partial z_i} $, for $0\leq i < j\leq N$. 

Let us consider its blow-up in the chart $U_N$.
Geometrically speaking, $\widetilde{\mathcal{F}_\varphi}= \mathcal O_{U_N} \otimes_{\mathcal O}\mathcal{F}_\varphi$ is the $\mathcal{O}_{ U_N}$-module generated  by the blown-up vector fields $\widetilde{\Delta }_{ij}=z_N\left( z^2_i\frac{\partial}{\partial z_j}-z_j^2\frac{\partial}{\partial z_i}\right),\,j\neq N$, and $\widetilde{\Delta}_{iN}=z_N\left( z_Nz^2_i\frac{\partial}{\partial z_N}-\frac{\partial}{\partial z_i} -z^2_i\sum_{j=0}^{N-1}z_j\frac{\partial}{\partial z_j}\right)$. The vector fields $\widetilde{\Delta}_{ij},\,j\neq N$, belong to the $\mathcal{O}_{ U_N}$-module generated by the vector fields $\widetilde{\Delta}_{iN}$. Since they are independent, the singular foliation $\widetilde{\mathcal{F}_\varphi}$ is a free $\mathcal O_{ U_N} $-module. Its universal Lie $\infty $-algebroid can be concentrated in degree $-1$, it is then a Lie algebroid.

On the other hand, the blow-up of the universal Lie $\infty $-algebroid of $\mathcal F_\varphi $ is not homotopy equivalent to a Lie algebroid. Indeed, the pull-back Lie $\infty $-algebroid $(\tilde{E}, (\tilde{\ell}_k)_{k \geq 1}, \tilde{\rho}, \tilde{\pi}) $ verifies by construction that  $ \tilde{E}_{-i} \neq 0$ for $i=1, \dots, N $ and that $\tilde{\ell}_1|_x=0 $ for every $x$ in the inverse image of zero, and such a complex can not be homotopy equivalent\footnote{We are in fact proving that $\mathcal O_{U_i} \otimes_\mathcal O \cdot$ is not an exact functor in this case.} to a complex of length $1$.

This example tells us that the blow-up of the universal Lie $\infty $-algebroid of a Lie-Rinehart algebra may not be the universal Lie $\infty$-algebroid of its blow-up.

\end{example}

\subsection{Universal Lie $\infty$-algebroids  of some singular foliations}

\subsubsection{Vector fields annihilating a Koszul function $\varphi $ }\label{koszul}
This universal Lie $\infty $-algebroid was already described in Section 3.7 of \cite{LLS}, where the brackets were simply checked to satisfy the higher Jacobi identities - with many computations left to the reader. Here, we give a theoretical explanation of the construction presented in \cite{LLS}. 

Let $\mathcal O$ be the algebra of all polynomials on $V:=\mathbb{C}^{d}$.
A function $\varphi \in \mathcal O$ 
is said to be a \emph{Koszul polynomial}, if the \emph{Koszul complex} 
\begin{equation}
    \label{eq:KoszulComplex}
\ldots\xrightarrow{\iota_{\dif\varphi}}\X^3 (V) \xrightarrow{\iota_{\dif\varphi}}\X^2 (V)\xrightarrow{\iota_{\dif\varphi}}\X(V) \xrightarrow{\iota_{\dif\varphi}}\mathcal O \longrightarrow 0\end{equation}

is exact in all degree, except in degree $0$.
By virtue of a theorem of Koszul \cite{zbMATH00704831}, see \cite{Matsumura} Theorem 16.5 $(i)$, $\varphi$ is \emph{Koszul} if $\left(\frac{\partial\varphi}{\partial x_1} ,\cdots,\frac{\partial\varphi}{\partial x_d}\right) $ is a regular sequence. 

From now on, we choose $\varphi $ a Koszul function, and consider the  singular foliation \begin{equation}\label{varphi}
\mathcal{F}_\varphi:=\{X\in\mathfrak{X}(V):X[\varphi]=0\} = {\mathrm{Ker}}( \iota_{\dif\varphi} )\colon \X(V) \xrightarrow{\iota_{\dif\varphi}}\mathcal O 
.\end{equation} 
 The Koszul complex \eqref{eq:KoszulComplex} truncated of its degree $0$ term  gives  a free resolution $(\E,\dif,\rho)$ of $\mathcal{F}_\varphi$, with $\E_{-i}:=\X^{i+1}(V)$, $\dd:=\iota_{\dif\varphi}$, and $\rho:=-\iota_{\dif\varphi}$.
 
\begin{remark}
Exactness of the Koszul complex implies in particular that $\mathcal{F}_\varphi$ is generated by the vector fields
: \begin{equation}\left\lbrace  \frac{\partial\varphi}{\partial x_i}\frac{\partial}{\partial x_j}-\frac{\partial\varphi}{\partial x_j}\frac{\partial\varphi}{\partial x_i},\mid 1\leq i<j\leq d\right\rbrace .
\end{equation}
\end{remark} 

In \cite{LLS}, this resolution is equipped with a Lie $\infty $-algebroid structure, whose brackets we now recall. 
 
\begin{prop}\label{prop-koszul}
A universal Lie $\infty $-algebroid of $\mathcal{F}_\varphi\subset\mathfrak{X}(V)$ is given on the free resolution $\left(\E_{-\bullet} = \X^{\bullet+1}(V),\text{\emph{d}}=\iota_{\dd \varphi},\rho= -\iota_{\dd \varphi} \right)$ 
by defining the following $n$-ary brackets:
\begin{equation}
\label{eq:nary}
\left\lbrace\partial_{I_{1}},\cdots, \partial_{I_{n}}\right\rbrace_{n}:=\sum_{i_{1}\in I_{1},\ldots,i_{n}\in I_{n}} \epsilon(i_{1},\ldots,i_{n})\varphi_{i_{1}\cdots i_{n}}\partial_{I_{1}^{i_{1}}\bullet\cdots\bullet I_{n}^{i_{n}}};
\end{equation}and the anchor map given for all $ i,j \in \{1, \dots,n\}$ by\begin{equation}
\rho\left(\frac{\partial}{\partial x_i}\wedge\frac{\partial}{\partial x_j}\right) :=\frac{\partial\varphi}{\partial x_j}\frac{\partial}{\partial x_i}-\frac{\partial\varphi}{\partial x_i}\frac{\partial}{\partial x_j}.
\end{equation}Above, for every multi-index $J=\left\lbrace j_1,\ldots ,j_n\right\rbrace\subseteq\left\lbrace 1,\ldots,d\right\rbrace$ of length $n$, $\partial_J$ stands for the $n$-vector field $\frac{\partial}{\partial x_{j_1}}\wedge\cdots\wedge\frac{\partial}{\partial x_{j_n}}$ and $\varphi_{j_{1}\cdots j_{n}}:=\frac{\partial^{n}\varphi}{\partial x_{j_1}\cdots\partial x_{j_n}}$ . Also, $I_{1}\bullet\cdots\bullet I_{n}$ is a multi-index obtained by concatenation of $n$ multi-indices $I_{1},\ldots,I_{n}$. For every $i_1\in I_1,\ldots,i_n\in I_n$,\;$\epsilon(i_1,\ldots,i_n)$ is the signature of the permutation which brings $i_1,\ldots,i_n$ to the first $n$ slots of $I_{1}\bullet\cdots\bullet I_{n}$. Last, for $i_s\in I_s$, we define $I_{s}^{i_s}:=I_s\backslash i_s$.\end{prop}
 
To understand this structure, let us first define a sequence of degree $ +1$ graded symmetric polyderivations on $\X^\bullet (V) $ (by convention, $i$-vector fields are of degree $-i+1 $) by:
\begin{equation}
\left\lbrace \partial_{i_1},\ldots,\partial_{i_k}\right\rbrace _k':=\frac{\partial^{k}\varphi}{\partial x_{i_1}\cdots\partial x_{i_k}}.
\end{equation}
We extend them to a graded poly-derivation of $\X^\bullet (V)$.

\begin{lemma}
\label{lem:Poisson}
The poly-derivations $(\left\lbrace \cdots \right\rbrace_k' )_{k \geq 1} $ are $\mathcal O $-multilinear and equip $\X^\bullet (V) $ with a (graded symmetric) Poisson $\infty$-algebra structure. 
Also, $\{\cdot \}_1'= \iota_{\dd \varphi} $.
\end{lemma}   
\begin{proof} 
For degree reason, $\left\lbrace F,  X_1, \dots, X_{k-1}  \right\rbrace_k'= 0$ for all $X_1, \dots, X_{k-1}\in \X^\bullet (V)$ and all $ F \in \X^0(V) = \mathcal O$. This implies the required $\mathcal O $-multilinearity. It is clear that the higher Jacobi identities hold since brackets of generators $\{\delta_{i_1}, \cdots, \delta_{i_n}\}' $ are elements in $ \mathcal O $, and all brackets are zero when applied an element in $ \mathcal O$.
\end{proof}

\begin{proof} [Proof (of Proposition \ref{prop-koszul})]
The brackets introduced in Proposition \ref{prop-koszul} are modifications of the Poisson ${\infty}$-algebra described in Lemma \ref{lem:Poisson}. By construction, $\left\lbrace \cdots\right\rbrace^{'}_{n}=\left\lbrace
\cdots\right\rbrace_{n}$ when all arguments are generators of the form $ \partial_{I}$ for some $I \subset \{1 \dots, n\} $ of cardinal $\geq 2 $. By $\mathcal O $-multilinearity, this implies  $\left\lbrace \cdots\right\rbrace^{'}_{n}=\left\lbrace
\cdots\right\rbrace_{n}$ when $n \geq 3$, or when $n=2$ and no argument is a bivector-field, or when  $n=1$ and the argument is not a bivector field. As a consequence, all higher Jacobi identities hold when applied to $n$-vector fields with $n \geq 3$.

Let us see what happens when one of the arguments is a bivector field, i.e. in the case where we deal with at least an element of degree $-1$. Let us assume that there is one such element, i.e. $Q_1=\partial_i\wedge\partial_j,\,Q_2=\partial_{I_{2}},\,\ldots,Q_n=\partial_{I_{n}}$ with $\lvert I_{j}\rvert\geq 3,\; j=2,\ldots,n$. Then, in view of the higher Jacobi identity for the Poisson $\infty$-brackets $ (\{\cdots\}_k')_{k\geq 1}$ gives:
\begin{align}
0=\sum_{2 \leq k  \leq n-2}&\label{a}\sum_{\sigma\in S_{k,n-k}}\epsilon(\sigma)
\left\lbrace \left\lbrace  Q_{\sigma(1)},\ldots,Q_{\sigma(k)}\right\rbrace_{k}',Q_{\sigma(k+1)},\ldots,Q_{\sigma(n)}\right\rbrace_{n-k+1}'\\&\label{b}+\sum_{\sigma\in S_{n-1,1}, \sigma(n) \neq 1}\epsilon(\sigma)\left\lbrace \left\lbrace  Q_{\sigma(1)},\ldots,Q_{\sigma(n-1)}\right\rbrace_{n-1}',Q_{\sigma(n)}\right\rbrace _{2}' \\&
\label{bb}+\sum_{\sigma\in S_{1,n-1}, \sigma(1) \neq 1 }\epsilon(\sigma)  \left\lbrace \left\lbrace  Q_{\sigma(1)}\right\rbrace_{1}',Q_{\sigma(2)}\ldots,Q_{\sigma(n)}\right\rbrace _{n}'
\\&\label{c}+(-1)^{\sum_{k=2}^n\lvert\partial_{I_k}\rvert}\left\lbrace \left\lbrace  Q_{2},\ldots,Q_{n}\right\rbrace_{n-1}',Q_{1}\right\rbrace _{2}'
\\&\label{d}+  \left\lbrace \left\lbrace  Q_{1}\right\rbrace_{1}',Q_{\sigma(2)}\ldots,Q_{\sigma(n)}\right\rbrace _{n}'.
\end{align}
In lines \eqref{a}-\eqref{b}-\eqref{bb} above, we have $\{\cdots \}'=\{\cdots \} $ for all the terms involved. This is not the case for \eqref{c}-\eqref{d}.
Indeed: \begin{align*}
\left\lbrace\left\lbrace \partial_{I_{2}},\cdots,\partial_{I_{n}}\right\rbrace^{'}_{n-1} ,\partial_i\wedge\partial_j\right\rbrace^{'}_2
&=
\left\lbrace\left\lbrace \partial_{I_{2}},\cdots,\partial_{I_{n}}\right\rbrace_{n-1} ,\partial_i\wedge\partial_j\right\rbrace_2
\\&-  \sum_{i_{2}\in I_{2},\ldots,i_{n}\in I_{n}}\epsilon(i_{2},\ldots,i_{n})\rho(\partial_i\wedge\partial_j)[\varphi_{i_{2}\cdots i_{n}} ] \, \partial_{I_{2}^{i_{2}}\bullet\cdots\bullet I_{n}^{i_{n}}}
\end{align*}
and
\begin{align*}
\left\lbrace\left\lbrace \partial_i\wedge\partial_j\right\rbrace _{1}',\partial_{I_{2}}\ldots,\partial_{I_{n}} \right\rbrace _{n}'&=(-1)^{\sum_{k=2}^n\lvert\partial_{I_k}\rvert+1}\left( \varphi_i\left\lbrace\partial_j, \partial_{I_{2}}\ldots,\partial_{I_{n}}\right\rbrace_{n}' -\varphi_j\left\lbrace\partial_i, \partial_{I_{2}}\ldots,\partial_{I_{n}}\right\rbrace_{n}'\right) \\&=-(-1)^{\sum_{k=2}^n\lvert\partial_{I_k}\rvert}\sum_{i_{2}\in I_{2},\ldots,i_{n}\in I_{n}}\epsilon(i_{2},\ldots,i_{n})\iota_{\dd\varphi}(\partial_i\wedge\partial_j)[\varphi_{i_{2}\cdots i_{n}} ]\,\partial_{I_{2}^{i_{2}}\bullet\cdots\bullet I_{n}^{i_{n}}} \\ 
& = (-1)^{\sum_{k=2}^n\lvert\partial_{I_k}\rvert}\sum_{i_{2}\in I_{2},\ldots,i_{n}\in I_{n}}\epsilon(i_{2},\ldots,i_{n})
\rho(\partial_i\wedge\partial_j)[\varphi_{i_{2}\cdots i_{n}} ] \,\partial_{I_{2}^{i_{2}}\bullet\cdots\bullet I_{n}^{i_{n}}}
\end{align*} since $\rho=-\iota_{\dd\varphi}$.
Hence the quantities in lines \eqref{d} and  \eqref{c} add up, when we re-write them in term of the new brackets $\left\lbrace \cdots\right\rbrace^{} _{k}$,  to yield precisely  the higher Jacobi identity for this new bracket. It is then not difficult to see this is still the case if there is more than one bivector field, by using many times the same computations.
\end{proof}

\subsubsection{Restriction to $\varphi=0 $ of vector fields annihilating $\varphi $}

We keep the convention and notations of the previous section.
Let us consider  the restriction $\mathfrak i_W^* \mathcal{F}_\varphi$ of the Lie-Rinehart algebra $\mathcal{F}_\varphi$ to the zero-locus $W$ of a Koszul polynomial $\varphi $. 
Since all vector fields in $\mathcal F_\varphi$ are tangent to $W$, this restriction is now a Lie-Rinehart algebra over $\mathcal{O}_W = \tfrac{\mathcal O}{\mathcal O \varphi}$
(see Example \ref{LR-Mor}).

\begin{proposition}
\label{prop:univphiIsZero}
Let $\varphi$ be a Koszul Polynomial. The restriction of the universal Lie $\infty$-algebroid of Proposition \ref{prop-koszul}
to the zero-locus $W$ of $\varphi $ is
a universal Lie $\infty$-algebroid of the Lie-Rinehart algebra $\mathfrak i_W^* \mathcal{F}_\varphi$. 
\end{proposition}

Since the image of its anchor map are vector fields tangent to $W $,
it is clear that the universal Lie $\infty$-algebroid of Proposition \ref{prop-koszul} restricts to $W$. To prove Proposition \ref{prop:univphiIsZero}, it suffices to check that the restriction $\mathfrak i_W^* \X (V) $ to $W $ of the Koszul complex is still exact, except in degree $0$. This is a simple lemma, whose proof is left to the reader.

\begin{lemma}The restriction to the zero locus $W$ of $ \varphi $ of the Koszul complex \eqref{eq:KoszulComplex}, namely the complex
$$\ldots\xrightarrow{\iota_{\dif\varphi}}\mathfrak i_W^* \X^{3}(V)\xrightarrow{\iota_{\dif\varphi}}\mathfrak i_W^* \X^{2}(V)\xrightarrow{\iota_{\dif\varphi}}\mathfrak i_W^* \mathcal F  $$
is a free resolution of $\mathfrak i_W^* \mathcal F$ in the category of $\mathcal O_W $-modules.
\end{lemma}

\subsubsection{Vector fields vanishing on subsets of a vector space} 
\label{ideal}
Let $ \mathcal O$ be the algebra of smooth or holomorphic or polynomial or formal functions on $\mathbb K^d $, and $\mathcal I \subset \mathcal O$ be an ideal.
Then $ \mathcal I {\mathrm{Der}} (\mathcal O) $, i.e. vector fields of the form: 
 $  \sum_{i=1}^d f_i  \frac{\partial}{\partial x_i} $, with $f_1, \dots, f_d \in \mathcal I$, is a Lie-Rinehart algebra. 

\begin{remark}
 Geometrically, when $\mathcal I $ corresponds to functions vanishing on a sub-variety $N \subset \mathbb K^n$, 
  $ \mathcal I {\mathrm{Der}} (\mathcal O) $ must be interpreted as vector fields vanishing along $N$.
\end{remark}

Let us describe a Lie $\infty $-algebroid that terminates at $ \mathcal I {\mathrm{Der}} (\mathcal O)$, then discuss when it is universal. 
Let $(\varphi_i)_{i \in I}$  
be generators of $\mathcal I $. Consider the free graded algebra $ \mathcal K = \mathcal O [(\mu_i)_{i \in I} ] $
generated by variables $(\mu_i )_{i \in I} $ of degree $-1 $. The degree $-1 $ derivation 
$\partial:=\sum_{i\in I} \varphi_i\frac{\partial}{\partial\mu_i}$ squares to zero. 
The $\mathcal O$-module $\mathcal K_{-j}$ of elements degree $j$ in $\mathcal K_\bullet $ is made of all sums
$\sum_{i_1 , \dots , i_j \in I }   f_{i_1 \dots i_j}   \mu_{i_1}\cdots\mu_{i_j} $
with $ f_{i_1 \dots i_j} \in \mathcal O$. Consider the complex of free $\mathcal O $-modules
\begin{equation}\label{eq:resolIX} \cdots\stackrel{\partial\otimes_\mathcal O \text{id}}{\longrightarrow} \mathcal K_{-2} \otimes_\mathcal O {\mathrm{Der}}(\mathcal O)   \stackrel{\partial\otimes_\mathcal O \text{id}}{\longrightarrow} \mathcal K_{-1}\otimes_\mathcal O{\mathrm{Der}}(\mathcal O)\end{equation}

\begin{proposition} \label{prop:existsLieInfty}
The complex \eqref{eq:resolIX} comes equipped with a Lie $\infty $-algebroid structure that terminates in $\mathcal I {\mathrm{Der}} (\mathcal O)$ through the anchor map  given by $\mu_i\frac{\partial}{\partial x_j}   \mapsto \varphi_i \frac{\partial}{\partial x_j} $ for all $i \in I$, and $ j \in 1, \dots,d$. 
\end{proposition} 
\begin{proof}
 First, one defines a $\mathcal O$-linear Poisson-$\infty $-algebra structure on the free algebra generated by $(\mu_i)_{i\in I} $ (in degree $-1$) and $  \left( \frac{\partial}{\partial x_j} \right)_{j=1}^d $ (in degree $0$) and $1$ by:
     \begin{equation}
         \label{PoissonInfty}
\left\lbrace\mu_{i},{  \frac{\partial}{\partial x_{j_1}}},\ldots,\frac{\partial}{\partial x_{j_r}} \right\rbrace_{r+1}':=   \frac{\partial^r \varphi_i}{\partial x_{i_1} \dots \partial x_{i_r} }   
 \end{equation}
all other brackets of generators being equal to $0$. Since the brackets of generators take values in $\mathcal O$, and since an $n$-ary bracket where an element of $\mathcal O $ appears is zero, this is easily seen to be a Poisson $\infty $-structure. The general formula is
\begin{equation}
\label{eq:onGenMinusALl}
\left\lbrace\mu_{I_1}\otimes_\mathcal O\partial_{x_{a_1}},\ldots,\mu_{I_n}\otimes_\mathcal O\partial_{x_{a_n}} \right\rbrace_n':= \sum_{\hbox{\scalebox{0.5}{$
\begin{array}{c} j=1 , \dots, n \\ 
i_j \in I_j \end{array}$}}
} \epsilon
\frac{\partial^{n-1}\varphi_{i_j}}{\partial x_{a_1}\cdots\widehat{\partial x}_{a_j}\cdots\partial x_{a_n}}  \mu_{I_1} \cdots \mu_{I_j}^{i_j} \cdots \mu_{I_n\otimes_{\mathcal O} \partial_{x_{a_j}}},
\end{equation}    
 where $\mu_J = \mu_{j_1} \dots \mu_{j_s}$ for every list $J= \{j_1, \dots, j_s\} $, 
where $ \epsilon$ is the Koszul sign, and where for a list $J$ containing $j$, $J^j $ stands for the list $J$ from which the element $j$ is crossed out, as in Equation \eqref{eq:nary}.

The $\mathcal O $-module generated by $\mu_{i_1} \cdots \mu_{i_k } \otimes_{\mathcal O}  \partial_{ x_{a}}  $ , i.e.
the complex \eqref{eq:resolIX} is easily seen to be stable under the brackets $\{\cdots \}_k' $
for all $k \geq 1$, so that we can define on $\mathcal K \otimes_{\mathcal O}  {\mathrm{Der}} (\mathcal O)$ a sequence of brackets
$(\ell_k = \{\cdots\}_k)_{k \geq 1} $ by letting them coincide with the previous brackets on the generators, i.e. $\{\cdots\}_n$ is given by Equation \eqref{eq:onGenMinusALl} for all $n \geq 1$. The brackets are then extended by derivation, $\mathcal O$-linearity or Leibniz identity with respect to the given anchor map, depending on the degree.
\noindent 
In particular, $\{\cdots \}'_k= \{\cdots \}_k $ 
for $k\geq 2$. For $k=1$,  $\{\cdots \}'_1= \{\cdots \}_1 $ on  $\oplus_{i \geq 2} \mathcal K_{-i} \otimes_\mathcal O  {\mathrm{Der}} (\mathcal O) $.
For $k=2$, we still have $\{\cdots \}'_2= \{\cdots \}_2 $ on  $\oplus_{i,j \geq 2} \mathcal K_{-i} \odot \mathcal K_{-j} \otimes_\mathcal O  {\mathrm{Der}} (\mathcal O) $.\\

Let us verify that all required axioms are satisfied.  
For $n=2$, Equation
\eqref{eq:onGenMinusALl} specializes to:
 \begin{align*} 
 \ell_2 ( \mu_i \otimes_\mathcal O \partial_{x_a} , \mu_j \otimes_\mathcal O \partial_{x_b} ) = \frac{\partial \varphi_j}{\partial x_a} \,  \mu_i \otimes_\mathcal O \partial_{x_b} - \frac{\partial \varphi_i}{\partial x_b} \,  \mu_j \otimes_\mathcal O \partial_{x_a} 
 \end{align*}
 which proves that the anchor map is a morphism when compared with the relation:
  $$  \left[ \varphi_i \partial_{x_a}  \, , \, \varphi_j \partial_{x_b}  \right] = 
  \frac{\partial \varphi_j}{\partial x_a} \,  \varphi_i  \, \partial_{x_b} - \frac{\partial \varphi_i}{\partial x_b} \,  \varphi_j  \, \partial_{x_a}.
  $$
  The higher Jacobi identities are checked on generators as follows:
 \begin{enumerate}
     \item When there are no degree $-1$ generators, it follows from the higher Jacobi identities of the Poisson $\infty$-structure \eqref{PoissonInfty} and the $ \mathcal O$-multilinearity of all Lie $\infty $-algebroid brackets involved.
     \item When generators of degree $-1 $ are involved, the higher Jacobi identities are obtained by doing the same procedure as in the proof of Proposition \ref{lem:Poisson}, that is, we first consider the higher Jacobi identities for the Poisson $\infty $-structure \eqref{PoissonInfty}, and we put aside the terms where $\{\cdot\}_1 $ is applied to these degree $-1 $ generators. We then check that the latter terms are exactly the terms coming from an anchor map when the $2$-ary bracket is applied to generators of degree $-1$ and the $(n-1)$-ary brackets of the remaining generators. \end{enumerate}\end{proof}

\begin{proposition}\label{complete} 
When $\mathcal O = \mathbb C[x_1, \dots, x_n]$, and when $N \subset   \mathbb C^n$ is an affine variety defined by a regular sequence $\varphi_1, \dots, \varphi_k $, then the Lie $ \infty$-algebroid described in Proposition \ref{prop:existsLieInfty} is the universal Lie $\infty$-algebroid of the singular foliation of vector fields vanishing along $N$.
\end{proposition}
\begin{proof}
For a regular sequence $\varphi_1, \dots, \varphi_k $, $ \mathcal K_\bullet$ equipped with the derivation $ \partial=\sum_{i=1}^k \varphi_i \frac{\partial}{\partial \mu_k} $ is a free $ \mathcal O$-resolution of the ideal $\mathcal I_N  $ of functions vanishing along $N$. Since $  {\mathrm{Der}}(\mathcal O) $ is a flat $\mathcal O $-module,  the sequence \begin{equation}\label{eq:resolIXX}
\xymatrix{\cdots\ar[rr]^-{\partial\otimes_\mathcal O \text{id}} & & \mathcal K_{-2} \otimes_\mathcal O {\mathrm{Der}}(\mathcal O) \ar^{\partial\otimes_\mathcal O \text{id}}[rr]& & \mathcal K_{-1}\otimes_\mathcal O {\mathrm{Der}}(\mathcal O) \ar^ {\partial\otimes _\mathcal O\text{id}}[rr] & & \mathcal I_N {\mathrm{Der}}(\mathcal O).
}
\end{equation} is a free $\mathcal O $-resolution of the singular foliation $\mathcal I_N {\mathrm{Der}}(\mathcal O)$.
The Lie $\infty $-algebroid structure of Proposition \ref{prop:existsLieInfty} is therefore universal.
\end{proof}

\begin{example}
As a special case of the Proposition \ref{complete},  let us consider a complete intersection defined by \emph{one} function, i.e. an affine variety $W$ whose ideal $\langle\varphi\rangle$ is generated by a regular polynomial $\varphi \in \mathbb{C}[X_1,\ldots,X_d]$. One has a free resolution of the space of vector fields vanishing on $W$ given as follows: 
$$ 
\xymatrix{0\ar[r] & \mathcal{O}\mu \otimes_\mathcal O \X(\mathbb C^d) \ar^{\varphi\frac{\partial}{\partial\mu}\otimes_\mathcal O \text{id}}[rr] & & I_W \mathfrak X_{\mathbb C^d}},
$$
where $\mu$ is a degree $ -1$ variable, so that $\mu^2=0$. The universal Lie $\infty$-algebroid structure over that resolution is given on the set of generators by :$$\{\mu\otimes_\mathcal O \partial_{x_a},\mu\otimes_\mathcal O \partial_{x_b}\}_2:=\frac{\partial\varphi}{\partial x_a}\mu\otimes_\mathcal O \partial_{x_b}-\frac{\partial\varphi}{\partial x_b}\mu\otimes_\mathcal O \partial_{x_a}$$ and $\{\cdots\}_k:=0$ for every $k\geq 3$. It is a Lie algebroid structure. Notice that this construction could be also be recovered using Section \ref{sec:codim1}.
\end{example}

\subsection{Universal Lie $\infty $-algebroids associated to an affine variety $W$}

This section is mainly programmatic. It explains how we can attach the purely algebraic objects \textquotedblleft universal Lie $\infty $-algebroids\textquotedblright\,to the geometric objects \textquotedblleft affine varieties\textquotedblright. Their relations will be studied in a future article. For any affine variety $W$ over $\mathbb C $, with ring of function $\mathcal O_W $, derivations of $W$ (i.e. vector fields on $W$) are a Lie-Rinehart algebra over $\mathcal O_W $ denoted $\mathfrak X_W $. Its universal Lie-$\infty$-algebroid can therefore be constructed. 
\begin{definition}
We call \emph{universal Lie $\infty$-algebroid of an affine variety $W$} any Lie-$\infty$-algebroid associated to its  Lie-Rinehart algebra $\mathfrak X_W$ of vector fields on $W$.
\end{definition}

Let us state a few results about the universal Lie $\infty$-algebroid of an affine variety $W$. For every point $x_0 \in W$, let $\mathcal O_{W,x_0}  $ be the algebra of germs of functions at $x_0$. 
The germ at $x_0$ of the Lie-Rinehart algebra of vector fields on $W$ is easily checked to coincide with the Lie-Rinehart algebra of derivations of $\mathcal O_{W,x_0}  $.

Here is an immediate consequence of Proposition \ref{prop:germ}.

\begin{prop}
Let $W$ be an affine variety.
For every $x_0 \in W$, the germ at $x_0$ of the universal Lie $\infty$-algebroid of $W$ is the universal Lie $\infty$-algebroid of
${\mathrm{Der}}(\mathcal O_{W,x_0})$.
\end{prop}

Let us choose $x_0\in W$.
As stated in Section \ref{ssec:res}, the Lie $\infty$-algebroid structure of $W$ over $x_0$ restricts at $x_0$ (i.e. goes to the quotient with respect to the ideal $ \mathfrak m_{x_0} $) if and only if $\rho_\E(\E_{-1})[\mathfrak m_{x_0}]\subseteq \mathfrak m_{x_0}$, (i.e $\mathfrak m_{x_0}$ is a Lie-Rinehart ideal). This is the case, in particular, if $x_0$ is an isolated singular point. In that case, we obtain a Lie $\infty $-algebroid over $\mathcal O_W/\mathfrak m_{x_0} $. Since this quotient is the base field $\mathbb K $, we obtain in fact a Lie $\infty $-algebra. By Section \ref{ssec:res} again,  it is a Lie $\infty $-algebra on ${\mathrm{Tor}}_{\mathcal O_W} (\mathfrak X_W, \mathbb C) $. We call it the \emph{Lie $\infty $-algebra of the isolated singular point $x_0$}. To describe this structure, let us start with the following Lemma.

	\begin{lemma}\label{lem:finite}
		Let  $x_0$ be a point of  an affine variety $W$.  
		The universal Lie $\infty$-algebroid of
${\mathrm{Der}}(\mathcal O_{W,x_0})$
		can be constructed on a resolution $((\E_{-i}^{x_0})_{i\geq 1}, \ell_1, \pi) $, with $\mathcal E_{-i}^{x_0} $  free $\mathcal O_{W,x_0}$-modules of finite rank for all $i \geq 1$, which is minimal in the sense that $\ell_1 (\E_{-i-1}) \subset \mathfrak m_{x_0}\E_{-i} $ for all $i \geq 1$. 
	\end{lemma}
	\begin{proof}
		Since Noetherian property is stable by passage to localization, the ring $\mathcal O_{W,x_0}$ is a Noetherian local ring. Proposition 8.2 in \cite{CohenMacauley}  assures that $\mathcal O_{W,x_0}\otimes_{\mathcal O_W}\text{Der}(\mathcal O_{W})$ admits a free minimal resolution  by free finitely generated $\mathcal O_{W,x_0}$-modules. 
		Since $\mathcal O_{W,x_0}$ is a local ring with maximal idea  $ \mathfrak m_{x_0}$, we can assume that this resolution is minimal.
		In view of Theorem \ref{thm:existence}, there exists a Lie $\infty$-algebroid structure over this resolution, and the latter is universal for ${\mathrm{Der}}(\mathcal O_{W,x_0})$.
	\end{proof}
	
By Theorem \ref{thm:existence}, a resolution of  ${\mathrm{Der}}(\mathcal O_{W,x_0})$ as in Lemma \ref{lem:finite} comes equipped with a universal Lie $\infty $-algebroid structure for  ${\mathrm{Der}}(\mathcal O_{W,x_0})$. The quotient with respect to $\mathfrak m_{x_0} $ is a
	Lie $\infty $-algebra of the isolated singular point $x_0$
	with trivial $1$-ary bracket.
Using Corollary \ref{cor:LRideal2} and its subsequent discussion, we can prove the next statement.
\begin{prop}
For any universal Lie $\infty$-algebroid structure on a resolution of
${\mathrm{Der}}(\mathcal O_{W,x_0})$ as in Lemma \ref{lem:finite},
the quotient with respect to the ideal $ \mathfrak m_{x_0}$ is a representative of the Lie $\infty $-algebra of the isolated singular point $x_0$, with trivial $1$-ary bracket, on a graded vector space canonically isomorphic to ${\mathrm{Tor}}_{\mathcal O_W} (\mathfrak X_W, \mathbb C) $
($\mathbb C$ being a $\mathcal O_{W}$-module through evaluation at $x_0$).

In particular, its $2$-ary bracket is a graded Lie bracket on ${\mathrm{Tor}}_{\mathcal O_W} (\mathfrak X_W, \mathbb C) $ which does not depend on any choice made in the construction, and its $3$-ary bracket is a Chevalley-Eilenberg cocycle whose class is also canonical. 
\end{prop}

This discussion leads to the natural question:\\

\noindent
{\textbf{Question}.\emph{ How is the geometry of an affine variety related to its universal Lie $\infty$-algebroid?}}\\

\noindent
This will be the topic of a forthcoming article.

\bibliographystyle{plain}
\bibliography{LRbiblio}

\end{document}